\documentclass[aip,graphicx]{revtex4}
\usepackage{amssymb}
\usepackage{amsmath, etoolbox}
\usepackage{commath,hyperref}
\usepackage{times,animate}
\usepackage{breqn}
\usepackage{float,here}
\usepackage{rotating}
\usepackage{array,ragged2e}
\usepackage{amsthm,bm}
\usepackage{setspace}
\usepackage{multirow,booktabs}
\usepackage{tikz,gensymb}
\usepackage{tikz-3dplot}
\usepackage{pgfplots}

\usepackage{graphicx, epsfig, graphics, graphpap}

\usepackage{breqn,epstopdf}
\newtheorem{theorem}{Theorem}[section]

\newtheorem{lemma}[theorem]{Lemma}
\newtheorem{remark}{Remark}
\makeatletter
\def\@email#1#2{%
 \endgroup
 \patchcmd{\titleblock@produce}
  {\frontmatter@RRAPformat}
  {\frontmatter@RRAPformat{\produce@RRAP{*#1\href{mailto:#2}{#2}}}\frontmatter@RRAPformat}
  {}{}
}%
\begin{document}


\title[Sample title]{A Novel Higher Order Compact-Immersed Interface Approach For Elliptic Problems}
\author{Raghav Singhal}
\email[]{raghav2016@iitg.ac.in}
\affiliation{Department of Mathematics; Indian Institute of Technology Guwahati; Guwahati-781039; Assam; India}
\author{Jiten C. Kalita}
\email[]{jiten@iitg.ac.in}
\affiliation{Department of Mathematics; Indian Institute of Technology Guwahati; Guwahati-781039; Assam; India}

\date{\today}

\begin{abstract}
 We present a new higher-order accurate finite difference explicit jump Immersed Interface
Method (HEJIIM) for solving two-dimensional elliptic problems with singular source and discontinuous coefficients in the irregular region on a compact Cartesian mesh. We propose a new strategy for discretizing the solution at irregular points on a nine point compact stencil such that the higher-order compactness is maintained throughout the whole computational domain. The scheme is employed to solve four problems embedded with circular and star shaped interfaces in a rectangular region having analytical solutions and varied discontinuities across the interface in source and the coefficient terms. We also simulate a plethora of fluid flow problems  past bluff bodies  in complex flow situations, which are governed by the Navier-Stokes equations; they include problems involving multiple bodies immersed in the flow as well.  In the process, we show the superiority of the proposed strategy over the EJIIM and other existing IIM  methods by establishing the rate of convergence and grid independence of the computed solutions. In all the cases our computed results extremely close to the available numerical and experimental results. 
\end{abstract}
\maketitle
\section{Introduction}
Interface problems have been of good interest in many applications, such as two-fluid interactions, multiphase flows with fixed or moving interface at which states (Solid/Liquid/Gas) are different across the interface, but allows the same material. For Example, water or air, water or ice, bubble formation, free-surface flow, relaxation of an elastic membrane, Rayleigh-Taylor instability of binary flows. The main difficulty in simulating multiphase flow lie in handling the interface. First, generating an excellent body-fitted grid is non-trivial and more time-consuming. Computationally it is quite challenging to regenerate a good body-fitted mesh in moving boundaries since it can undergo modification, merge and separation throughout the course of simulation. In contrast, the construction of a Cartesian mesh is trivial where the interface can get cut by the grid lines, with no additional computational cost. Numerically the interface can be computed by, amongst others, the following techniques: the boundary element method, the front tracking method (Lagrangian method),  the volume-of-fluid method and the level set method (Eulerian methods).
 Moreover, for those problems governed by Navier-Stokes (N-S) equations, such as flow past bluff bodies,  discontinuities occur in the coefficients and the source term across the interface may become singular, which leads to the discontinuous or non-smooth solution. Therefore, such problems pose great challenge to the Mathematicians and Engineers alike.

The origin of IIM lies in the early work of Peskin \cite{peskin1972flow}, who developed the Immeresed Boundary Method (IBM) in 1972 primarily to handle interface discontinuity to solve Navier-Stokes equations for simulating cardiac mechanics and associated blood flow problems. The main feature of IBM in modelling an interface was to add a source term in the form of the Dirac delta function to the N-S equations. Standard finite difference discretization is used on Cartesian grid that approximates the singular delta function in the interface's nearby region. The major disadvantage of Peskin's approach has been its first order accuracy despite using higher-order approximation to the delta function and being restricted to problems having continuous solutions only. Subsequently, several researchers endeavoured to improve upon IBM by using the front tracking  and level set methods in combination with Peskin's approach for tackling interface discontinuities. In 1984 \cite{mayo1984fast}, Mayo developed a Cartesian mesh method to solve biharmonic and Laplace problems on the irregular region where second-order accuracy was reported in maximum norm. In this formulation, Fredholm integral equation of second kind was used to extend the piecewise solution of the remaining part of the domain.

In 1994, Leveque and Li \cite{leveque1994immersed} achieved vital success in this area, and developed the Immersed Interface Method (IIM) for solving the elliptic equations with discontinuous coefficients and singular sources. This method, which is a successor to the Immersed Boundary Method of Peskin \cite{peskin1972flow}, improved upon both in dealing with the order of accuracy as well with discontinuous coefficients and singular source function simultaneously. At an irregular point, they accounted for the jump conditions in the solution and in the normal derivative by using Taylor expansions of the discretization on both sides of the interface with first-order accuracy and standard second-order central finite difference were used on regular points. Note that, for 2D or 3D interface problems, the jump conditions are available in the normal direction to the interface, which necessitates the use of a local coordinate system to approximate it. Over the years, the IIM has emerged as a very powerful and effective tool in numerically solving problems involving interfaces. They have been extended to  the polar coordinate system \cite{lipolar}, and successfully implemented to moving \cite{dumett2003immersed} and  3D interface problems \cite{li3d} as well.  Apart from finite difference, they can also be accommodated into finite volume \cite{calhoun} and the finite element \cite{lifem} approaches.

IIM generally leads to a non-symmetric coefficients matrix; however, the problems are strictly elliptic and self-adjoint. As such, traditional iterative solvers like Gauss-Seidel may diverge or  converge very slowly. In 1999, Huang and Li \cite{huang1999convergence} exhibited that the method in \cite{leveque1994immersed} is stable for one-dimensional problems and  in two-dimensions, only for problems having piecewise-constant coefficients. In order to improve convergence, Li et al. \cite{limaximum} constructed a new IIM approach where they implemented a maximum principle preserving immersed interface method (MIIM)  to achieve a diagonally dominant linear system which allows the use of  specially designed multigrid techniques to speed up convergence. In 1997, the same group developed  a second-order accurate method for elliptic interface problems by roping in a Fast IIM algorithm \cite{lifiim}. They devised a mechanism which uses auxiliary unknowns revealing the normal derivative at the interface for problems having piecewise constant coefficients. This generates a correction term and experts feel that the success of the fast IIM lies in modelling the jump conditions for the dependent variable and its normal derivative using the standard FD scheme along with this correction term. It enabled the application of various standard fast Poisson solvers. Around the same time, in 1998, Fedkiw et al. \cite{fedkiw1999non} introduced a non-oscillatory sharp interface approach GFM (Ghost Fluid Method) to capture discontinuities in the hyperbolic equations. The methodology involves developing the function across the interface using fictitious points. Later on, they generalized the GFM \cite{fedkiw2002ghost} to solve viscous N-S equations by imposing the jump conditions implicitly.  In 2000, Liu et al. \cite{liu2000boundary} extended this GFM for solving elliptic equations with variable coefficients having discontinuous across the interface where order of accuracy is reduced to one while handling mutiphase flows. 

Berthelsen \cite{berthelsen2004decomposed} constructed a second-order accurate decomposed immersed interface method (DIIM) on a Cartesian grid which includes more jump conditions to improve accuracy.  This method interpolates the jump conditions component-wise iteratively on a nine-point stencil and adds them to the right-hand side of the difference scheme near an interfacial node.  It uses the level-set function to capture the interface and maintains the symmetry and diagonal dominance of the coefficient matrix. In 2005,  Zhou et al.  \cite{zhou2006high} introduced a new higher-order Matched Interface Boundary (MIB) method for 2D and 3D elliptic problems which is based on the use of fictitious points. In order to accomplish higher-order accuracy, the MIB method compensates  the bypassing of the implementation of higher-order jump conditions by repeatedly enforcing the lowest order jump conditions. While MIB demonstrated accuracy up to sixth order in  dealing with elliptic curves on irregular interfaces, for straight, regular interfaces, it could go as high as 16th order. In 2007, Zhong \cite{zhong2007new} proposed an explicit higher-order finite difference method by approximating the derivative in jump condition using Lagrange polynomial with a larger stencil at an irregular node and achieved the accuracy up to $O(h^4)$. One of the drawbacks of the method is that it did not retain the original finite difference expression in the absence of the interface's jump. However, it was equivalent to a local higher-order spline approximation at the interface. Although Zhou's and Zhong's methods are higher-order, Zhou's method uses fictitious points and not explicitly derived, while Zhong's approach is explicit and doesn't use any auxiliary points.

Employing a higher-order compact approach \cite{kalita2004transformation} at the regular points, Mittal et al. \cite{mittal2016class} developed an at least a second-order accurate method to solve elliptic and parabolic equations in 1D and 2D on circular interfaces with an HOC approach in both Cartesian and polar grids in 2016. The scheme was on non-uniform grids and employed clustering near the interfaces. In 2018 \cite{mittal2018solving}, they introduced a new second-order interfacial points-based approach for solving 2D and 3D elliptic equations. They modified Zhong's \cite{zhong2007new} idea in considering interfacial points to be one of the grid points in approximating the derivative in jump conditions by a Lagrange polynomial. However, at the irregular points, their stencil failed to maintain its compactness.

In the current work, we propose a new higher-order compact, explicit jump finite difference Immersed Interface Method (HEJIIM) for solving two-dimensional elliptic problems with singular source and discontinuous coefficients in the irregular region on Cartesian mesh. Contrary to the schemes in \cite{mittal2016class,mittal2018solving}, the proposed scheme maintains its compactness on a nine point stencil at both the regular and irregular points. In order to treat the jump across the interface, we modified the explicit jump immersed interface strategy of Wiegmann and Bube \cite{wiegmann2000explicit} in such a way that fourth order accuracy is retained throughout the whole computational domain. Using the proposed scheme, firstly we solve four problems embedded with circular and star shaped interfaces in a rectangular region having analytical solutions. Then we compute the steady-state flow past bluff bodies  in different flow set ups, which are governed by the Navier-Stokes equations. Our simulations include problems involving multiple bodies immersed in the flow as well, thus extending the scope of application of the proposed approach developed in this study to multiply connected regions as well. For the problems having analytical solutions, our results are excellent match with the analytical ones and for the cylinder, the computed flow is extremely close to the experimental results.

The paper is organized in the following way. In section 2, we detail the development of the proposed scheme, section 3 deals with a brief description of the solution of the algebraic systems, section 4 discusses the numerical examples and finally in conclusion, we summarize our achievements.
 \section{Mathematical Formulation}
 The conservative form of an elliptic equation in two dimensions is given by
  \begin{equation}
  (\beta u_x)_x+(\beta u_y)_y + \kappa(x,y) u=f(x,y)+\sigma \delta\{(x-x^{*})(y-y^{*})\} \\
(x,y)\in \Omega\;,\;\;(x^{*},y^{*})\in \Gamma \label{s1}
  \end{equation}
where $\Gamma$ is an interface embedded in a rectangular domain $\Omega= [x_0,x_f] \times [y_0,y_f]$ with defined boundary conditions. The coefficients $\beta(x,y) $, $\kappa(x,y)$ and $f(x,y)$ may be non smooth functions or may have discontinuities across the interface $\Gamma$ leading to discontinuities in the solution and its derivatives at the interface. To solve an interface problem we generally require two physical jump conditions in the solution $u$ and in the normal direction to the interface which is defined by:
\begin{eqnarray}
[u]= &\lim_{(x,y) \to \Gamma ^+}u(x,y) - \lim_{(x,y) \to \Gamma ^-}u(x,y) = &u^+-u^-=\hat{S},\\	\label{s2}
[\beta u_n]=& \lim_{(x,y) \to \Gamma ^+}\beta u_n(x,y) - \lim_{(x,y) \to \Gamma ^-} \beta u_n(x,y) =& \beta^+\frac{\partial u^+}{\partial n}-\beta^-\frac{\partial u^-}{\partial n} =\sigma.	\label{s3}
\end{eqnarray}
where $\hat{S}$ and $\sigma$ are strength and flux of the variable $u$ respectively. Here $(x,y) \to \Gamma ^{+} $ represents that interface is approaching from $\Omega^{+}$ and vice versa for $(x,y) \to \Gamma ^{-}.$ 

In order to capture the interface  $\Gamma$, Osher and Sethian \cite{sethian1996level} conceived a function, widely known as the {\bf Level-Set function}. The interface divides $\Omega$ into two sub-domains $\Omega^-$ and $\Omega^+$ and therefore $\Omega = \Omega^- \cup \Gamma \cup \Omega^+$. We use zero level set for a two dimensional function $\phi (x,y)$ to represent the interface, i.e
\begin{equation}
\left\{\begin{array}{cc}
&\phi (x,y) < 0, \;\;\;\;\;\;\;\;\;\; \text{if}\;\; (x,y) \in \Omega^- \\
&\phi (x,y) = 0, \;\;\;\;\;\;\;\text{if}\;\; (x,y) \in \Gamma \\
&\phi (x,y) > 0 \;\;\;\;\;\;\;\;\;\; \text{if}\;\; (x,y) \in \Omega^+.
\end{array}\right.
\end{equation}
The schematic of the level-set function representing the interface and the sub-domains in the computational plane can be seen in figure \ref{dl}(a). Here ($\eta, \xi$) represents the local coordinate system at an interfacial point $(x^\star,y^\star)$ with each of them representing the tangent and normal direction respectively at the point along the interface.
\begin{figure}[!h]
\minipage{0.33\textwidth}
  \includegraphics[width=\linewidth]{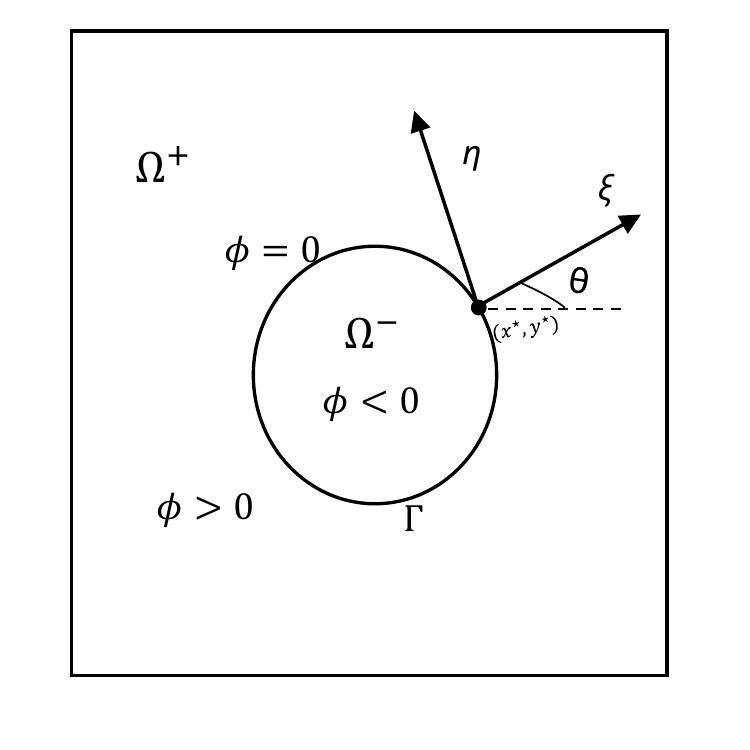}
  \begin{center}
   (a)
  \end{center}
\endminipage\hfill
\minipage{0.33\textwidth}
  \includegraphics[width=\linewidth]{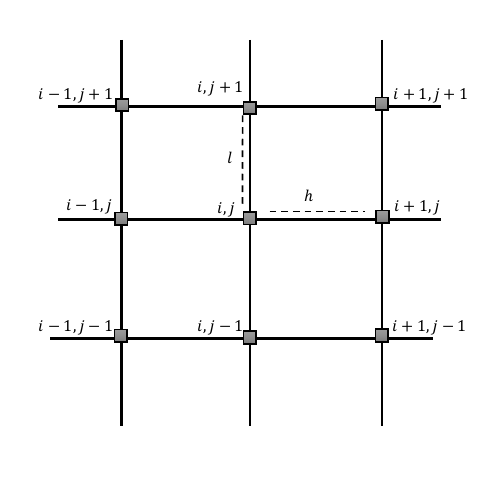}
  \begin{center}
   (b)
  \end{center}
\endminipage\hfill
\minipage{0.33\textwidth}%
  \includegraphics[width=\linewidth]{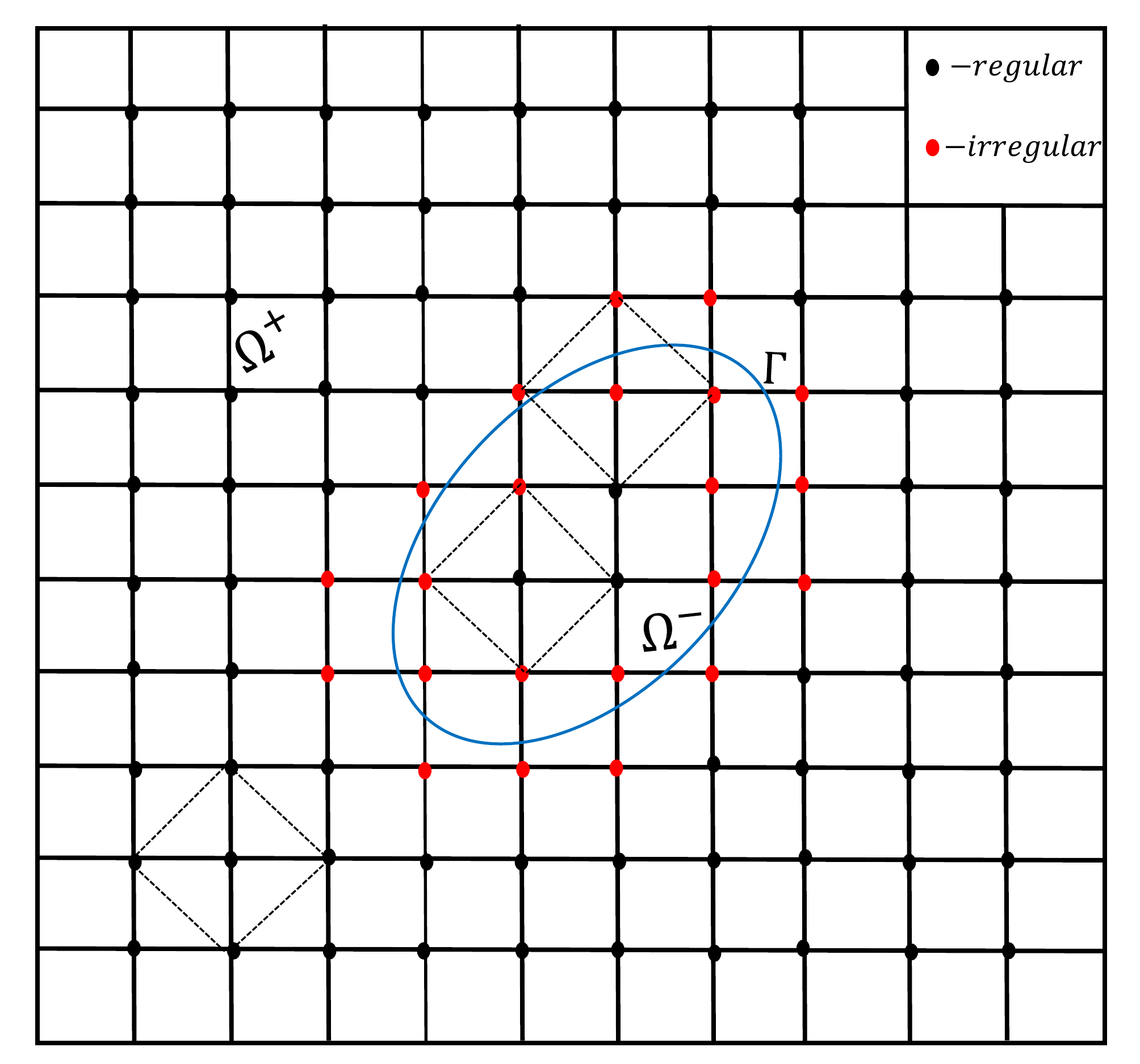}
  \begin{center}
   (c)
  \end{center}
\endminipage
\caption{{\sl (a) Schematic of the level-set function along with the local coordinates on an interfacial point, (b) the compact HOC nine point stencil and, (c) regular and irregular points.} }
\label{dl}
\end{figure}
For approximating the jump conditions eq(\ref{s3}) on a Cartesian mesh at the point $(x^{\star}, y^{\star})$, we have
\begin{eqnarray}
\xi=&(x-x^{\star})  cos(\theta )+ (y-y^{\star}) sin(\theta),\\
\eta=&-(x-x^{\star})  sin(\theta )+ (y-y^{\star}) cos(\theta).
\end{eqnarray}
where $\theta$ is the angle between $x$-axis and $\xi$-direction. The jump conditions for the derivatives up to third order can be calculated by the following formulas.
\[
\begin{pmatrix}
[u_{x}] \\ [u_{y}] 
\end{pmatrix}
=
\begin{pmatrix}
c & -s  \\
s & c 
\end{pmatrix}
\begin{pmatrix}
[u_{\xi}]\\ [ u_{\eta}] 
\end{pmatrix},
\]

\[
\begin{pmatrix}
\lbrack u_{xx}\rbrack \\\lbrack u_{xy} \rbrack \\\lbrack u_{yy}\rbrack
\end{pmatrix}
=
\begin{pmatrix}
c^{2} & -2cs & s^{2} \\
cs & c^{2}-s^{2} & -cs \\
s^{2}  & 2cs  & c^{2}  
\end{pmatrix}
\begin{pmatrix}
 [u_{\xi \xi}] \\ [u_{\xi \eta}]\\ [u_{\eta \eta}] 
\end{pmatrix},
\]

\[
\begin{pmatrix}
\lbrack u_{xxx}\rbrack \\ \lbrack u_{xxy} \rbrack \\ \lbrack u_{xyy}\rbrack \\ \lbrack u_{yyy}\rbrack
\end{pmatrix}
=
\begin{pmatrix}
c^{3} & -3c^{2}s & 3cs^{2} & -s^{3} \\
c^{2}s & c^{3}-2cs^{2} & s^{3}-2c^{2}s & cs^{2} \\
cs^{2}  & 2c^{2}s-s^{3}  & c^{3}-2cs^{2} &  -cs^{2}\\
s^{3} & 3s^{2}c & 3c^{2}s & c^{3}
\end{pmatrix}
\begin{pmatrix}
 [u_{\xi \xi \xi}] \\  [u_{\xi \xi \eta}] \\ [u_{\xi \eta \eta}]\\ [u_{\eta  \eta \eta}] 
\end{pmatrix}.
\]
where $c=cos(\theta)$ and $s=sin(\theta)$.
The non conservative form of equation (\ref{s1}) can be written as
\begin{equation}
\beta (u_{xx}+u_{yy}) + c(x,y) u_{x} + d(x,y) u_{y} + \kappa(x,y) u =f(x,y). \label{s4} 
\end{equation}
where $c(x,y)=\frac{\partial \beta(x,y)}{\partial x}$ and  $d(x,y)=\frac{\partial \beta(x,y)}{\partial y}$.
We discretize the domain $\Omega$ by vertical and horizontal lines passing through the points $(x_i,y_j)$ given by
$$x_i=x_0+ih, \;\;\;\ y_j=y_0+jl, \;\;\;\; i=0,1,2,...,M \;\;\;\text{and}\;\;\; j=0,1,2,...,N.$$
The mesh length along $x$- and $y$-directions are defined as $h = \displaystyle{\frac{x_f-x_0}{M-1}}$ and $l =\displaystyle{\frac{y_f-y_0}{N-1}}$ respectively (see figure \ref{dl}(a)). In order to a  derive a higher-order finite difference scheme on a compact stencil, we divide the grid points into the two following categories $:$ regular and irregular points. A point $(i,j)$ in the computational grid is said to be a regular point if the interface $\Gamma$ does not cross the standard five point stencil centered at $(i,j)$. In other words, if the grid points of the five point stencil share the same side of the interface either in $\Omega^{-}$ or $\Omega^{+}$ then it is called a regular point, and if it shares both sides, then it is an irregular point. We show the regular and irregular points in figure \ref{dl}(c), where the grid points marked red are irregular points and all the other points (marked black) are regular points.

A compact finite difference scheme \cite{gupta1985high,gupta1984single,kalita2002class,kalita2004transformation,kalita2001fully} utilizes grid points located only one step length away from the point about which the finite difference is considered. Additionally, if the order of accuracy is more than two the scheme is termed as a Higher-Order Compact (HOC) scheme. For discretizing equation eq(\ref{s5}) at regular points, we have utilized the HOC formulation developed by Kalita et al. \cite{kalita2002class} on uniform grids on a nine point stencil shown in figure  \ref{dl}(b), which is $O(h^4,l^4)$. At the point $(x_i,y_j)$ on the computational domain, this scheme is given by
\begin{equation}
\left[ A_{ij} \delta^{2}_{x} + B_{ij} \delta^{2}_{y} + C_{ij} \delta_{x} + D_{ij} \delta_{y} + E_{ij} \delta^{2}_{x}\delta^{2}_{y} + H_{ij} \delta_{x}\delta^{2}_{y} + K_{ij} \delta^{2}_{x}\delta_{y} +  L_{ij} \delta_{x} \delta_{y}  + M_{ij}\right] u_{ij}= F_{ij}. \label{s5}
\end{equation}
where $\delta^{2}_{x}$, $\delta^{2}_{y}$ , $ \delta_{x}$ , $\delta_{y}$, $\delta_{x} \delta_{y}$, $\delta_{x}\delta^{2}_{y}$, $\delta^{2}_{x}\delta_{y}$ and $\delta^{2}_{x}\delta^{2}_{y}$ are second order accurate central difference operators along $x$- and $y$- directions and,
\begin{eqnarray}
A_{ij}=&\beta_{ij} + \frac{h^{2}}{12 } \left(2 c_{x}+\kappa_{ij} + \beta_{xx}-\frac{c_{ij}}{\beta_{ij}} (c_{ij}+\beta_{x})\right) +\frac{l^{2}}{12 } \left( \beta_{yy} + \frac{d_{ij}}{\beta_{ij}}\beta_{y} \right), \nonumber \\
B_{ij}=&\beta_{ij} + \frac{h^{2}}{12 } \left( \beta_{xx} + \frac{c_{ij}}{\beta_{ij}}\beta_{c} \right)+ \frac{l^{2}}{12 } \left(2 d_{y}+\kappa_{ij} + \beta_{yy}-\frac{d_{ij}}{\beta_{ij}} (d_{ij}+\beta_{y})\right),\nonumber  \\
C_{ij}=&c_{ij} + \frac{h^{2}}{12 } \left( c_{xx}+2\kappa_{x}-\frac{c_{ij}}{\beta_{ij}}(c_{x}+\kappa_{ij} )\right)+ \frac{l^{2}}{12} \left(c_{yy}-\frac{d_{ij}}{\beta_{ij}}c_{y}\right),\nonumber \\
D_{ij}=&d_{ij} + \frac{h^{2}}{12 }\left(d_{xx}-\frac{c_{ij}}{\beta_{ij}}d_{x}\right) + \frac{l^{2}}{12} \left( d_{yy}+2\kappa_{y}-\frac{d_{ij}}{\beta_{ij}}(d_{y}+\kappa_{ij} )\right),\nonumber 
\end{eqnarray}
\begin{equation}
E_{ij}=\beta_{ij}\bigg(\frac{h^{2} }{12 } + \frac{l^{2}}{12 }\bigg) \textnormal{, }  H_{ij}=\frac{h^{2} }{12 } (2\beta_{x}-c_{ij})+ \frac{l^{2}}{12 }c_{ij} \textnormal{, }  K_{ij}=\frac{h^{2} }{12 } d_{ij} + \frac{l^{2}}{12 } (2\beta_{y}-d_{ij}), \nonumber 
\end{equation}
\begin{equation}
L_{ij}=\frac{h^{2}}{12}\left(2 d_{x} -\frac{c_{ij} d_{ij}}{\beta_{ij}}\right) + \frac{l^{2}}{12}\left(2 c_{y} -\frac{c_{ij} d_{ij}}{\beta_{ij}}\right),\nonumber 
\end{equation}
\begin{equation}
M_{ij}=\kappa_{ij} + \frac{h^{2}}{12 \beta_{ij}} ( \beta_{ij} \kappa_{xx} - c_{ij} \kappa_{x}) +\frac{l^{2}}{12 \beta_{ij}} (\beta_{ij} \kappa_{yy} - d_{ij} \kappa_{y}),\nonumber 
\end{equation}
\begin{equation}
F_{ij}=f_{ij} + \frac{h^{2}}{12 } f_{xx} + \frac{l^{2}}{12} f_{yy} -\frac{h^{2}}{12 \beta_{ij}} f_{x} c_{ij} - \frac{l^{2}}{12 \beta_{ij}} f_{y} d_{ij}.\nonumber
\end{equation}
On expanding, eq(\ref{s5}) reduces to
\begin{equation}
c_{1} u_{i-1,j-1}+c_{2} u_{i,j-1}+c_{3} u_{i+1,j-1}+c_{4} u_{i-1,j}+c_{5} u_{i,j}+c_{6} u_{i+1,j}+c_{7} u_{i-1,j+1}+c_{8} u_{i,j+1}+c_{9} u_{i+1,j+1}= G_{ij}.  \label{s6}
\end{equation}
where $G_{ij}=F_{ij}$ and others coefficients are given by,
\begin{equation}
c_{1}=\frac{E_{ij}}{h^{2}l^{2}}-\frac{H_{ij}}{2hl^{2}}-\frac{K_{ij}}{2h^{2}l} + \frac{L_{ij}}{4hl} \textnormal{, }  c_{2}=\frac{B_{ij}}{l^{2}}-\frac{D_{ij}}{2l}-\frac{2E_{ij}}{h^{2}l^{2}}+\frac{K_{ij}}{h^{2}l} \textnormal{, } c_{3}=\frac{E_{ij}}{h^{2}l^{2}}+\frac{H_{ij}}{2hl^{2}}-\frac{K_{ij}}{2h^{2}l} - \frac{L_{ij}}{4hl} \textnormal{, }
\end{equation}
\begin{equation}
c_{4}=\frac{A_{ij}}{h^{2}}-\frac{C_{ij}}{2h}-\frac{2E_{ij}}{h^{2}l^{2}}+\frac{H_{ij}}{hl^{2}} \textnormal{, }  c_{5}=-\frac{2A_{ij}}{h^{2}}-\frac{2B_{ij}}{l^{2}}+\frac{4E_{ij}}{h^{2}l^{2}}+M_{ij} \textnormal{, }  c_{6}=\frac{A_{ij}}{h^{2}}+\frac{C_{ij}}{2h}-\frac{2E_{ij}}{h^{2}l^{2}}-\frac{H_{ij}}{ hl^{2}} \textnormal{, }
\end{equation}
\begin{equation}
c_{7}=\frac{E_{ij}}{h^{2}l^{2}}+\frac{H_{ij}}{2hl^{2}}+\frac{K_{ij}}{2h^{2}l} + \frac{L_{ij}}{4hl} \textnormal{, } c_{8}=\frac{B_{ij}}{l^{2}}+\frac{D_{ij}}{2l}+\frac{2E_{ij}}{h^{2}l^{2}}-\frac{K_{ij}}{h^{2}l} \textnormal{, } c_{9}=\frac{E_{ij}}{h^{2}l^{2}}-\frac{H_{ij}}{2hl^{2}}+\frac{K_{ij}}{2h^{2}l} - \frac{L_{ij}}{4hl}  \textnormal{. }
\end{equation}
Because of the discontinuity in the solution $u$ at the interface, such an approximation doesn't work at irregular points. Therefore modification is required at these points. At the irregular points, we can have irregularity either in the $x$- or in $y$-directions alone, or both $x$- as well in $y$-directions at the same time. We modify the above HOC scheme at the irregular points if the interface crosses the finite difference grid in the  $x$- or and $y$-direction alone, and if it crosses in both the directions simultaneously, we adopt the mechanisms described in the following sections.
\subsection{Irregular points lying on grid lines parallel to $x$-axis only}\label{secx}
Consider the situation when the irregular point lies only on grid lines parallel to $x-$axis. Suppose we have the interface between $(i,j)$ and $(i+1,j)$ but there can be three possibilities (see Figure\ref{in_x}) the interface cut by the grid lines on three interfacial points in the compact stencil i.e $(x_{1}^{\star},y_{j+1})$ lies between $(x_{i},y_{j+1})$ and $(x_{i+1},y_{j+1})$ , $(x_{2}^{\star},y_{j})$ lies between $(x_{i},y_{j})$, $(x_{i+1},y_{j})$   and $(x_{3}^{\star},y_{j-1})$ lies between $(x_{i},y_{j-1})$, $(x_{i+1},y_{j-1})$. Since all six points lie on left of the interface (refer to case 1 of figure \ref{in_x}) having same sign of $\phi$ function,  we need to approximate $ u(x_{i+1},y_{j+1})$, $ u(x_{i+1},y_{j})$ and $ u(x_{i+1},y_{j-1})$ to the point $(x_{i},y_{j})$ using Taylor series expansions. We prove the following lemma, which will allow us to approximate higher-order derivatives including mixed derivatives resulting from discretization.
\begin{figure}[!h]
\minipage{0.33\textwidth}
  \includegraphics[width=\linewidth]{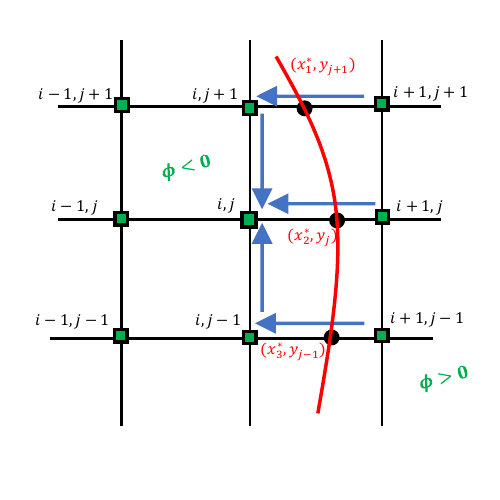}
  \begin{center}
   Case 1
  \end{center}
\endminipage\hfill
\minipage{0.33\textwidth}
  \includegraphics[width=\linewidth]{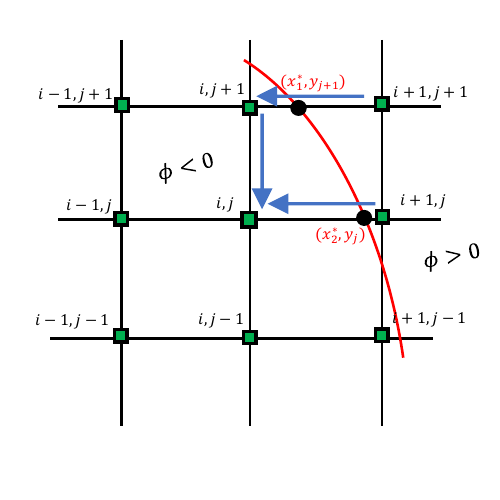}
  \begin{center}
   Case 2
  \end{center}
\endminipage\hfill
\minipage{0.33\textwidth}%
  \includegraphics[width=\linewidth]{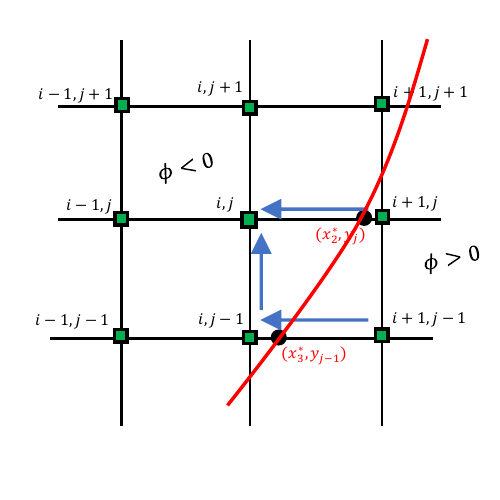}
  \begin{center}
   Case 3
  \end{center}
\endminipage
\caption{{\sl  Stencils for irregular points lying on grid lines parallel to $x$-axis only.}}
\label{in_x}
\end{figure}
\begin{lemma} \label{l2}
 Consider the interface lie between $(i,j)$ and $(i+1,j)$. Assume $u^{-} \in C^{k+1} [x_0,x_{1}^{\star}] \times [y_0,y_f]$, $u^{+} \in C^{k+1} [x_{1}^{\star},x_f] \times [y_0,y_f]$, $h_{1}^{+}=x_{i+1}-x_{1}^{\star}$ and $h_{1}^{-}$=$x_{i}-x_{1}^{\star}$ then we have the following inequality
\begin{eqnarray} \label{p1}
\bigg\Vert u(x_{i+1},y_{j+1}) -\sum_{p=0}^{k}\sum_{q=0}^{k-p} \frac{h^{p} l^{q}}{p! q!} \frac{\partial^{p+q} u}{\partial x^{p} \partial y^{q}}(x_{i},y_{j}) - \sum_{n=0}^{k} \frac{(h_{1}^{+})^{n} }{n!} \left[\frac{\partial^n u}{\partial x^{n}}(x_{1}^{\star},y_{j+1})\right]  \bigg \Vert \leq \nonumber
 K \frac{h^{k+1}}{(k+1)!}+\frac{M_{b}}{(k+1)!}(|h|+|l|)^{k+1}.
\end{eqnarray}
where $K$=$\max(\max_{x \in [x_{i},x_{1}^{\star}) } \mid u^{k+1}(x_{1}^{\star},y_{j+1})\mid$ , $\max_{x \in (x_{1}^{\star},x_{i+1}] } \mid u^{k+1}(x_{1}^{\star},y_{j+1})\mid)$
\end{lemma}
\begin{proof}
Using Taylor expansions for $u^{+}$ at $(x_{1}^{\star},y_{j+1})$ in $x$- direction
 \begin{equation}
u(x_{i+1},y_{j+1}) =\sum_{n=0}^{k} \frac{(h_{1}^{+})^{n} }{n!} \frac{\partial^n u^{+}}{\partial x^{n}}(x_{1}^{\star},y_{j+1}) + \frac{\partial^{k+1} u^{+}}{\partial x^{k+1}}(\xi_{k+1},y_{j+1}) \frac{(h_{1}^{+})^{k+1} }{(k+1)!} \label{a2}
 \end{equation}
 for some $\xi_{k+1}$ $\in$ $((x_{1}^{\star},x_{i+1}),y_{j+1})$. Also, we know that
 \begin{equation}
 \frac{\partial^n u^{+}}{\partial x^{n}}(x_{1}^{\star},y_{j+1})=\frac{\partial^n u^{-}}{\partial x^{n}}(x_{1}^{\star},y_{j+1})+ \left[\frac{\partial^n u}{\partial x^{n}}(x_{1}^{\star},y_{j+1})\right].  \label{a3}
 \end{equation}
substituting (\ref{a3}) in (\ref{a2}), we have 
\begin{equation}
u(x_{i+1},y_{j+1})=\sum_{n=0}^{k} \frac{(h_{1}^{+})^{n} }{n!} \left( \frac{\partial^n u^{-}}{\partial x^{n}}(x_{1}^{\star},y_{j+1})+ \left[\frac{\partial^n u}{\partial x^{n}}(x_{1}^{\star},y_{j+1})\right] \right) + \frac{\partial^{k+1} u^{+}}{\partial x^{k+1}}(\xi_{k+1},y_{j+1}) \frac{(h_{1}^{+})^{k+1} }{(k+1)!}, \label{a4}
\end{equation}
 Taylor expansions of $u^{-}$ at $(x_{i},y_{j+1})$ yields for $n=0, 1 , 2, . . . k$
 \begin{equation}
\frac{\partial^n u^{-}}{\partial x^{n}}(x_{1}^{\star},y_{j+1}) =\sum_{i=n}^{k} \frac{(-h_{1}^{-})^{i-n} }{(i-n)!} \frac{\partial^{i} u^{-}}{\partial x^{i}}(x_{i},y_{j+1}) +  \frac{(-h_{1}^{-})^{k-n+1} }{(k-n+1)!}  \frac{\partial^{k+1}  u^{-}}{\partial x^{k+1}}(\xi_{m},y_{j+1}),\label{a6}
 \end{equation}
 for some $\xi_{m}$ $\in$ $((x_{i},x_{1}^{\star}),y_{j+1})$. Since this point lies on the left side of the interface, we can replace $u^{-}$ simply by $u$. Making use of relation (\ref{a6}) in (\ref{a4}), therefore
 \begin{eqnarray}
 u(x_{i+1},y_{j+1}) =\sum_{n=0}^{k} \frac{(h_{1}^{+})^{n}}{n!} \left( \sum_{i=n}^{k}  \frac{(-h_{1}^{-})^{i-n}}{(i-n)!} \frac{\partial^{i} u}{\partial x^{i}}(x_{i},y_{j+1}) \right)+ \sum_{n=0}^{k} \frac{(h_{1 }^{+})^{n} }{n!} \left[\frac{\partial^n u}{\partial x^{n}}(x_{1}^{\star},y_{j+1})\right] \nonumber \\ 
+\sum_{n=0}^{k} \frac{(h_{1}^{+})^{n}}{n!}\frac{(-h_{1}^{-})^{k-n+1} }{ (k-n+1)! }  \frac{\partial^{k+1} u^{-}}{\partial x^{k+1}}(\xi_{m},y_{j+1}) + \frac{(h_{1}^{+})^{k+1}}{(k+1)!}\frac{\partial^{k+1} u}{\partial x^{k+1}}(\xi_{k+1},y_{j+1}).  \label{a7}
   \end{eqnarray}
   Now,
   \begin{equation}
   \sum_{n=0}^{k} \frac{(h_{1}^{+})^{n} }{n!} \left( \sum_{i=n}^{k}  \frac{(-h_{1}^{-})^{i-n} }{(i-n)!} \frac{\partial^{i} u}{\partial x^{i}}(x_{i},y_{j+1}) \right)= u^{(0)}(x_{i},y_{j+1})+ \left(\frac{h_{1}^{+}}{1!} +   \frac{(-h_{1}^{-})}{1!} \right)   \frac{\partial u}{\partial x} (x_{i},y_{j+1})  +\left(\frac{(h_{1}^{+})^{2}}{2!} + \frac{h_{1}^{+} (-h_{1}^{-})}{1! 1!} + \frac{(h_{1}^{-})^{2}}{2!} \right) \frac{\partial^{2} u}{\partial x^{2}}(x_{i},y_{j+1})\nonumber
   \end{equation}
   \begin{equation}
   + \hdots+\frac{\partial^{k} u}{\partial x^{k}} (x_{i},y_{j+1}) \bigg(\frac{(-h_{1}^{-})^{k}}{k!} + \frac{h_{1}^{+} {(-h_{1}^{-})^{k-1}}}{1! (k-1)!} + \frac{(h_{1}^{+})^{2}(-h_{1}^{-})^{k-2}}{2! (k-2)!} +... \frac{(h_{1}^{+})^{k-1} (-h_{1}^{-})}{(k-1)!}+ \frac{(h_{1}^{+})^{k}}{k!} \bigg),
   \end{equation}
   \begin{equation}
   = u^{(0)}(x_{i},y_{j+1})+ \frac{h^{1}}{1!} \frac{\partial u}{\partial x} (x_{i},y_{j+1})
 +\hdots \quad +\frac{h^{k}}{k!} \frac{\partial^{k} u}{\partial x^{k}} (x_{i},y_{j+1}),  \label{a9}
   \end{equation}
on applying the Binomial expansion, it becomes
   \begin{equation}
    \sum_{n=0}^{k} \frac{(h_{1}^{+})^{n} }{n!} \left( \sum_{i=n}^{k}  \frac{(-h_{1}^{-})^{i-n} }{(i-n)!} \frac{\partial^{i} u}{\partial x^{i}}(x_{i},y_{j+1}) \right)= \sum_{p=0}^{k} \frac{h^{p}}{p!} \frac{\partial^{p} u}{\partial x^{p}}(x_{i},y_{j+1}). \label{a11}
   \end{equation}
   using (\ref{a11}) in (\ref{a7})
   \begin{eqnarray}
   u(x_{i+1},y_{j+1}) =\sum_{p=0}^{k} \frac{h^{p}}{p!} \frac{\partial^{p} u}{\partial x^{p}}(x_{i},y_{j+1})+ \sum_{n=0}^{k} \frac{(h_{1 }^{+})^{n} }{n!} \left[\frac{\partial^n u}{\partial x^{n}}(x_{1}^{\star},y_{j+1})\right] +\sum_{n=0}^{k}  \frac{(h_{1}^{+})^{n}(-h_{1}^{-})^{k-n+1}}{n!(k-n+1)!} \frac{\partial^{k+1} u^{-}}{\partial x^{k+1}}(\xi_{m},y_{j+1}) \nonumber\\
    + \frac{(h_{1}^{+})^{k+1}}{(k+1)!}  \frac{\partial^{k+1} u}{\partial x^{k+1}}(\xi_{k+1},y_{j+1}).\label{a12}
   \end{eqnarray}
   Now we apply the Taylor series expansion for the first part of the above equation in $y$-direction and $p+q\leq n$ i.e
   \begin{equation}
   \sum_{p=0}^{k} \frac{h^{p}}{p!} \frac{\partial^{p} u}{\partial x^{p}}(x_{i},y_{j+1})=\sum_{p=0}^{k}\sum_{q=0}^{k-p} \frac{h^{p} l^{q}}{p! q!} \frac{\partial^{p+q} u}{\partial x^{p} \partial y^{q}}(x_{i},y_{j})+ \frac{1}{(k+1)!} U^{k+1}(\eta).\label{a13}
   \end{equation}
Suppose that all the partial derivative of u of order $(k+1)$ is bounded by $M_{b}$,
   \begin{equation}
    |U^{k+1}(\eta)|=\left |\sum_{r=0}^{k+1}\binom{k+1}{r} h^{r} l^{k+1-r} \frac{\partial^{(k+1)} u}{\partial x^{r} \partial y^{k+1-r}}(\eta)\right | \leq M_{b}(|h|+|l|)^{k+1}.\label{a14}
   \end{equation}
Substituting (\ref{a14}) into (\ref{a13}) and (\ref{a13}) into (\ref{a12})
    \begin{eqnarray}\label{n1}
 u(x_{i+1},y_{j+1}) =\sum_{p=0}^{k}\sum_{q=0}^{k-p} \frac{h^{p} l^{q}}{p! q!} \frac{\partial^{p+q} u}{\partial x^{p} \partial y^{q}}(x_{i},y_{j})+\frac{M_{b}}{(k+1)!}(|h|+|l|)^{k+1}+ \sum_{n=0}^{k} \frac{(h_{1 }^{+})^{n} }{n!} \left[\frac{\partial^n u}{\partial x^{n}}(x_{1}^{\star},y_{j+1})\right] \nonumber \\ 
+\sum_{n=0}^{k} \frac{\partial^{k+1} u}{\partial x^{k+1}}(\xi_{m},y_{j+1}) \frac{(h_{1}^{+})^{n}(-h_{1}^{-})^{k-n+1}}{n! (k-n+1)!} + \frac{\partial^{k+1} u}{\partial x^{k+1}}(\xi_{k+1},y_{j+1}) \frac{(h_{1}^{+})^{k+1}}{(k+1)!},
   \end{eqnarray}
Letting $K$=$\displaystyle{max}(\displaystyle{max}_{x \epsilon [x_{i},x_{1}^{\star}) } \mid u^{k+1}(x_{1}^{\star},y_{j+1})\mid$ , $max_{x \epsilon (x_{1}^{\star},x_{i+1}] } \mid u^{k+1}(x_{1}^{\star},y_{j+1})\mid$)\\
   \begin{equation}\label{n2}
   \sum_{n=0}^{k} \frac{\partial^{k+1} u}{\partial x^{k+1}}(\xi_{m},y_{j+1}) \frac{(h_{1}^{+})^{n}(-h_{1}^{-})^{k-n+1}}{n! (k-n+1)!} + \frac{\partial^{k+1} u}{\partial x^{k+1}}(\xi_{k+1},y_{j+1}) \frac{(h_{1}^{+})^{k+1}}{(k+1)!} \leq K \sum_{n=0}^{k+1}  \frac{(h_{1}^{+})^{n}(-h_{1}^{-})^{k-n+1}}{n! (k-n+1)!} \leq K \frac{h^{k+1}}{(k+1)!}.
   \end{equation}
 Substituting (\ref{n2}) into (\ref{n1}) and get the desired result.
\end{proof}
\begin{remark}\label{r2}
 Let  $h_{3}^{+}=x_{i+1}-x_{3}^{\star}$ and $h_{3}^{-}$=$x_{i}-x_{3}^{\star}$ then we have the following inequality
\begin{eqnarray}
\bigg\Vert u(x_{i+1},y_{j-1}) -\sum_{p=0}^{k}\sum_{q=0}^{k-p} \frac{h^{p} (-l)^{q}}{p! q!} \frac{\partial^{p+q} u}{\partial x^{p} \partial y^{q}}(x_{i},y_{j}) - \sum_{n=0}^{k} \frac{(h_{3}^{+})^{n} }{n!} \left[\frac{\partial^n u}{\partial x^{n}}(x_{3}^{\star},y_{j-1})\right]  \bigg \Vert \leq O(h^{k+1},l^{k+1}).
\end{eqnarray}
\end{remark}  
 \begin{remark} \label{r3}
Similarly, for the other side of the interface.
\begin{eqnarray}
\bigg\Vert u(x_{i-1},y_{j+1}) -\sum_{p=0}^{k}\sum_{q=0}^{k-p} \frac{(-h)^{p} l^{q}}{p! q!} \frac{\partial^{p+q} u}{\partial x^{p} \partial y^{q}}(x_{i},y_{j}) + \sum_{n=0}^{k} \frac{(h_{1}^{-})^{n} }{n!} \left[\frac{\partial^n u}{\partial x^{n}}(x_{1}^{\star},y_{j+1})\right]  \bigg \Vert \leq O(h^{k+1},l^{k+1}).
\end{eqnarray}
\begin{eqnarray}\label{a15}
\bigg\Vert u(x_{i-1},y_{j-1}) -\sum_{p=0}^{k}\sum_{q=0}^{k-p} \frac{(-h)^{p} (-l)^{q}}{p! q!} \frac{\partial^{p+q} u}{\partial x^{p} \partial y^{q}}(x_{i},y_{j}) + \sum_{n=0}^{k} \frac{(h_{3}^{-})^{n} }{n!} \left[\frac{\partial^n u}{\partial x^{n}}(x_{3}^{\star},y_{j-1})\right]  \bigg \Vert \leq O(h^{k+1},l^{k+1}).
\end{eqnarray}
\end{remark}   
 Now, we approximate the mixed derivative $u_{xy}$ at the point $(x_{i},y_{j})$ using the Lemma (\ref{l2}) and Remark (\ref{r2}) 
\begin{eqnarray}\label{a17}
\begin{split}
u_{xy}(x_{i},y_{j}) =&\frac{u_{i+1,j+1}-u_{i+1,j-1}-u_{i-1,j+1}+u_{i-1,j-1}}{4hl} -\frac{1}{4hl}\left(C_{1}x-C_{3}x \right)+O(h^{2},l^{2}),\\
u_{xxy}(x_{i},y_{j}) =&\frac{u_{i,j-1}-u_{i,j+1}}{h^{2} l}+\frac{u_{i+1,j+1}-u_{i+1,j-1}+u_{i-1,j+1}-u_{i-1,j-1}}{2h^{2}l} -\frac{1}{2h^{2}l}\left(C_{1}x-C_{3}x \right)+O(h^{2},l^{2}),\\
u_{xyy}(x_{i},y_{j}) =&\frac{u_{i-1,j}-u_{i+1,j}}{h l^{2} }+\frac{u_{i+1,j+1}+u_{i+1,j-1}-u_{i-1,j+1}-u_{i-1,j-1}}{2hl^{2}} -\frac{1}{2hl^{2}}\left(C_{1}x-2C_{2}x+C_{3}x \right)+O(h^{2},l^{2}),\\
u_{xxyy}(x_{i},y_{j}) =&\frac{4u_{i,j}-2(u_{i-1,j}+u_{i+1,j}+u_{i,j-1}+u_{i,j+1})}{h^{2} l^{2} }+\frac{u_{i+1,j+1}+u_{i+1,j-1}+u_{i-1,j+1}+u_{i-1,j-1}}{h^{2}l^{2}} \\&-\frac{1}{h^{2}l^{2}}\left(C_{1}x-2C_{2}x+C_{3}x \right)+O(h^{2},l^{2}).
\end{split}
\end{eqnarray}

Similarly, other operators in equation (\ref{s5}) can also be approximated. Depending on how the interface crosses the grid-lines as depicted in figure \ref{in_x}, the right hand side of equation (\ref{s6}) can be rewritten. \\
Case $1$: $G_{ij}=F_{ij}+c_{9} C_{1}x +c_{6}C_{2}x + c_{3} C_{3}x $,\\
Case $2$: $G_{ij}=F_{ij}+c_{9} C_{1}x +c_{6}C_{2}x $, and\\
Case $3$: $G_{ij}=F_{ij}+c_{6}C_{2}x + c_{3} C_{3}x $.\\
which shows that the jumps can be obtained explicitly at the irregular points, where the jump corrections are given by   $\displaystyle{C_{1}x=\sum_{n=0}^{k}\frac{(h_{1}^{+})^{n}}{n!} \left[\frac{\partial^n u}{\partial x^{n}}(x_{1}^{\star},y_{j+1})\right]}$, $\displaystyle{C_{2}x=\sum_{n=0}^{k}\frac{(h_{2}^{+})^{n}  }{n!} \left[\frac{\partial^n u}{\partial x^{n}}(x_{2}^{\star},y_{j})\right]}$ and $\displaystyle{C_{3}x=\sum_{n=0}^{k}\frac{(h_{3}^{+})^{n} }{n!}  \left[\frac{\partial^n u}{\partial x^{n}}(x_{3}^{\star},y_{j-1})\right]}$ with $\displaystyle{h_{2}^{+}=x_{i+1}-x_{2}^{\star}}$. A close look at the equations \eqref{p1}-\eqref{a17} along with the above jump correction would reveal that by choosing $k=3$ in those expressions, equation \eqref{s5} yields an $O(h^4,l^4)$ approximation at the irregular points as well. The same conclusions can be drawn from the irregular points across the other side of the interface as well as would be seen in sections \ref{secy} and \ref{secxy}.
\subsection{Irregular points lying on grid lines parallel to $y$-axis only}\label{secy}
In this section we discuss the scenarios when the irregularity lies only on points lying on grid lines parallel to $y$-axis. Let the interface cut between $(i,j)$ and $(i,j+1)$ points. Similar to the cases related to $x$-axis, one can have three possibilities here also: the interface being cut by the grid lines on three interfacial points in the compact stencil i.e $\displaystyle{(x_{i-1},y_{1}^{\star})}$ lying between $(x_{i-1},y_{j})$ and $(x_{i-1},y_{j+1})$ , $(x_{i},y_{2}^{\star})$ lying between $(x_{i},y_{j})$, $(x_{i},y_{j+1})$   and $(x_{i+1},y_{3}^{\star})$ lies between $(x_{i+1},y_{j})$, $(x_{i+1},y_{j+1})$. We approximate $ u(x_{i-1},y_{j+1})$, $ u(x_{i},y_{j+1})$ and $ u(x_{i+1},y_{j+1})$ to the point $(x_{i},y_{j})$ by using following lemma:
\begin{lemma} \label{l3}
 Consider the interface lie between $(i,j)$ and $(i,j+1)$. Assume $u^{-} \in C^{k+1} [x_0,x_f] \times [y_0,y_{1}^{\star}]$, $u^{+} \in C^{k+1} [x_0,x_f] \times [y_{1}^{\star},d]$, $k_{1}^{+}=y_{j+1}-y_{1}^{\star}$ and $k_{1}^{-}$=$y_{j}-y_{1}^{\star}$ then we have the following inequality
\begin{equation}
\bigg\Vert u(x_{i-1},y_{j+1}) -\sum_{p=0}^{k}\sum_{q=0}^{k-p} \frac{(-h)^{p} l^{q}}{p! q!} \frac{\partial^{p+q} u}{\partial x^{p} \partial y^{q}}(x_{i},y_{j}) + \sum_{n=0}^{k} \frac{(k_{1}^{-})^{n} }{n!} \left[\frac{\partial^n u}{\partial y^{n}}(x_{i-1},y_{1}^{\star})\right]  \bigg \Vert \leq \nonumber K \frac{l^{k+1}}{(k+1)!}+\frac{M_{b}}{(k+1)!}(|h|+|l|)^{k+1}
\end{equation}
where $K$=$\max(\max_{y \in [y_{j},y_{1}^{\star}) } \mid u^{k+1}(x_{i-1},y_{1}^{\star})\mid$ , $\max_{y \in (y_{1}^{\star},y_{j+1}] } \mid u^{k+1}(x_{i-1},y_{1}^{\star})\mid)$\\
 \end{lemma}
\begin{proof}
Similar to the proof of Lemma \ref{l2}; firstly one has to apply Taylor series expansion in the $y$-direction followed  by an expansion along $x$-direction.
\end{proof}
Other approximation can be derived in same way as in the previous section. With these, the right hand side of equation \eqref{s6} can be written as\\
Case $4$: $G_{ij}=F_{ij}-c_{9} C_{3}y -c_{8}C_{2}y - c_{7}C_{1}y $\\
Case $5$: $G_{ij}=F_{ij}-c_{7} C_{1}y -c_{8} C_{2}y $\\
Case $6$: $G_{ij}=F_{ij}-c_{9} C_{3}y -c_{8} C_{2}y  $\\ 
where  $\displaystyle{k_{2}^{-}=y_{j}-y_{2}^{\star}}$, $\displaystyle{k_{3}^{-}=y_{j}-y_{3}^{\star}}$, $\displaystyle{C_{1}y=\sum_{n=0}^{k}\frac{(k_{1}^{-})^{n}}{n!}  \left[\frac{\partial^n u}{\partial y^{n}}(x_{i-1},y_{1}^{\star})\right]}$, $\displaystyle{C_{2}y=\sum_{n=0}^{k}\frac{(k_{2}^{-})^{n} }{n!}  \left[\frac{\partial^n u}{\partial y^{n}}(x_{i},y_{2}^{\star})\right]}$ and $\displaystyle{C_{3}y=\sum_{n=0}^{k}\frac{(k_{3}^{-})^{n} }{n!}  \left[\frac{\partial^n u}{\partial y^{n}}(x_{i+1},y_{3}^{\star})\right]}$.
\begin{figure}[!h]
\minipage{0.33\textwidth}
  \includegraphics[width=\linewidth]{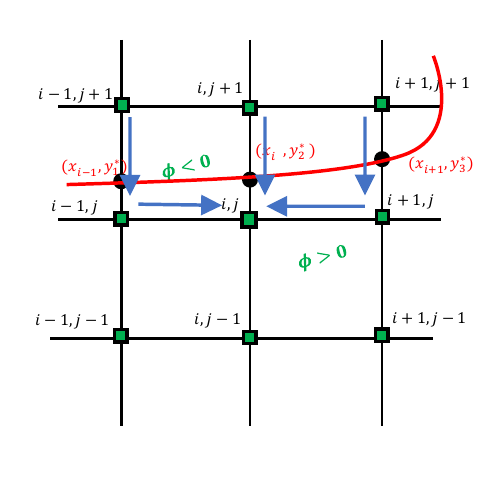}
  \begin{center}
   Case 4
  \end{center}
\endminipage\hfill
\minipage{0.33\textwidth}
  \includegraphics[width=\linewidth]{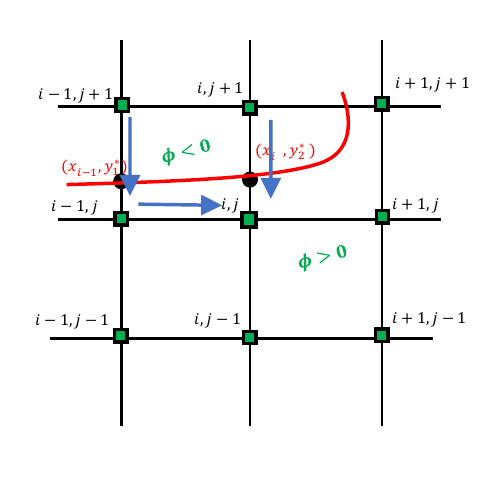}
  \begin{center}
   Case 5
  \end{center}
\endminipage\hfill
\minipage{0.33\textwidth}%
  \includegraphics[width=\linewidth]{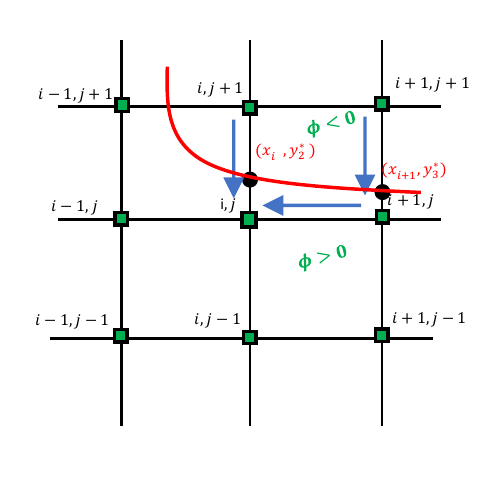}
  \begin{center}
   Case 6
  \end{center}
\endminipage
\caption{{\sl Stencils for irregular points lying on grid lines parallel to $x$-axis only.}}
\label{in_y}
\end{figure}

A close look at equation (\ref{a17}) and the expressions for $C_1x$, $C_2x$ and $C_3x$ would reveal that the jump conditions involve partial derivatives at the three interfacial points on different $y$-levels with respect to $x$ only. Likewise, the jump conditions $C_1y$, $C_2y$ and $C_3y$ above involve partial derivatives on different $x$-levels with respect to $y$ only. On the other hand, the formula for the approximation of mixed derivatives for jump conditions proposed in the EJIIM of Bube and Weigmann \cite{wiegmann2000explicit} used partial derivatives in both $x$ and $y$-directions simultaneously. However in actual computations, they used only a five point central difference stencil. This is probably the first time that a nine point compact stencil has been used for the jump conditions at the irregular points across the interface, unlike the SJIIM approach of Colnago {\it et al.} \cite{colnago2020high}. As a result, our proposed scheme maintains its compactness over a nine point stencil throughout the whole computational domain. Additionally, it also maintains its fourth order accuracy at both the regular and irregular points.
\vspace{-0.05cm}
\subsection{Irregular points lying simultaneously on grid lines parallel to both $x$-axis and $y$-axis } \label{secxy}
Here we have implemented the higher-order approximation at those points where irregularity lies in  $x$-axis as well in $y$-axis. We have used above lemma in both the axes simultaneously and followed it by taking the average of the correction terms. The two different cases as depicted in figure \ref{in_xy} have the followings in the right hand side of \eqref{s6}:
\begin{figure}[!h]
\minipage{0.40\textwidth}
  \includegraphics[width=\linewidth]{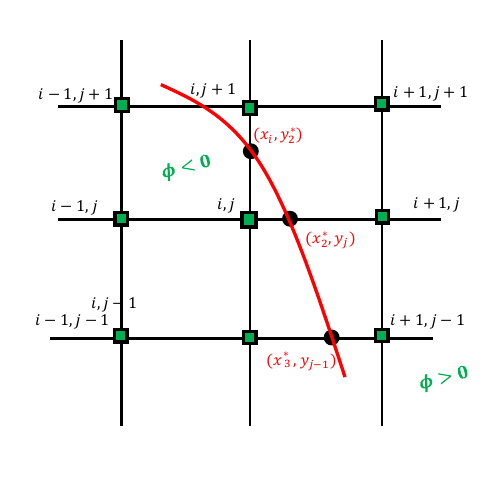}
  \begin{center}
   Case 7
  \end{center}
\endminipage\hfill
\minipage{0.405\textwidth}
  \includegraphics[width=\linewidth]{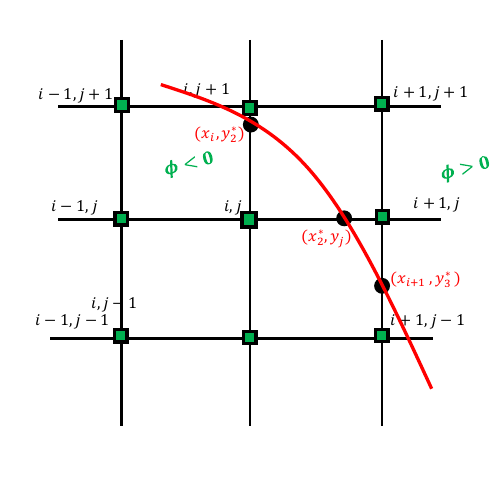}
  \begin{center}
   Case 8
  \end{center}
\endminipage
\caption{{\sl Stencils for irregular points lying on grid lines parallel to  both $x$-axis and $y$-axis.}}
\label{in_xy}
\end{figure}
\begin{eqnarray*}
\rm{Case 7:} G_{ij}&=F_{ij}+c_{9} (C_{2}y+C_{2}x)/2 +c_{8} C_{2}y +c_{6}C_{2}x + c_{3}(C_{2}x + C_{3}x )/2,	\\
\rm{Case 8:} G_{ij}&=F_{ij}+c_{9} (C_{2}y+C_{2}x)/2 +c_{8} C_{2}y +c_{6}C_{2}x. 
\end{eqnarray*}
where $C_{2}{x}, C_{3}{x}$ are same as defined in Section \ref{secx}, and $\displaystyle{C_{2}y=\sum_{n=0}^{k}\frac{(k_{2}^{+})^{n} }{n!}  \left[\frac{\partial^n u}{\partial y^{n}}(x_{i},y_{2}^{\star})\right]}$ with $k_{2}^{+}=y_{j+1}-y_{2}^{\star}$.
\section{Solution of Algebraic System}\label{algeb}
The matrix equation resulting from equation eq(\ref{s6}) can be written in the form
\begin{equation}
A{\bf u}={\bf b} \label{matrix1}
\end{equation}
where $A$ is a nona-diagonal matrix whose structure as follows.  
\begin{figure}[!h] 
\begin{center}
 \includegraphics[scale=0.5]{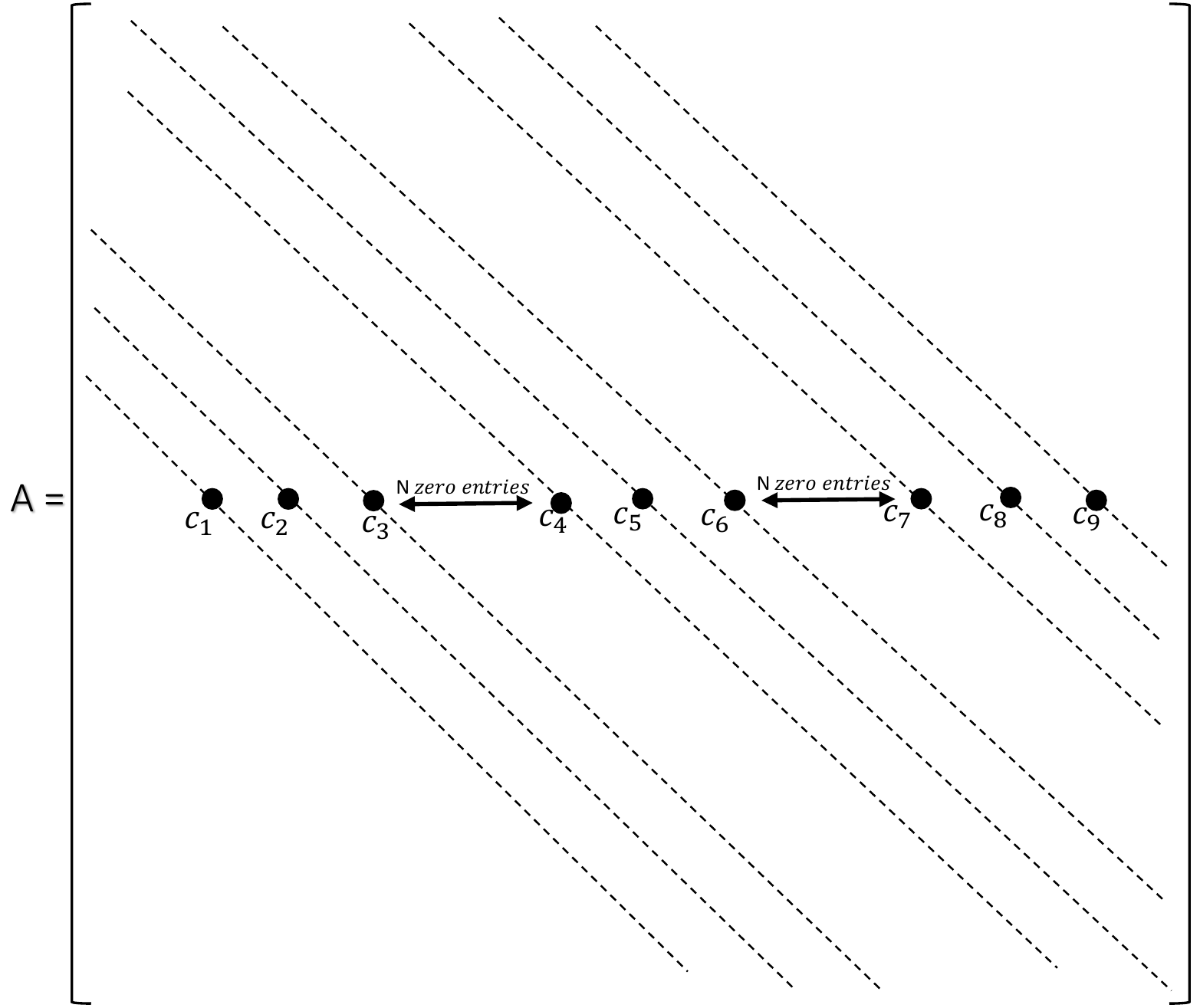} 
\label{matrix}
\end{center}
\end{figure}

Here one can see that nine non-zero coefficients corresponding to equation (\ref{s6}) defined on the stencils shown in figures \ref{dl}(b)-\ref{in_xy}. For a grid of size $\textit{MN}$, the coefficient matrix $A$ is of order $\textit{MN}$ and ${\bf u}$ and ${\bf b}$ are column matrices of order $\textit{MN} \times 1$. Equation eq(\ref{matrix1}) can be further decomposed as 
\[
\begin{bmatrix}
A_{R} & 0  \\
0 & A_{IR} 
\end{bmatrix}
\begin{bmatrix}
{\bf u_{R}}\\  {\bf u_{IR}} 
\end{bmatrix}
=
\begin{bmatrix}
{\bf b_{R}}\\  {\bf b_{R}+b_{C}} 
\end{bmatrix}
.
\]
The structures of $A_R$ and $A_{IR}$ are similar to $A$, and ${\bf u_R}$ and ${\bf u_{IR}}$   are $\textit{MN}-N_{IR}$ and $N_{IR}$ component vectors, where $N_{IR}$ is the number of irregular points inside the computational domain. Likewise the length of the vectors on the right hand side. Here ${\bf b_R}$ corresponds to the term $F_{ij}$ appearing in the list of coefficients following  equation \eqref{s6} and ${\bf b_{C}}$ corresponds to the correction terms $C_1x$, $C_2x$, $C_3x$, $C_1y$, $C_2y$, and $C_3y$ described in sections \ref{secx}, \ref{secy} and \ref{secxy}. The matrix equation \eqref{s6} is solved by the iterative solver biconjugate gradient stabilized (BiCGStab)\cite{kelley1995iterative}, where the iterations are stopped when the Euclidean norm of the residual vector ${\bf r}={\bf b}-A {\bf u}$ arising out of equation \eqref{s6} falls below $10^{-13}$. 

\section{Numerical Examples}
In order to study the efficiency of the proposed scheme and validate our algorithm, it has been applied to eight test cases. Four of them have analytical solutions while the remaining problems are governed by the highly non-linear Navier-Stokes (N-S) equations, viz., flow past bluff bodies including problems having multiple bodies in the flow domain. The problems have been chosen in such a way that they not only check the robustness of the proposed scheme in terms of tackling the varied nature of the interface geometry, but also in terms of possible discontinuities in the coefficients and the source terms.  As the first four problems
have analytical solutions, Dirichlet boundary conditions are used for them, whereas for the
ones governed by the N-S equations, both Dirichlet and Neumann boundary conditions are applied. All of our computations were carried out on a Intel Xeon processor based PC with 32 GB RAM.

\subsection{Test case 1}\label{ex1}
We consider the Poisson equation
\begin{equation}
 	u_{xx}+u_{yy}=2\int_{\Gamma}\delta(\overrightarrow{x}-\overrightarrow{X}(s))ds.  
\end{equation}
where the source term has a discontinuity in the form of Dirac delta function along the interface, $\Gamma$=$\lbrace (x,y)$, $x^{2}+y^{2}$=$1/4 \rbrace$, which is a circle of radius $0.5$. The jump conditions are given by $[u]$=$0$, $[u_n]$=$2$ and boundary conditions are derived from the analytical solution 
\begin{equation}
u(x,y)=\left\{\begin{array}{cc}
&1, \; \; \; \; \; \; \;\;\;\;\;\;\;\;\;\;\;\;\;\;\;\;\;\;\;\;\;\;\;\;\;\; \phi \leq 0\\
&1-log(2\sqrt{x^{2}+y^{2}}), \;\;\;\;\;\;
\phi > 0.
\end{array}\right.
\end{equation}
The level set function $\phi$ is defined as $\phi(x,y)$=$x^{2}+y^{2}-1/4$.

The surface plots of the computed solution on a grid of size $80 \times 80$ is shown in figure \ref{tcase1}(a). This figure clearly demonstrates the ability of the proposed scheme in resolving the sharp interface. While Berthelsen's decomposed immersed interface method (DIIM) \cite{berthelsen2004decomposed} found the use of higher-order differences at the interface complicating the computation owing to more grid points being roped in, our approach, despite using a nine point stencil was seen to capture the solution very efficiently. In table \ref{table_1}, we present the results from our computation on grids of sizes $N \times N$ with increasing values of $N$ and compare them with the  numerical results of \cite{berthelsen2004decomposed,leveque1994immersed,mittal2016class,wiegmann2000explicit}. The maximum error defined as $\displaystyle ||e||_{\infty}$, where $e=u_{\rm ex}-u_{\rm num}$ is tabulated as a function of the grid size $h$.  One can clearly see a much reduced error, decaying at a convergence rate (denoted by {\bf ROC} in this and subsequent tables) close to four, which is much higher than the ones reported in the existing literature. A graphical representation of the convergence rates of the current scheme along with that of the EJIIM \cite{wiegmann2000explicit} is provided in figure \ref{tcase1}(b), where the slope of the least square fit line shows the order of accuracy of the respective schemes. 
\begin{figure}[!t]
\begin{minipage}[b]{.45\linewidth}  
\includegraphics[scale=.45]{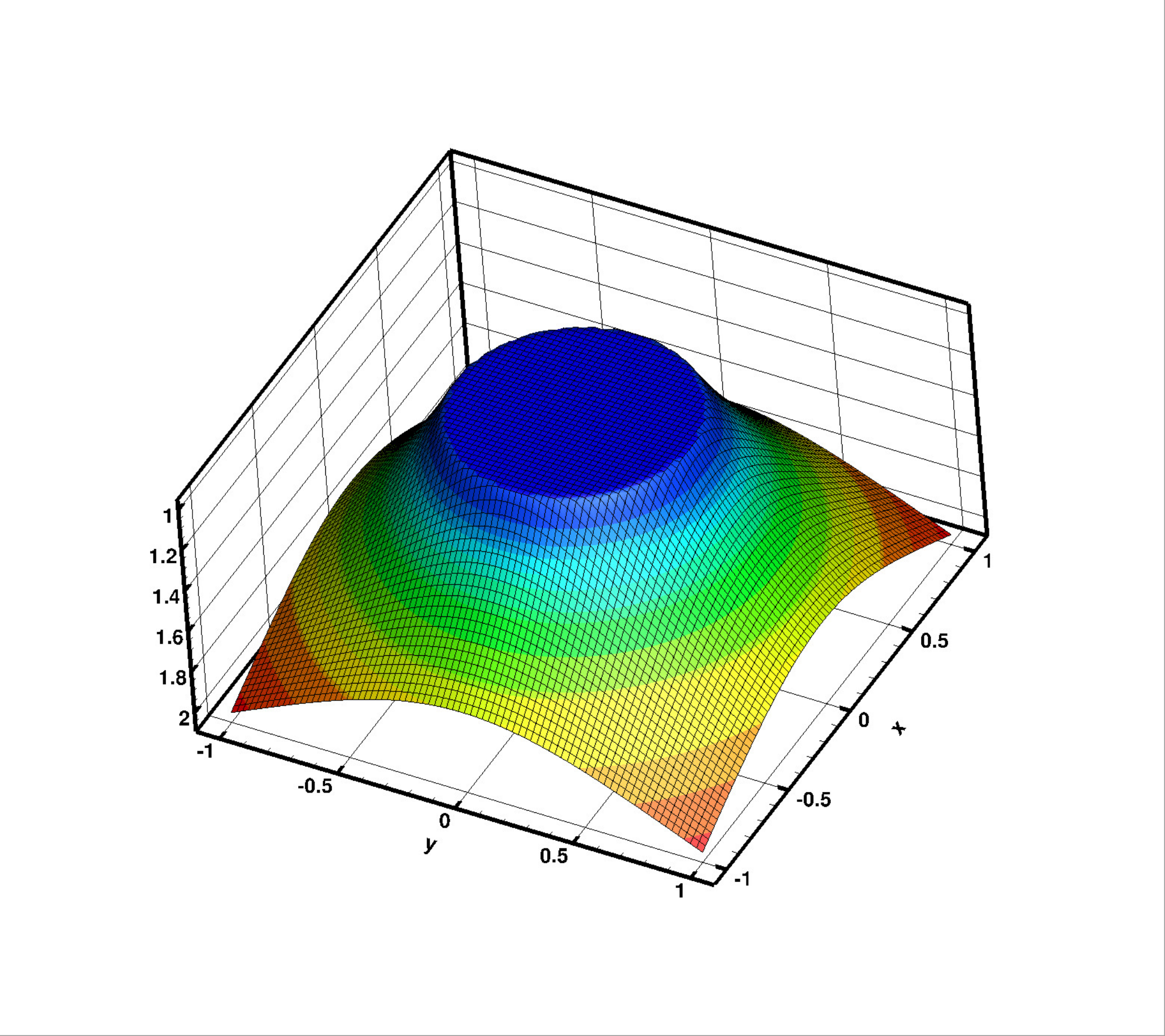} 
\centering (a)
\end{minipage}          
\begin{minipage}[b]{.45\linewidth}
\includegraphics[scale=.45]{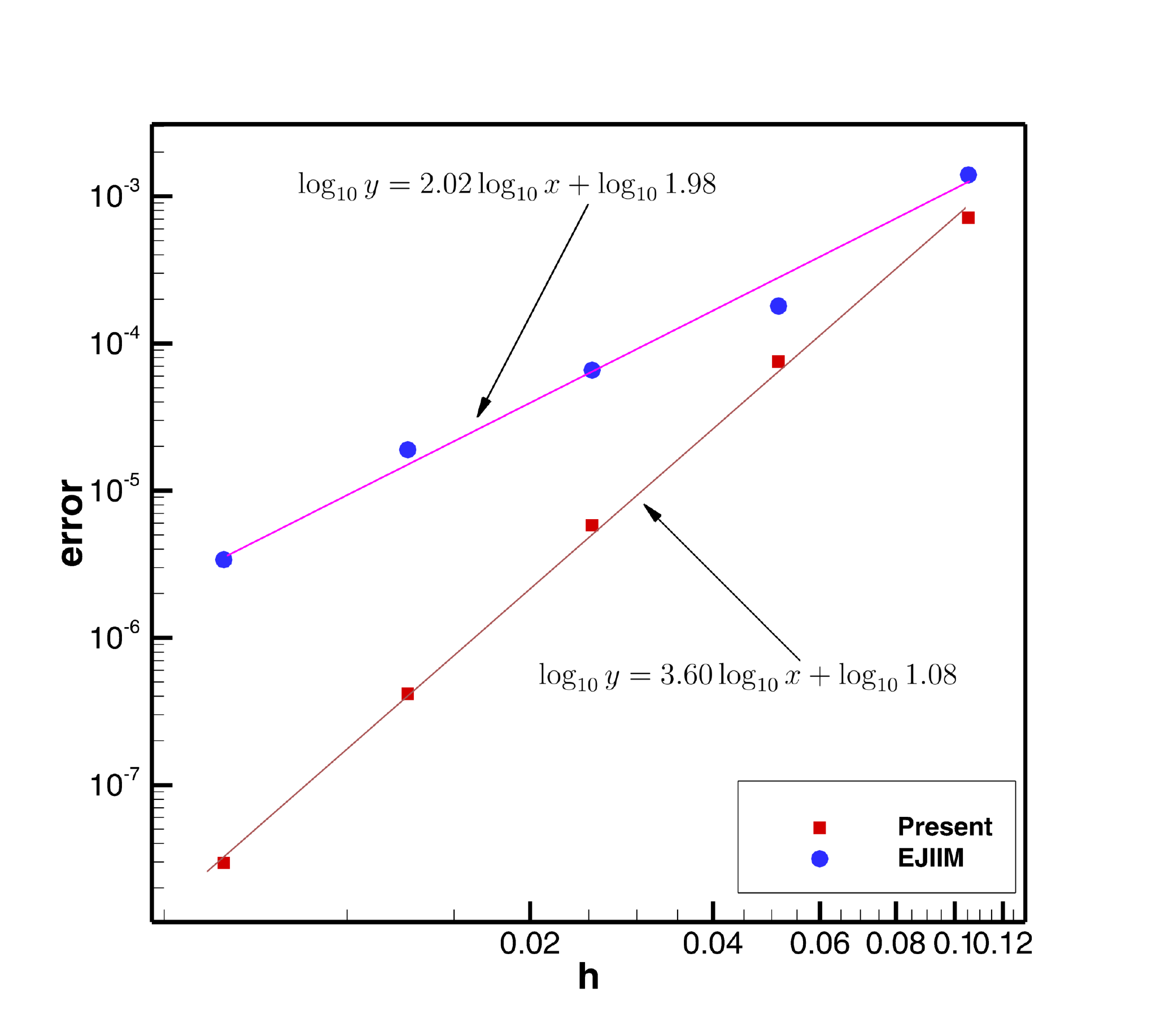} 
\centering (b)
\end{minipage} 
\caption{{\sl (a) Surface plots of the numerical solution and the error on grid size $80 \times 80$ and (b) the convergence results for Test Case 1.} }
\label{tcase1}
\end{figure}
\begin{figure}[!t]
\begin{minipage}[b]{.45\linewidth}  
\includegraphics[scale=.45]{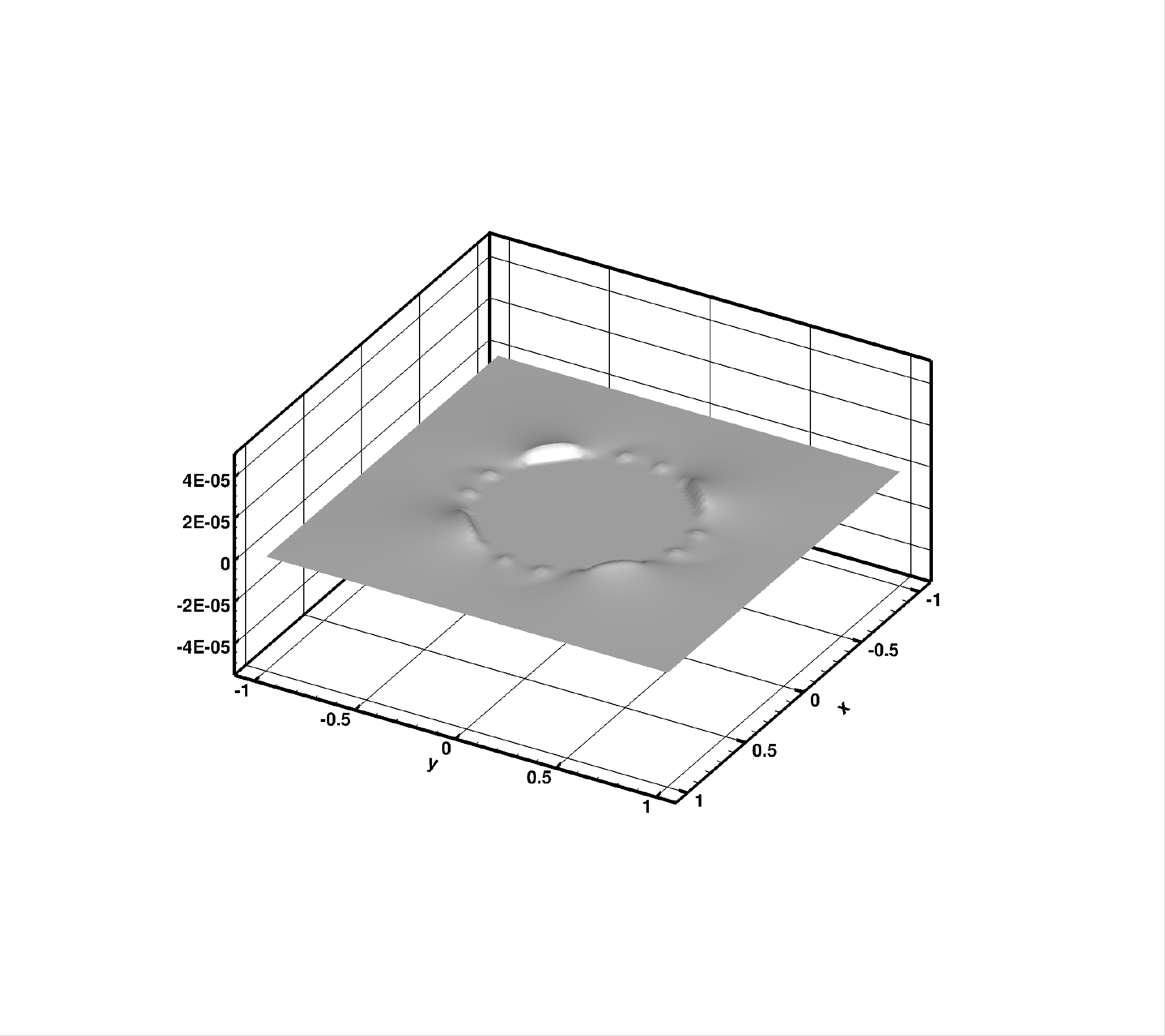} 
\centering (a)
\end{minipage}           
\begin{minipage}[b]{.45\linewidth}
\includegraphics[scale=.45]{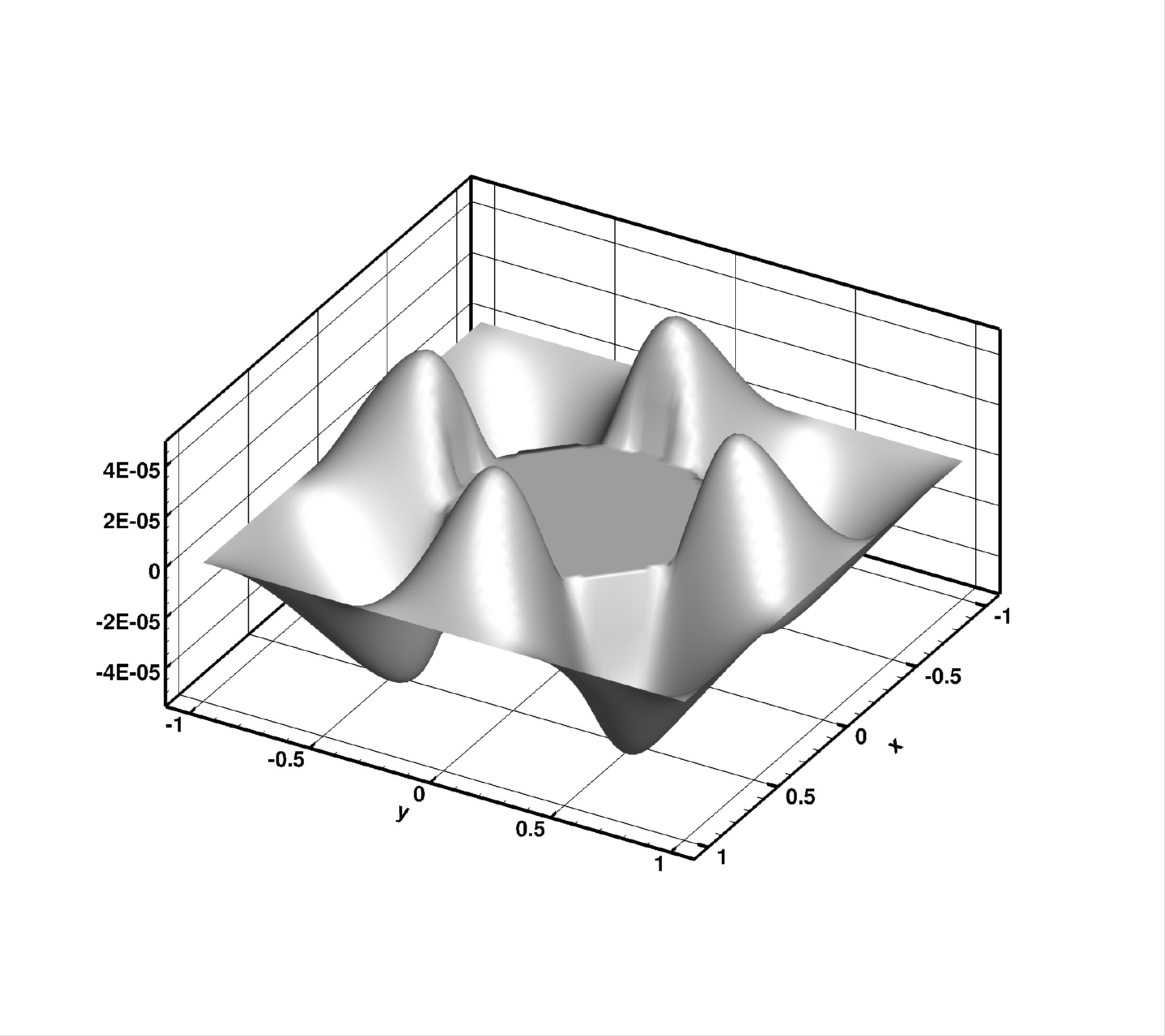}
\centering (b) 
\end{minipage} 
\caption{{\sl Error plots on a $80 \times 80$ grid by (a) The current approach and (b) EJIIM \cite{wiegmann2000explicit} for test case 1.} }
\label{they_us}
\end{figure}

In figure \ref{they_us}, we exhibit the surface plots of the errors resulting from the current computation on  a grid of size $80 \times 80$ along with the ones resulting from the  EJIIM \cite{wiegmann2000explicit}. The relative smoothness of the error in the neighbourhood of the interface along with the drastic reduction of the error clearly demonstrates the efficiency of the current approach over \cite{wiegmann2000explicit}. Further in table \ref{time_comparison} we have compared the CPU times consumed by the current scheme with EJIIM \cite{wiegmann2000explicit} on different grid sizes for test problem 1. As one can see, the CPU time loss because of the higher-order accuracy of the proposed scheme is minimal. However, in terms of the accuracy of the results, our scheme scores huge gain over EJIIM  as the error resulting from our computation decays much faster, which can be seen from figure \ref{they_us} and table \ref{table_1}.
\begin{table}[!h]
\caption{\small { Grid refinement analysis of maximum error for Test Case 1}.}
\begin{center}
\begin{tabular}{|c|c|c|c|c|c|c|c|c|}  \hline
N &	Present  &  ROC & DIIM \cite{berthelsen2004decomposed} & ROC & EJIIM \cite{wiegmann2000explicit} & ROC & CIM\cite{chern2007coupling} & ROC \\ \hline
 20   & $7.15 \times 10^{-4}$  & $-$ & $7.88 \times 10^{-4}$ & $-$ & $1.4 \times 10^{-3}$  &  $-$   & $7.60 \times 10^{-4}$ & $-$ \\
 40   & $7.54 \times 10^{-5}$ & $3.24$ & $2.01 \times 10^{-4}$ & $1.97$ & $1.8 \times 10^{-4}$   &  $2.95$  & $2.56 \times 10^{-4}$ & $1.56$\\
 80  & $5.82 \times 10^{-6}$  & $3.69$ & $5.03 \times 10^{-5}$ & $1.99$  & $6.6 \times 10^{-5}$ & $1.44$   & $5.21 \times 10^{-5}$ & $2.29$\\
 160  & $4.17 \times 10^{-7}$  & $3.80$ & $1.26 \times 10^{-5}$ & $2.01$ & $1.9 \times 10^{-5}$ & $1.79$ & $1.14 \times 10^{-5}$ & $2.19$\\ 
 320  & $2.96 \times 10^{-8}$  & $3.81$ & $3.18 \times 10^{-6}$ & $1.99$ & $3.4 \times 10^{-6}$ & $2.48$ & $2.72 \times 10^{-6}$ & $2.06$ \\ \hline
 \end{tabular} 
\end{center}
\label{table_1}
\end{table}
\begin{table}[!h]
\caption{ CPU time (in seconds) comparison with EJIIM \cite{wiegmann2000explicit} for Test Case 1.}
\begin{center}
\begin{tabular}{|c|c|c|}  \hline
N &	Present  & EJIIM \cite{wiegmann2000explicit} \\ \hline
 20   & $ 0.002$ & $ 0.002$  \\
 40   & $ 0.004$ & $ 0.004$\\
 80   & $ 0.018$ & $ 0.016$ \\
 160  & $ 0.134$ & $ 0.115$\\ 
 320  & $ 1.523$ & $ 1.142$\\ \hline
 \end{tabular}
\end{center}\label{time_comparison}
\end{table}
\subsection{Test case 2}
Here the differential equation considered is
  \begin{equation}
(\beta u_x)_x+(\beta
u_y)_y=f(x,y)+S \int_{\Gamma}\delta(\overrightarrow{x}-\overrightarrow{X}(s))ds.
\label{poisson}
\end{equation}
with $f(x,y)=8(x^2+y^2)+4$ defined on $\Omega$. $\delta$ is the Dirac delta function, $S$  is the strength of the point source at $\Gamma$ and coefficient $\beta$ is given by
\begin{equation}
 \beta(x,y)=\left\{\begin{array}{cc}
&x^2+y^2+1, \;\;\;\;\;\;\;\;\;\; \phi\leq 0\\
&b,\;\;\;\;\;\;\;\;\;\;\;\;\;\;\;\;\;\;\;\;\;\;\;\;\;\phi>0.
\end{array}\right.
\end{equation}
From the above, one can clearly see that the coefficient $\beta$ is discontinuous across the interface $x^2+y^2=r^2$ with $r = 1/2$. The analytical solution to this problem is given by
\begin{equation}
u(x,y)=\left\{\begin{array}{cc}
&r^2, \;\;\;\;\;\;\;\;\;\;\;\;\;\;\;\;\;\;\;\;\;\;\;\;\;\;\;\;\;\;\;\;\;\;\;\;\;\;\;\;\;\;\;\;\;\;\;\;\;\;\;\;\;\;\;\;\;\;\;\;\;\;\;\;\;\;\;\;\;\; \phi\leq 0\\
&(1-\frac{1}{8b}-\frac{1}{b})/4+(\frac{r^4}{2}+r^2)/b+S\log(2r/b),\;\;\;\;\;\;\;\;
\phi > 0.
\end{array}\right.
\end{equation}
We show our numerical results for $\beta=0.001$, $10$ and $1000$ in tables \ref{table_2}-\ref{table_4} and figures \ref{tcase2a}-\ref{tcase2b}. As in test case 1, when compared with established numerical results \cite{berthelsen2004decomposed, chern2007coupling, cisternino2012parallel, leveque1994immersed, limaximum, wiegmann2000explicit}, our results fare much better as the tabulated errors on different grid sizes suggest.
\begin{figure}[!h]
\begin{minipage}[b]{.45\linewidth}  
\includegraphics[scale=.45]{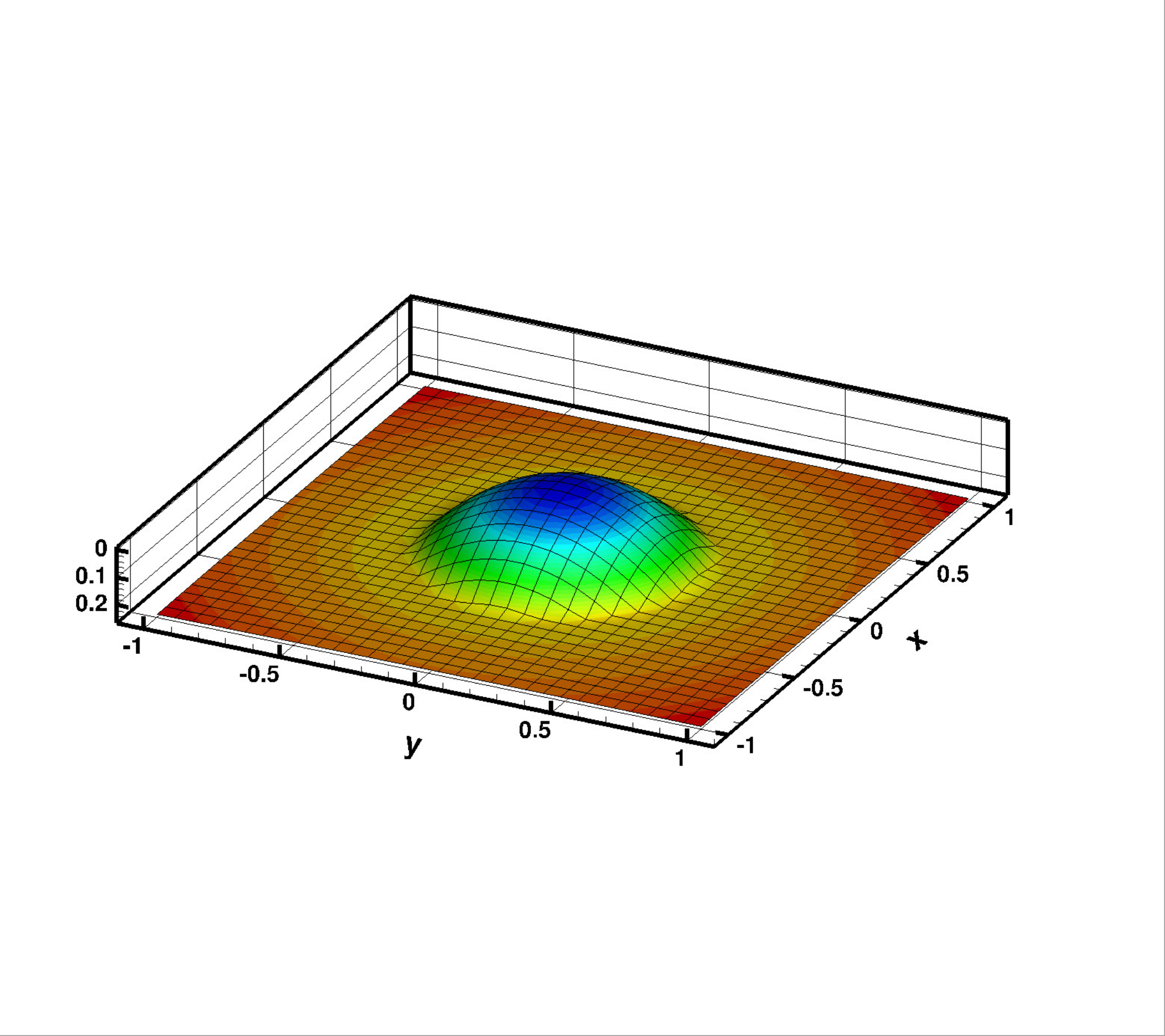} 
\end{minipage}          
\begin{minipage}[b]{.45\linewidth}
\includegraphics[scale=.45]{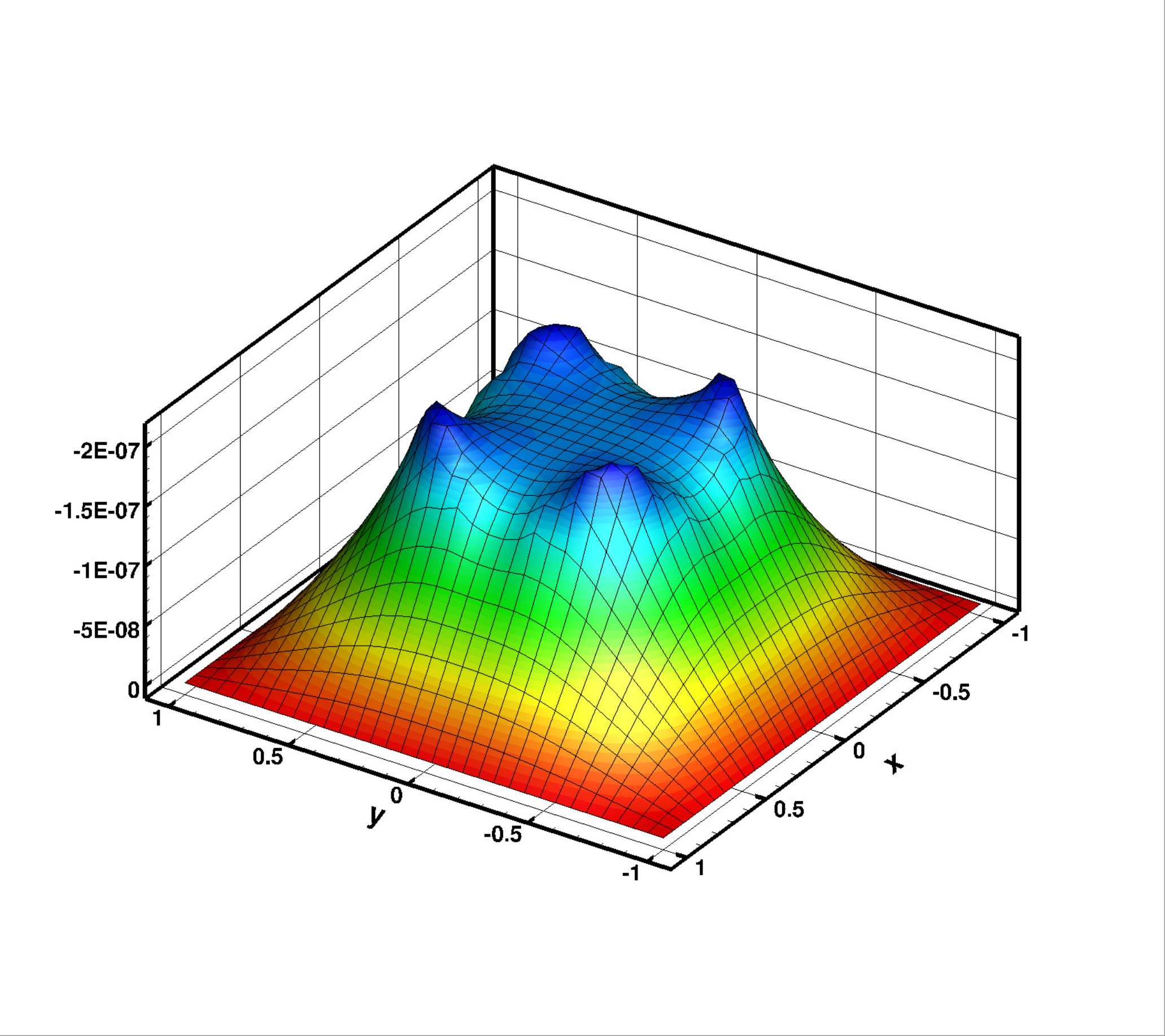} 
\end{minipage} 
\caption{{\sl  Surface plots of the numerical solution and the error on a grid size $32 \times 32$ for $b=1000$ for Test Case 2.} }
\label{tcase2a}
\end{figure}
\begin{table}[!h]
\caption{  Grid refinement analysis of maximum error for Test Case 2 for $b = 1000$.}
\begin{center}
\begin{tabular}{|c|c|c|c|c|c|c|}  \hline
N &	Present  &  \cite{mittal2018solving} & CIM \cite{chern2007coupling} & PCM \cite{cisternino2012parallel} & DIIM \cite{berthelsen2004decomposed} & MIM \cite{limaximum}   \\ \hline
 32   & $2.10 \times 10^{-7}$ &$1.98 \times 10^{-5}$ & $2.73 \times 10^{-4}$ &  $1.82 \times 10^{-4}$ & $2.08 \times 10^{-4}$ & $5.14 \times 10^{-4}$      \\
 64   & $2.65 \times 10^{-8}$ & $3.69 \times 10^{-6}$ & $3.88 \times 10^{-5}$ & $4.96 \times 10^{-5}$ & $5.30 \times 10^{-5}$ & $8.24 \times 10^{-5}$     \\
 128  & $3.17 \times 10^{-9}$ & $6.72 \times 10^{-7}$ & $5.34 \times 10^{-6}$ & $1.30 \times 10^{-5}$ & $1.33 \times 10^{-5}$ & $1.87 \times 10^{-5}$       \\
 256  & $4.32 \times 10^{-10}$ & $1.12 \times 10^{-7}$ & $7.24 \times 10^{-7}$ &  $3.33 \times 10^{-6}$ & $3.33 \times 10^{-6}$ & $4.03 \times 10^{-6}$    \\ \hline
\end{tabular}
\end{center}
\label{table_2}
\end{table}
\begin{figure}[!t]
\begin{minipage}[b]{.45\linewidth}  
\includegraphics[scale=0.40]{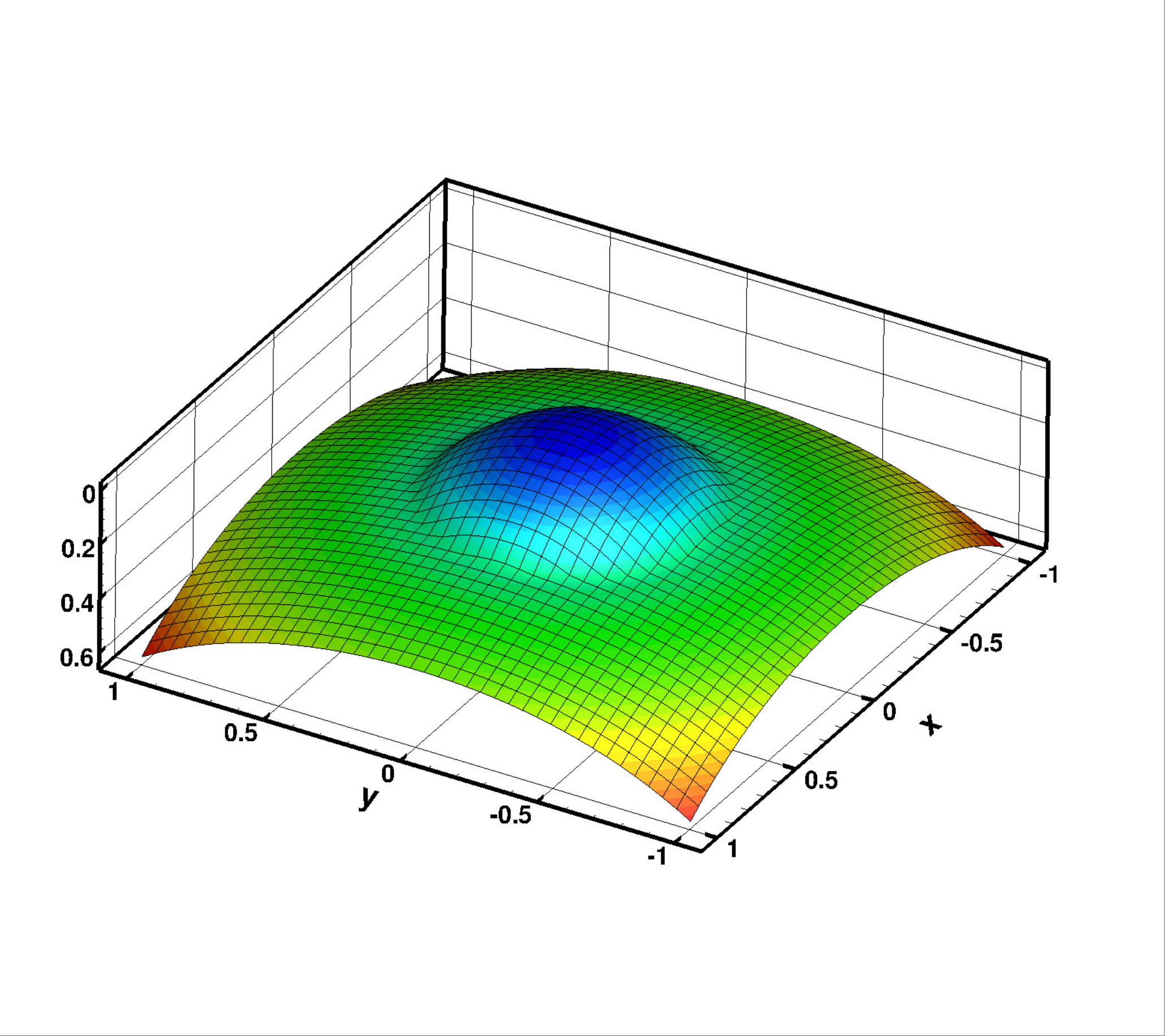} 
\end{minipage}           
\begin{minipage}[b]{.45\linewidth}
\includegraphics[scale=0.40]{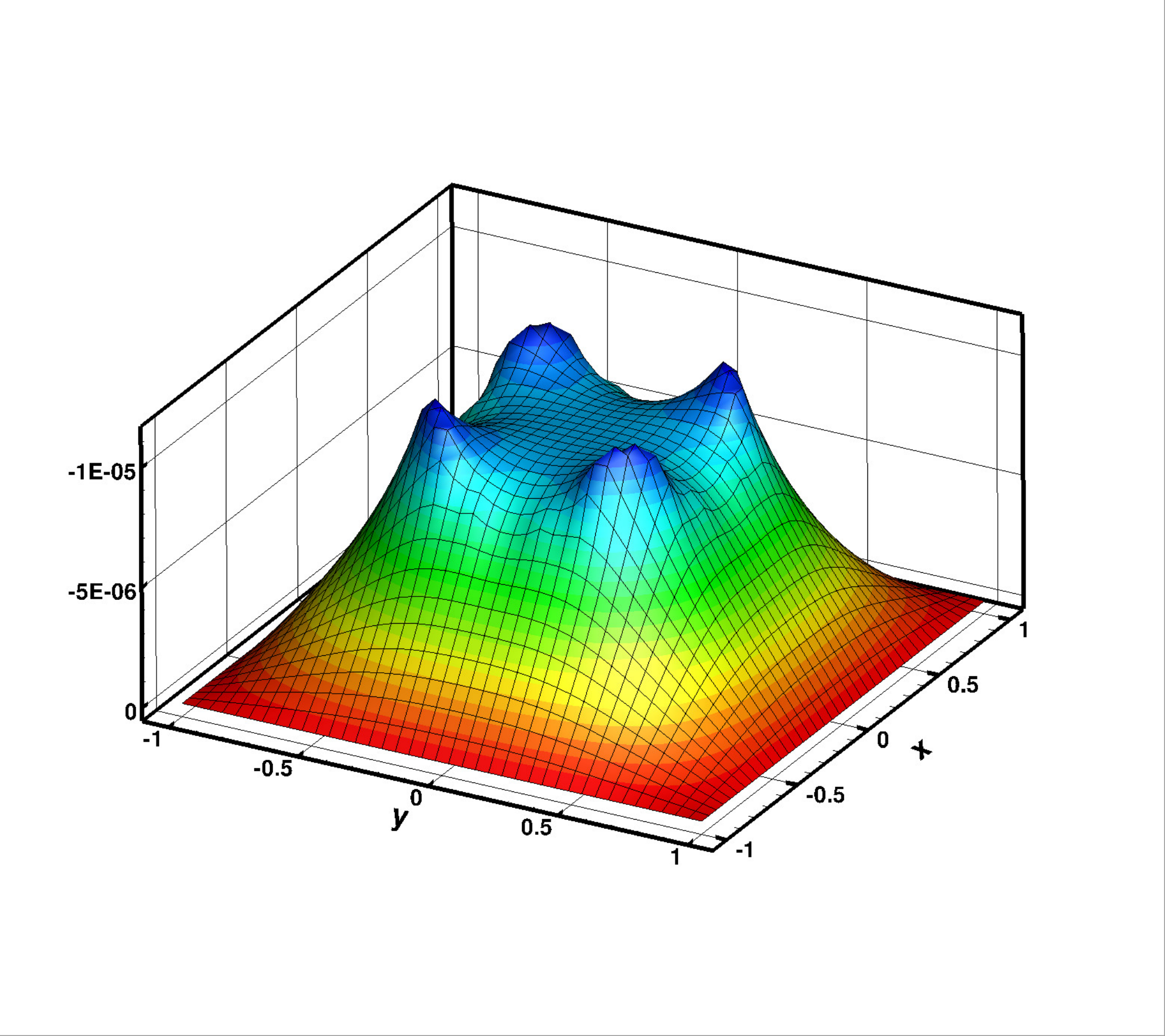} 
\end{minipage} 
\caption{{\sl  Numerical solution and contour plot of error on grid size $40 \times 40$ for $b=10$ for Test Case 2.} }
\label{tcase2b}
\end{figure}
\begin{table}[!h]
\caption{Grid refinement analysis of maximum error for Test Case 2 with $b = 0.001.$}
\begin{center}
\begin{tabular}{|c|c|c|c|c|c|c|}  \hline
N &	Present  & \cite{mittal2018solving} & CIM \cite{chern2007coupling} & PCM \cite{cisternino2012parallel}& DIIM \cite{berthelsen2004decomposed} & MIM  \cite{limaximum}   \\ \hline
 32   & $2.10 \times 10^{-1}$  & $1.26 \times 10^{-1}$ & $4.28 \times 10^{-1}$ & $2.03 \times 10^{0}$ & $4.97 \times 10^{0}$  &  $9.35 \times 10^{0}$    \\
 64   & $1.95 \times 10^{-2}$ & $2.12 \times 10^{-2}$ & $1.26 \times 10^{-1}$ &$3.52 \times 10^{-1}$ &  $1.18 \times 10^{0}$   &  $2.01 \times 10^{0}$  \\
 128  & $3.11 \times 10^{-3}$  & $3.85 \times 10^{-3}$ & $3.77 \times 10^{-2}$ & $7.25 \times 10^{-2}$  & $2.90 \times 10^{-1}$ & $5.80 \times 10^{-1}$    \\
 256  & $4.32 \times 10^{-4}$  & $6.42 \times 10^{-4}$ & $1.36 \times 10^{-2}$ & $1.80 \times 10^{-2}$ & $7.08 \times 10^{-2}$ & $1.37 \times 10^{-1}$ \\ \hline
 \end{tabular}
\end{center}
\label{table_3}
\end{table}
\begin{table}[!h]
\caption{ Grid refinement analysis of maximum error for Test Case 2 with $b =10$ and $S=0.1$.}
\begin{center}
\begin{tabular}{|c|c|c|c|c|c|c|c|c|}  \hline
N &	Present  &  ROC & PCM \cite{cisternino2012parallel}& ROC & DIIM \cite{berthelsen2004decomposed} & ROC & EJIIM \cite{wiegmann2000explicit} & ROC \\ \hline
 20   & $1.07 \times 10^{-4}$  & $-$ & $4.20 \times 10^{-4}$ & $-$ & $5.36 \times 10^{-4}$  &  $-$   & $7.6 \times 10^{-4}$ & $-$ \\
 40   & $1.11 \times 10^{-5}$ & $3.26$ & $1.16 \times 10^{-4}$ & $1.85$ & $1.38 \times 10^{-4}$   &  $1.95$  & $2.4 \times 10^{-4}$ & $1.66$\\
 80  & $1.30 \times 10^{-6}$  & $3.09$ & $3.75 \times 10^{-5}$ & $1.62$  & $3.47 \times 10^{-5}$ & $1.99$   & $7.90 \times 10^{-5}$ & $1.60$\\
 160  & $1.62 \times 10^{-7}$  & $3.00$ & $5.33 \times 10^{-6}$ & $2.81$ & $8.70 \times 10^{-6}$ & $1.99$ & $2.2 \times 10^{-5}$ & $1.84$\\ 
 320  & $2.05 \times 10^{-8}$  & $2.98$ & $1.58 \times 10^{-6}$ & $1.75$ & $2.17 \times 10^{-6}$ & $2.01$ & $5.3 \times 10^{-6}$ & $2.05$\\ \hline
 \end{tabular}
\end{center}
\label{table_4}
\end{table}
\begin{table}[!h]
\caption{ Comparison of maximum error for Test Case 2 with Feng et al \cite{feng2019augmented} for $\beta^{+} =10$, $\beta^{-}=2$.}
\begin{center}
\begin{tabular}{|c|c|c|}  \hline
 N &	Present  &  AMIB  \cite{feng2019augmented} \\ \hline
 32   & $2.09 \times 10^{-5}$  &$1.05 \times 10^{-4}$  \\
 64   & $2.65 \times 10^{-6}$ &  $2.02 \times 10^{-5}$\\
 128  & $3.17 \times 10^{-7}$  &  $5.90 \times 10^{-6}$ \\
 256  & $5.31 \times 10^{-8}$ & $1.04 \times 10^{-6}$\\  \hline
 \end{tabular}
\end{center}\label{table_feng}
\end{table}
 
Note that, recently, Feng et al. \cite{feng2019augmented} had also solved the above problem with a combination of the MIB \cite{zhou2006high}, Augmented IIM \cite{li2007augmented} and EJIIM \cite{wiegmann2000explicit} and claimed to have produced faster results. However, apart from the gain in CPU times, they could not improve the order of accuracy and the error as can be seen from table \ref{table_feng}.
 \subsection{Test case 3}
This test case is an example where there are discontinuities both in the diffusion coefficients as well as the source function simultaneously.  The equation is given by
\begin{equation}
(\beta u_x)_x+(\beta
u_y)_y=f(x,y).
\end{equation}
where the diffusion coefficient $\beta$ is given by
\begin{equation} 
 \beta(x,y)=\left\{\begin{array}{cc}
&\beta^{-}, \;\;\;\;\;\;\;\;\;\; \phi\leq 0\\
&\beta^{+},\;\;\;\;\;\;\;\;\;\;\;\;\phi>0
\end{array}\right.
\end{equation}
and
\begin{equation}
 f(x,y)=\left\{\begin{array}{cc}
&4/ \beta^{-}, \;\;\;\;\;\;\;\;\;\; \phi\leq 0\\
&16r^{2}/ \beta^{+}\;\;\;\;\;\;\;\;\phi>0.
\end{array}\right.
\end{equation}
The problem has the analytical solution
\begin{equation}
u(x,y)=\left\{\begin{array}{cc}
&\frac{x^{2}+y^{2}} { \beta^{-}}, \;\;\;\;\;\;\;\;\;\;\;\;\;\;\;\;\;\;\;\;\;\;\;\;\;\;\;\;\;\;\;\;\;\;\;\;\;\;\;\;\;\;\;\;\;\;\;\;\;\;\;\;\;\;
\;\;\;\;\;\;\;\;\;\;\;\;\;\;\;\;\;\;\;\;\;\;\;\;\;\;\;\;\;\phi\leq 0\\
&\frac{(x^{2}+y^{2})^{2}+S_{0} \log(2 \sqrt{x^{2}+y^{2}})}{\beta^{+}}+S_{1} \left(\frac{r_{0}^{2}}{\beta^{-}} - \frac{(x^{2}+y^{2})^{2}+S_{0} \log(2r_{0})}{\beta^{+}} \right) \;\;\;\;\;\;\;\;
\phi > 0.
\end{array}\right.
\end{equation}
\begin{figure}[!h]
\minipage{0.33\textwidth}
  \includegraphics[width=\linewidth]{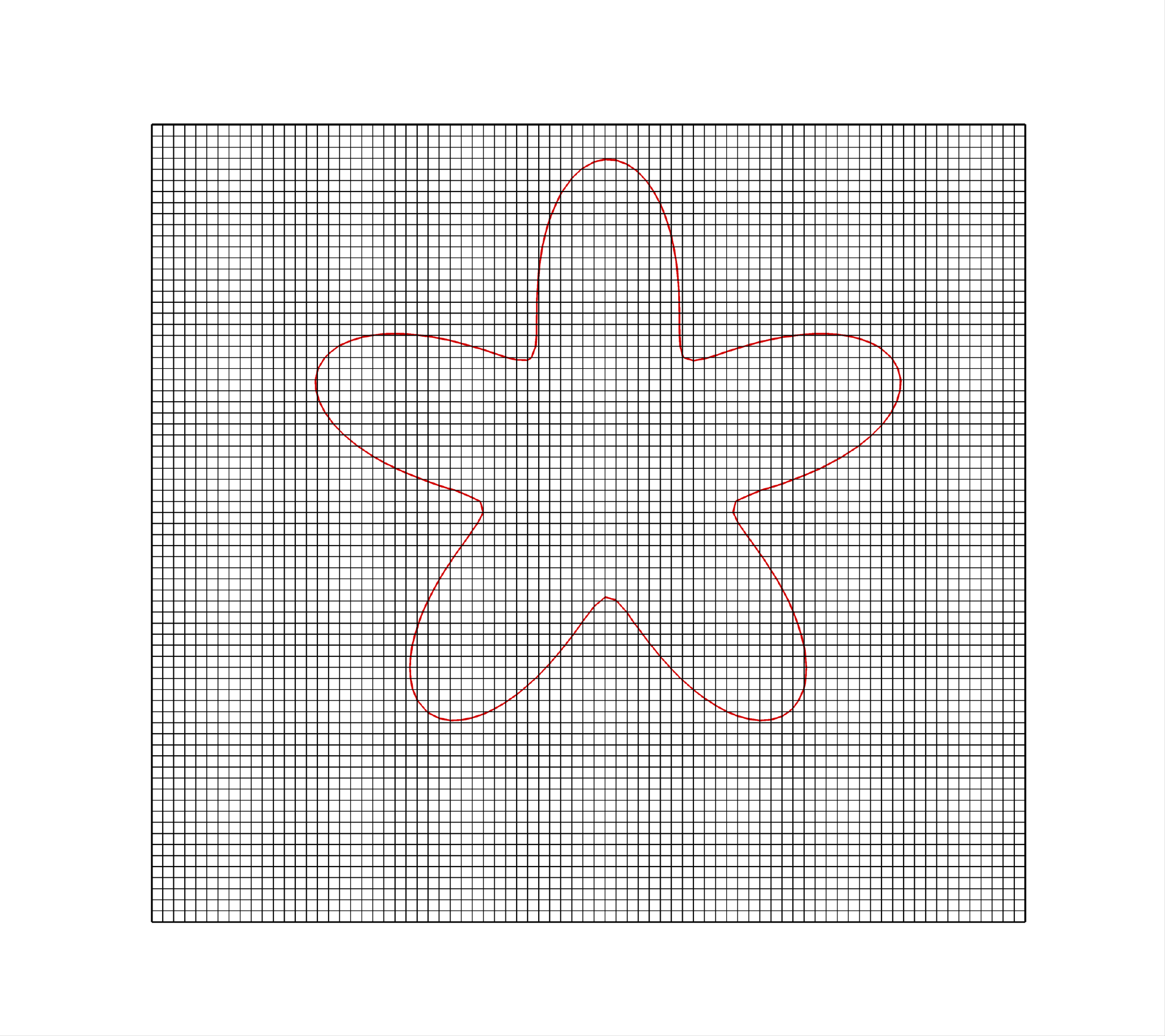}
  \begin{center}
     \end{center}
\endminipage\hfill
\minipage{0.33\textwidth}
  \includegraphics[width=\linewidth]{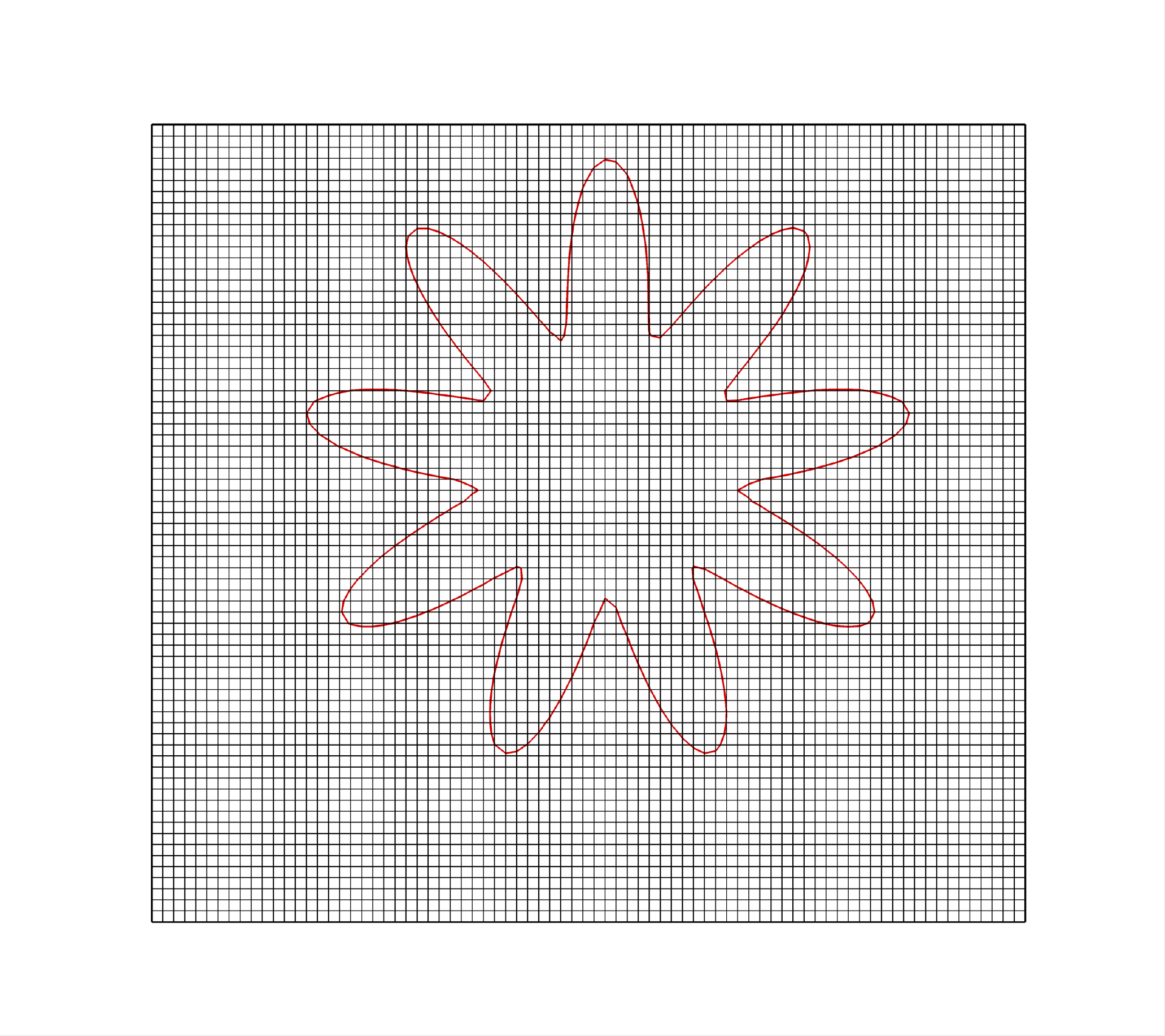}
  \begin{center}
  \end{center}
\endminipage\hfill
\minipage{0.33\textwidth}%
  \includegraphics[width=\linewidth]{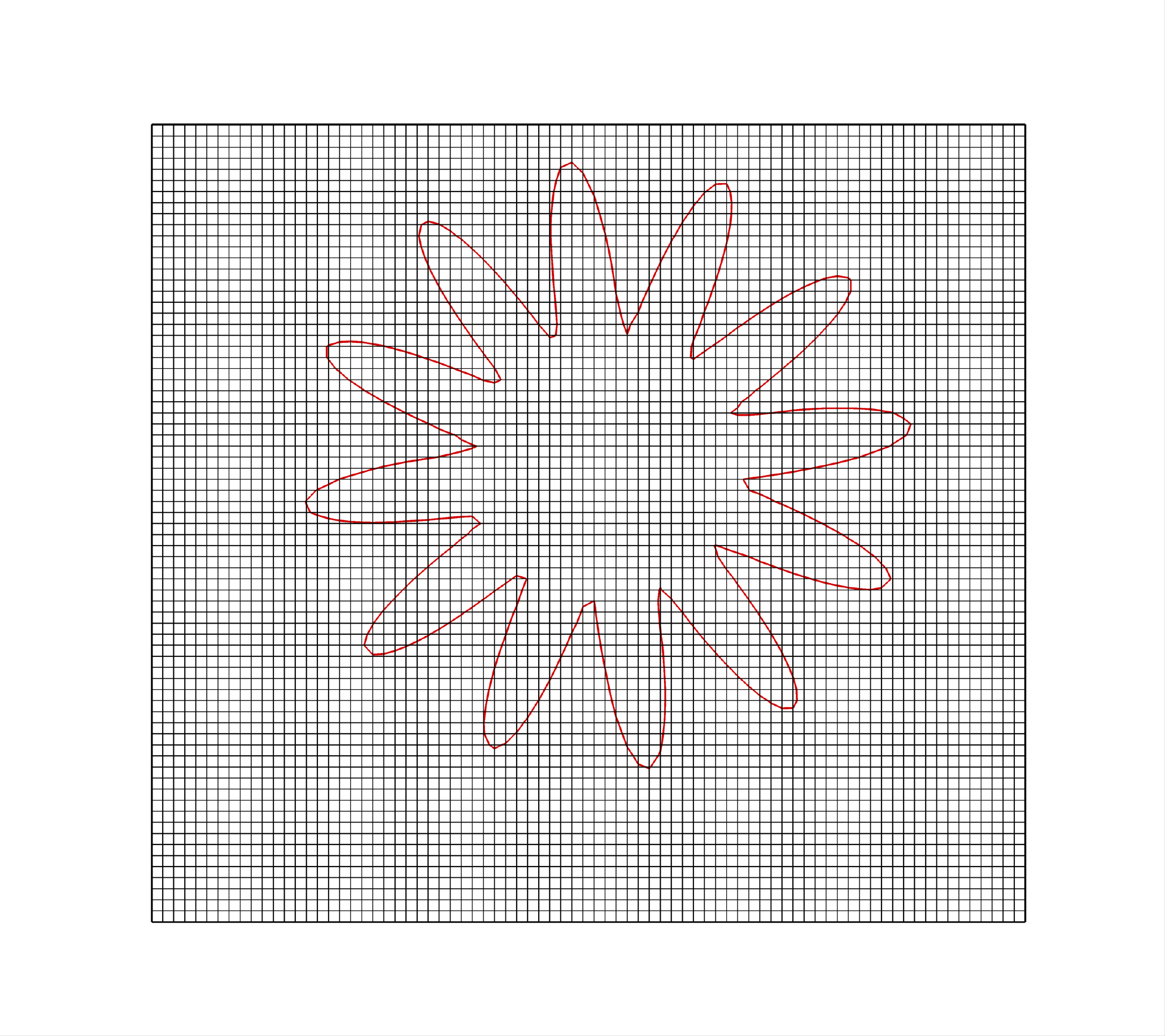}
  \begin{center}
  \end{center}
\endminipage
\caption{{\sl  The star shaped five, nine and twelve petal interface of Test Case 3 on a grid of size $80 \times 80$.}}
\label{star}
\end{figure}
The interface is a star shaped closed curve as shown in figure \ref{star} and can be described by the level set function
 $\phi\left( r, \theta \right)$  = $r- r_{0}-0.2 \sin(w \theta) $, where $r$=  $\displaystyle{\sqrt{(x-x_{c})^{2}+(y-y_{c})^{2}}}$,
  $\displaystyle{\theta= arctan((y-y_{c})/(x-x_{c}))}$, 
  and $x_c$= $y_{c}$= $0.2/\sqrt{20}$. We study this problem for $r_{0}$= $0.5$ and $w$= $5$, $9$ and $12$. The boundary conditions  are derived from analytical solution. 
\begin{figure}[!h]
\begin{minipage}[b]{.45\linewidth}  
\includegraphics[scale=0.45]{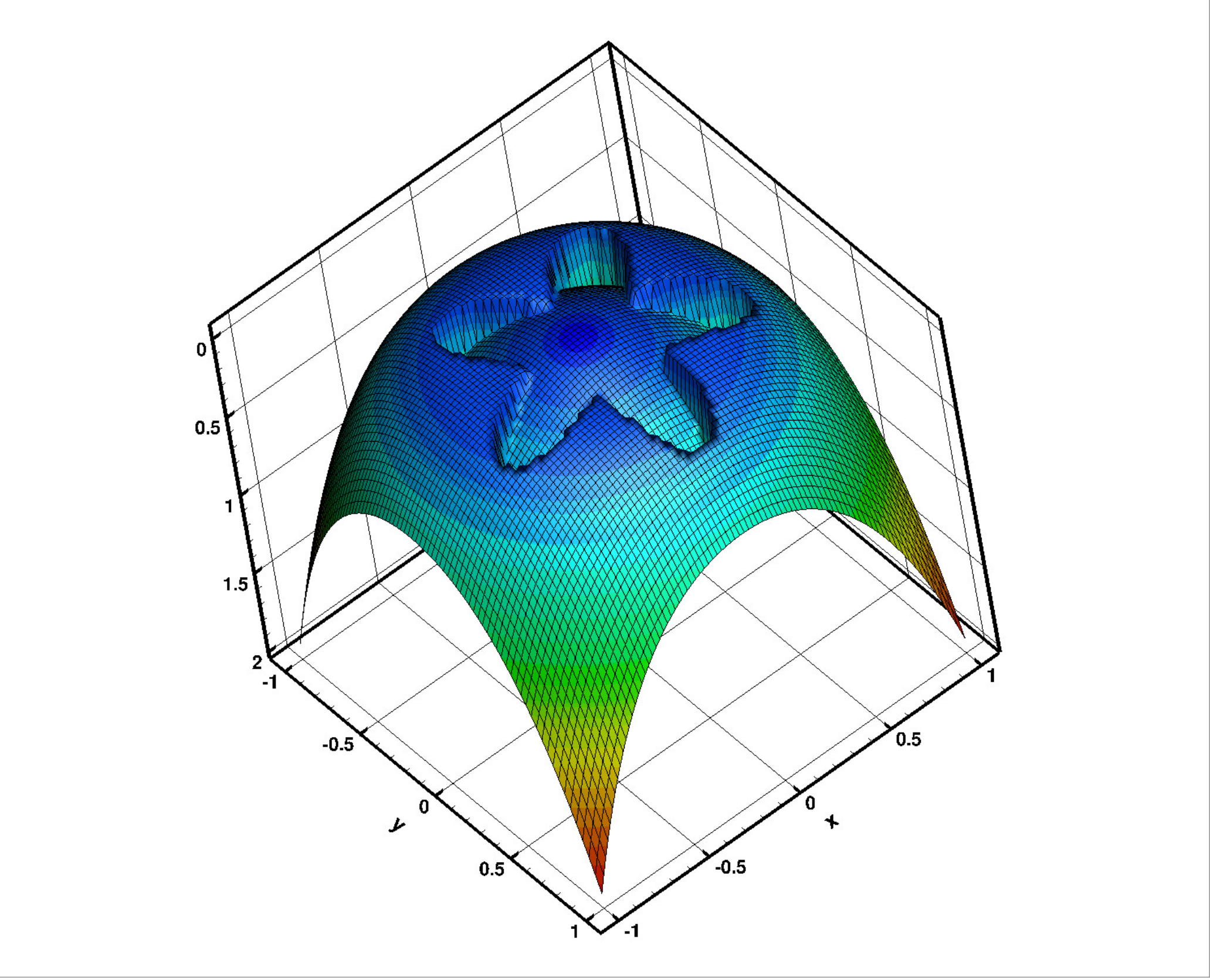} 
\end{minipage}           
\begin{minipage}[b]{.45\linewidth}
\includegraphics[scale=0.45]{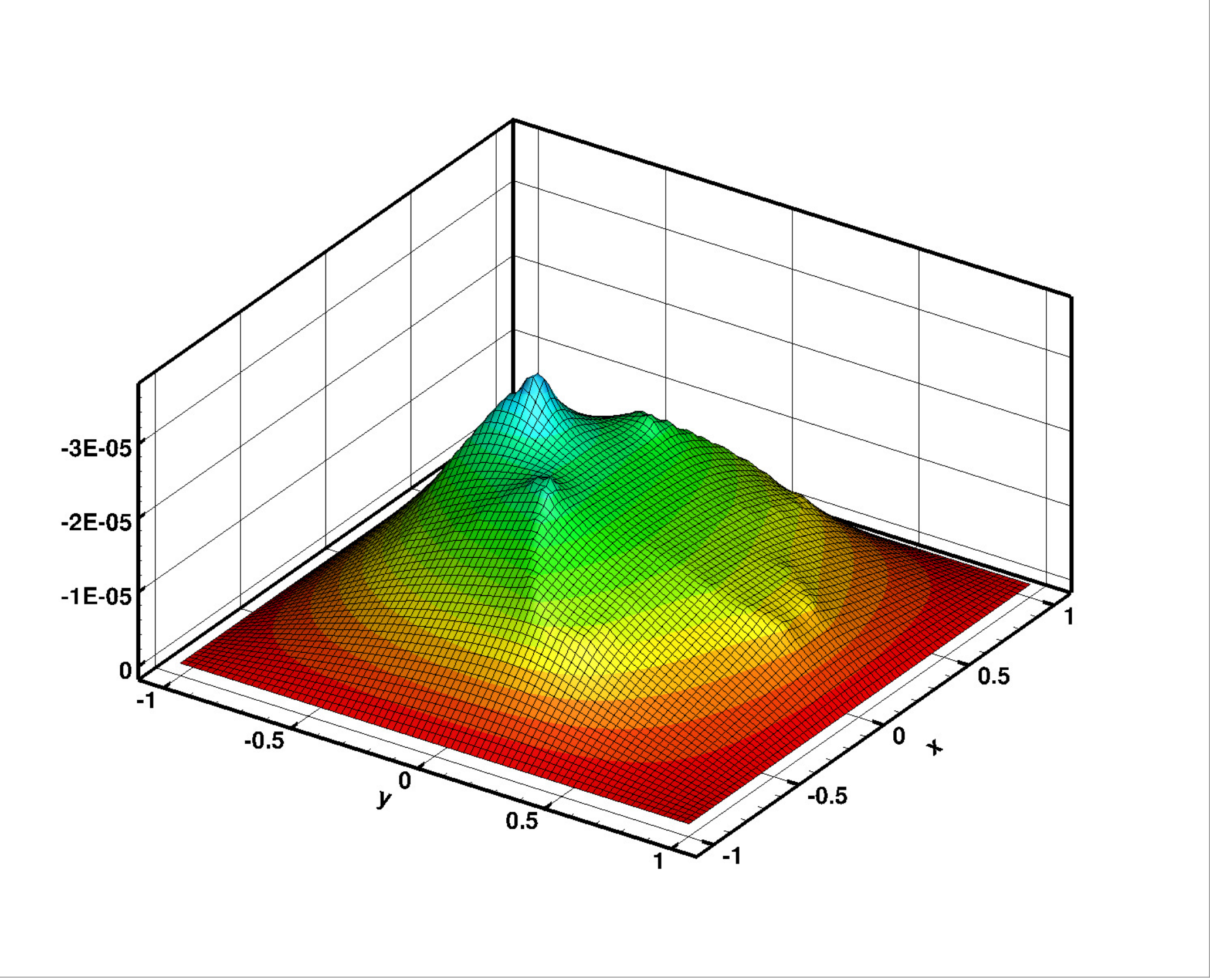} 
\end{minipage} 
\caption{{\sl Surface plots of the numerical solution (left) and the error (right) for Test Case 3 with $\beta^{+}=2$ on grid size $80 \times 80.$ } }
\label{star1}
\end{figure}
\begin{table}[!h]
\caption{ Grid refinement analysis of maximum error for Test Case $3$  with $\beta^{+} =2$, $\beta^{-}= 1$,  $S_{0}=-0.1$, $S_{1}=0$, $r_{0}=0.5$ and $w=5$.}
\begin{center}
\begin{tabular}{|c|c|c|c|c|}  \hline
N &	Present  & ROC & FIIM \cite{lifiim} & ROC \\ \hline
 $40 $   &  $1.82 \times 10^{-4}$&$-$   & $2.28 \times 10^{-3}$& $-$ \\
 $80 $   &  $3.01 \times 10^{-5}$&$2.59$& $5.22 \times 10^{-4}$& $2.12$ \\
$160 $  &  $2.87 \times 10^{-6}$&$3.38$& $1.26 \times 10^{-4}$& $2.05$ \\
 $320$ &  $4.22 \times 10^{-7}$&$2.76$& $2.98 \times 10^{-5}$& $2.08$ \\ \hline
 \end{tabular}
\end{center}
\label{table_5}
\end{table}
\begin{table}[!h]
\caption{ Grid refinement analysis of maximum error for Test Case $3$  with $\beta^{+} =10000$, $\beta^{-}= 1$,  $S_{0}=-0.1$, $S_{1}=0$, $r_{0}=0.5$ and $w=5$.}
\begin{center}
\begin{tabular}{|c|c|c|}  \hline
N &	Present  &  FIIM \cite{lifiim}  \\ \hline
 $40 $   &  $3.64 \times 10^{-7}$& $6.55 \times 10^{-5}$ \\
 $80 $   &  $5.57 \times 10^{-8}$& $7.84 \times 10^{-6}$ \\
$160 $  &  $2.16 \times 10^{-8}$& $5.98 \times 10^{-7}$ \\
 $320$ &  $2.51 \times 10^{-9}$& $5.85 \times 10^{-7}$ \\ \hline
 \end{tabular}
\end{center}
\label{table_6}
\end{table}
\begin{table}[!h]
\caption{Grid refinement analysis of maximum error for Test Case $3$  with $\beta^{+} =10$, $\beta^{-}= 1$,  $S_{0}=-0.1$, $S_{1}=0$, $r_{0}=0.5$ and $w=5$.}
\begin{center}
\begin{tabular}{|c|c|c|c|c|c|c|}  \hline
N &	Present  & ROC &  \cite{mittal2018solving} & ROC & \cite{liu2000boundary} & ROC  \\ \hline
 $40 $   & $3.62 \times 10^{-5}$& $-$   &  $7.20 \times 10^{-5}$	&$-$ &  $1.67 \times 10^{-4}$ & $-$ \\
 $80 $   & $6.01 \times 10^{-6}$& $2.59$& $1.75 \times 10^{-5}$ 	&$2.04$&  $7.35 \times 10^{-5}$& $1.18$\\
$160 $  & $5.91 \times 10^{-7}$& $3.34$& $4.21 \times 10^{-6}$ 	&$2.51$&  $-$ & $-$\\
 $320$ & $1.08 \times 10^{-7}$& $2.45$& $8.35 \times 10^{-7}$ 	&$2.33$&  $-$ & $-$\\ \hline
 \end{tabular}
\end{center}
\label{table_7}
\end{table}
\begin{figure}[!h]
\begin{minipage}[b]{.45\linewidth}  
\includegraphics[scale=0.40]{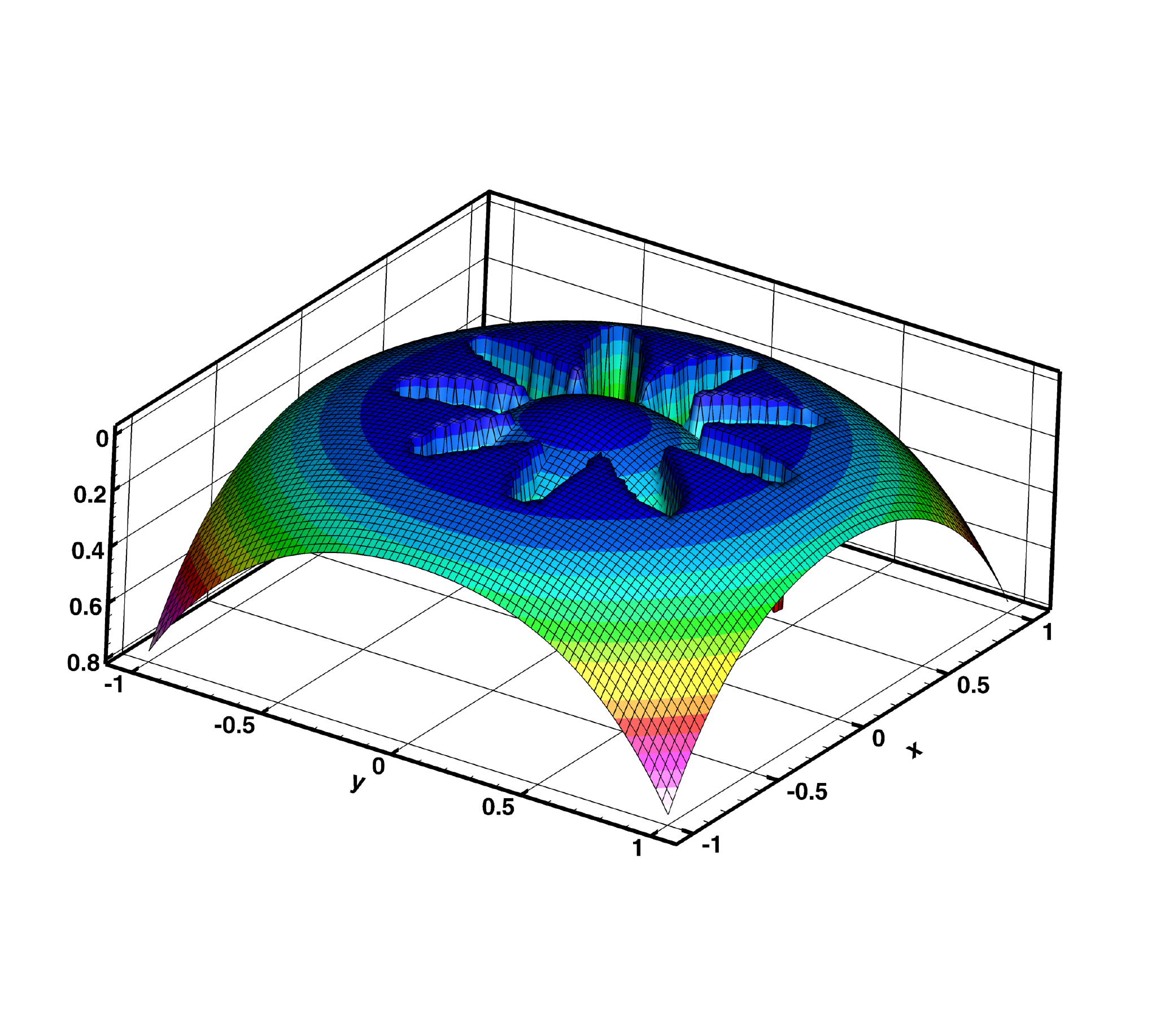} 
\end{minipage}           
\begin{minipage}[b]{.45\linewidth}
\includegraphics[scale=0.40]{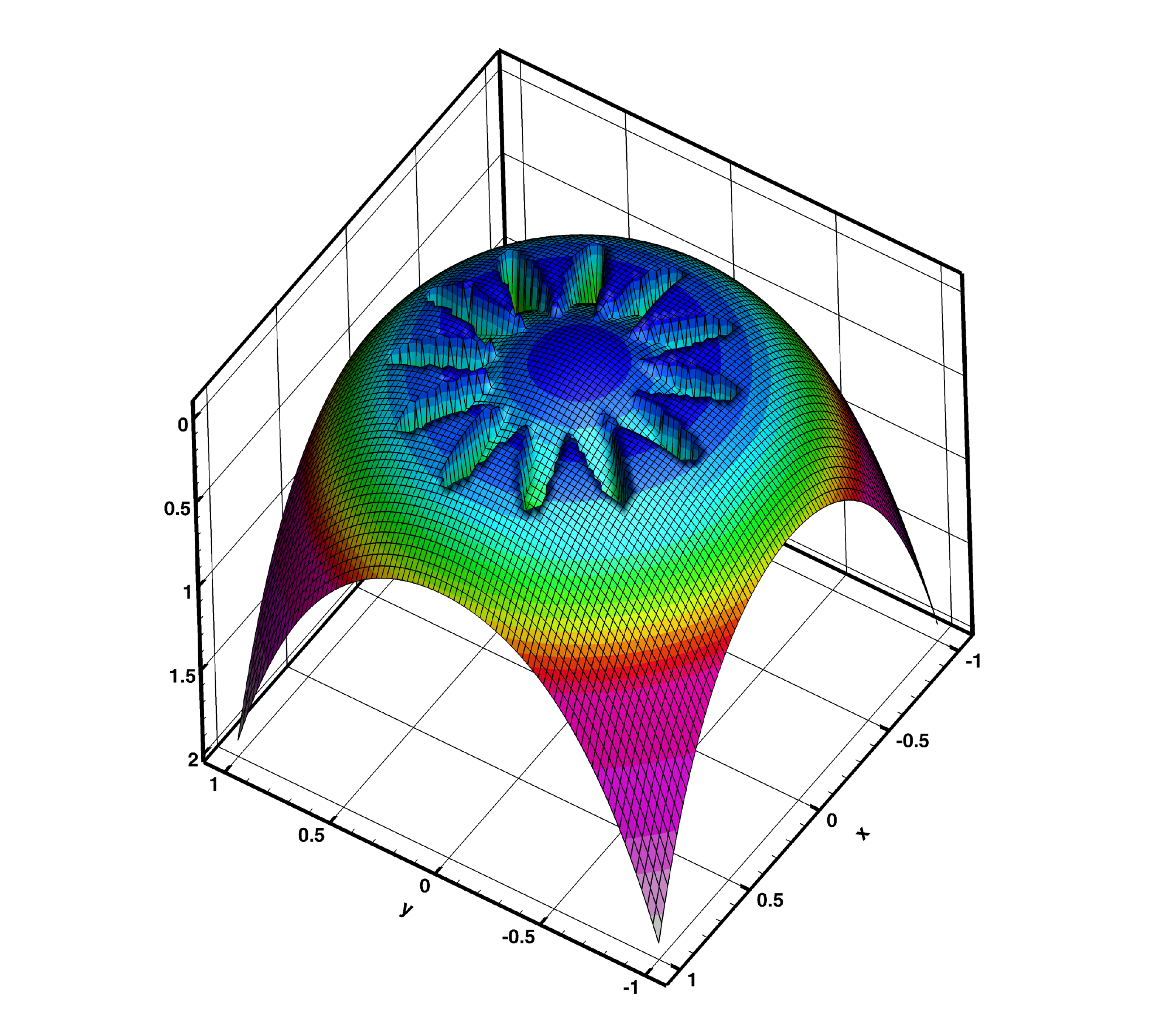} 
\end{minipage} 
\caption{{\sl  Numerical solution on grid size $80 \times 80$ for Test Case $3$ for the combinations $w=9$, $\beta^+=5$ and $w=12$, $\beta^+=2$.} }
\label{star912}
\end{figure}
\begin{figure}[!h]
\begin{minipage}[b]{.45\linewidth}  
\includegraphics[scale=0.4]{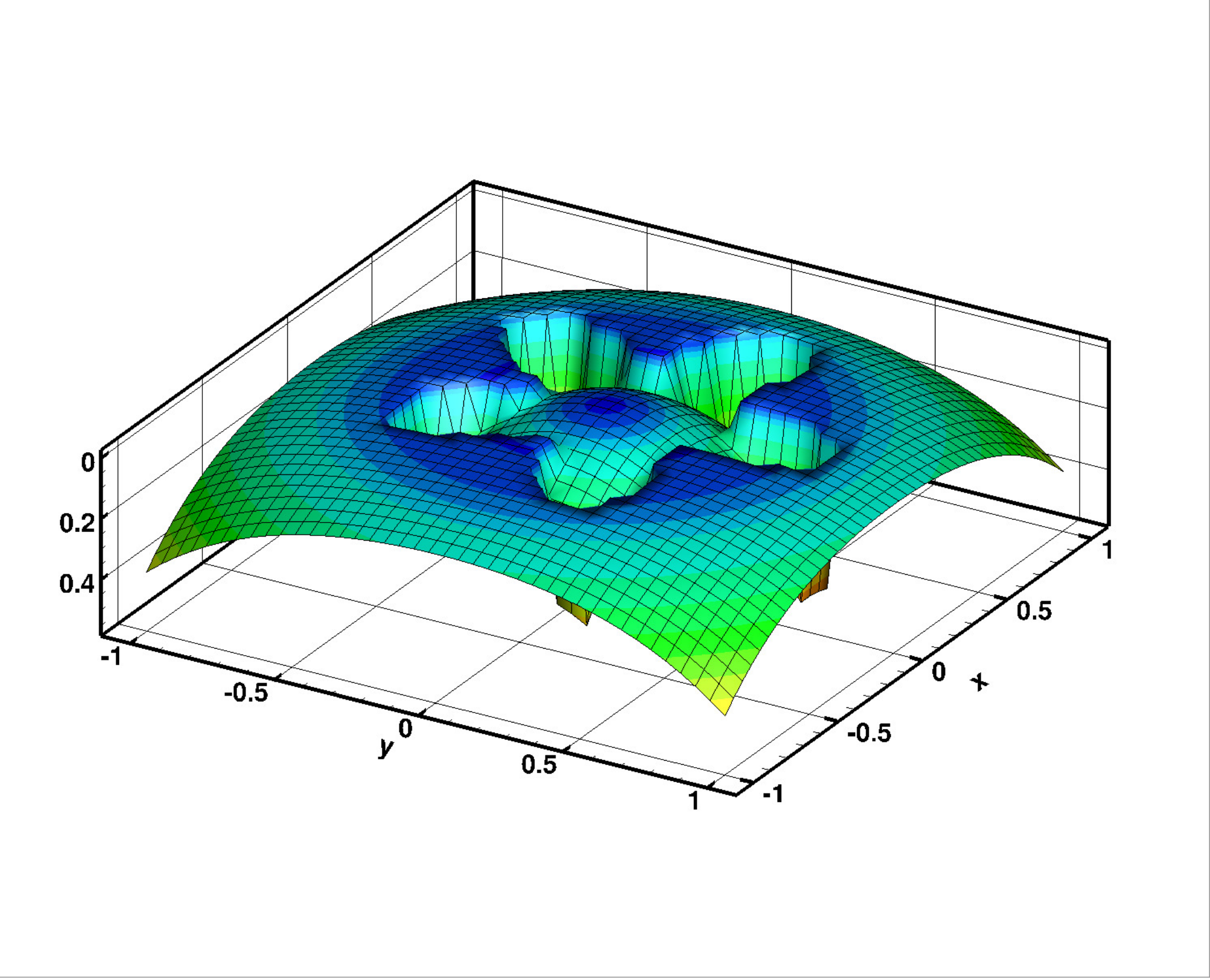} 
\centering (a)
\end{minipage}            
\begin{minipage}[b]{.45\linewidth}
\includegraphics[scale=0.4]{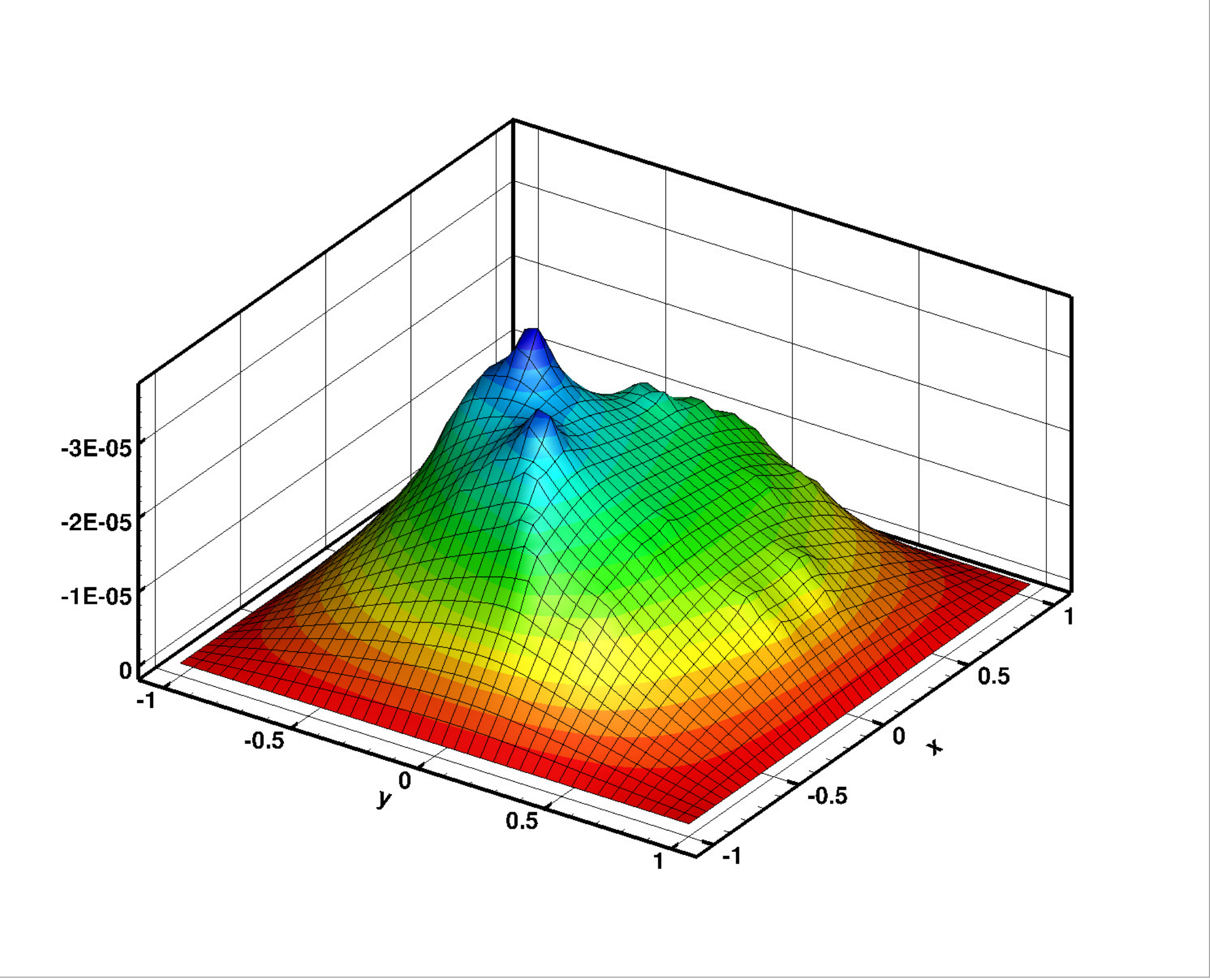} 
 \centering (b)
\end{minipage} 
\begin{minipage}[b]{.45\linewidth}   
\includegraphics[scale=0.4]{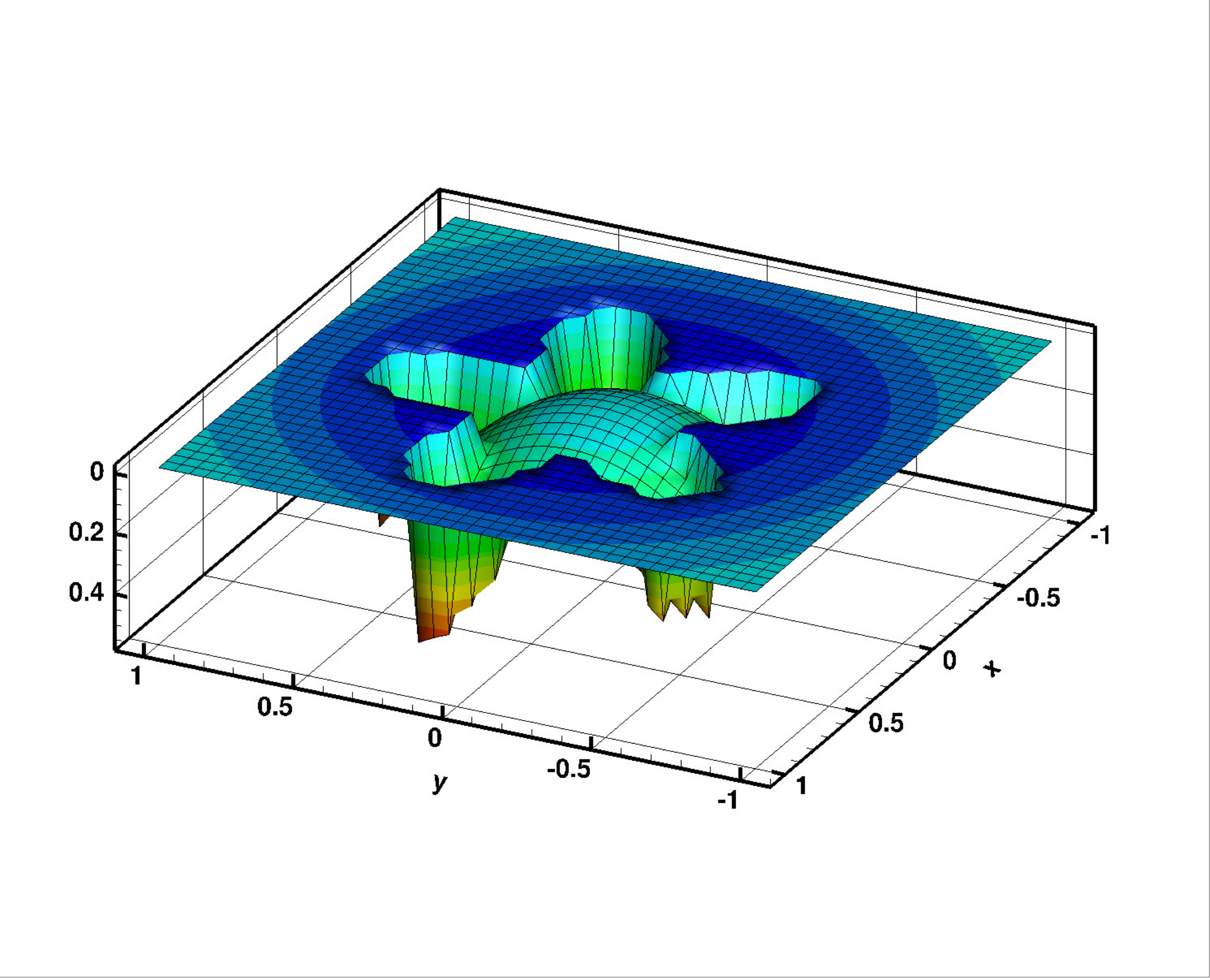} 
 \centering (c)
\end{minipage}           
\begin{minipage}[b]{.45\linewidth}
\includegraphics[scale=0.4]{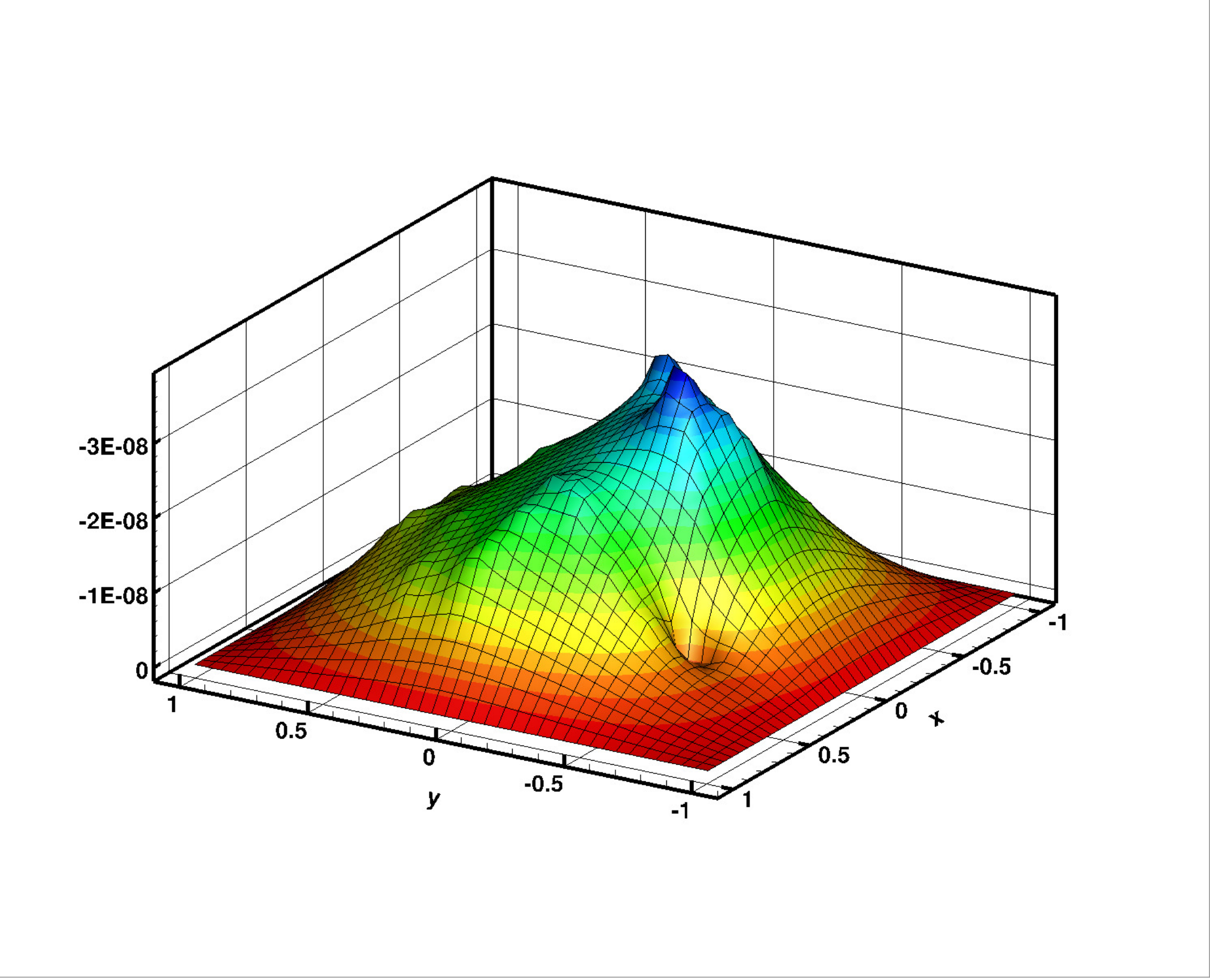} 
 \centering (d)
\end{minipage} 
\caption{{\sl Surface plots of the numerical solution (left) and the error (right) for Test Case 3 with (a)-(b) $\beta^{+}=10$ and (c)-(d) $\beta^{+}=10000$ on grid size $40 \times 40.$ } }
\label{star3}
\end{figure}

We compute the solution for different values of $\beta^{+}=2,\; 5,\; 10 \;{\rm and}\; 10000$ with $\beta^{-}=1$ such that the robustness of the scheme can be tested for cases with low as well as extremely high jump in the diffusion coefficients $\beta$. Our computed solutions are shown in figures \ref{star1}-\ref{star3} along with the surface error plots for the three combinations of $\beta^+$ and $\beta^-$ mentioned above. The effect of the jump in $\beta$ can be clearly observed from these figures. We also compare the maximum error norm on different grid sizes with $N=40, 80, 160, 320$ for these combinations with those of \cite{fedkiw2002ghost,lifiim,mittal2018solving} in tables \ref{table_5}-\ref{table_7}. Once again, one can clearly see that our scheme fare much better than them not only for small jump in $\beta$ but also for high jumps as well. The effect of the parameter $w$ determining the number of petals can also be seen in figure \ref{star912}.
\subsection{Test case 4}
This is an example of a composite material problem with piecewise constant coefficients. This problem is of specific interest in checking the effectiveness of newly developed numerical schemes because of the challenge posed by large differences in material properties. Let 
\begin{equation}
u(x,y)=\left\{\begin{array}{cc}
&\frac{2x} { \rho+1+s^{2}(\rho-1)}, \;\;\;\;\;\;\;\;\;\;\;\;\;\;\;\;\;
\;\;\;\;\;\;\;\;\;\;\;\;\;\;\;\;\;\;\;\;\;\;\;\;\;\phi\leq 0\\ \vspace{0.2cm}
&\frac{x(\rho+1)-s^{2}(\rho-1) x/(x^{2}+y^{2})}{\rho+1+s^{2}(\rho-1)}\;\;\;\;\;\;\;\;\;\;\;\;\;\;\;\;\;\;\;\;\; \;\;\;\;\;\;\;\;
\phi > 0
\end{array}\right.
\end{equation}
where $s$=$0.5$, the radius of the same circular interface as described in Test Case 1, and $\rho$=$\beta^{-}/ \beta^{+}$. The above is the solution to the Laplace equation $\nabla^2=0$ with $[u]=0$ and $[\beta u_{n}]=0$ at the interface and exterior boundary as given in the analytical solution. \\
\begin{figure}[!h]
\begin{minipage}[b]{.45\linewidth}  
\includegraphics[scale=0.40]{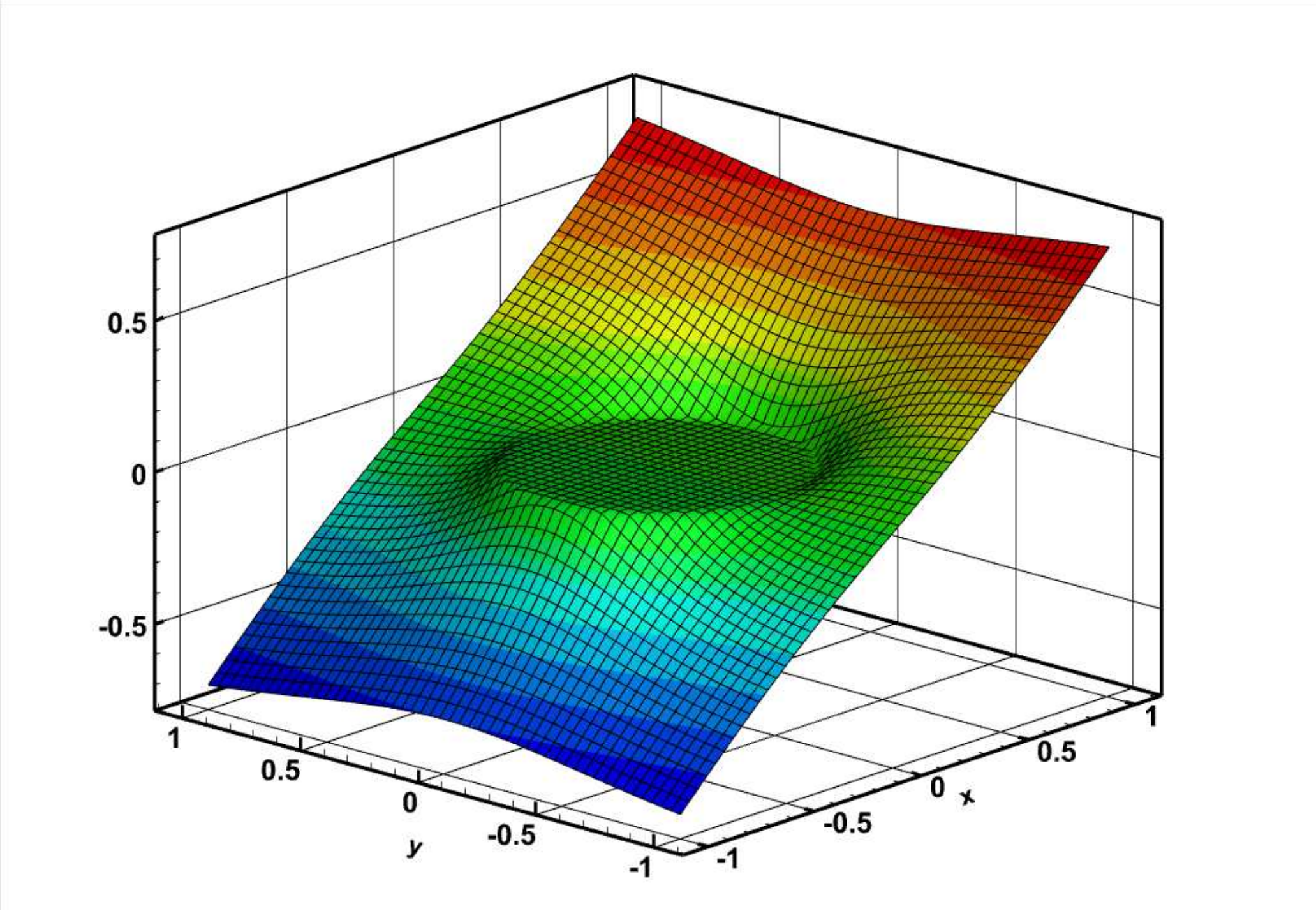} 
\end{minipage}          
\begin{minipage}[b]{.45\linewidth}
\includegraphics[scale=0.40]{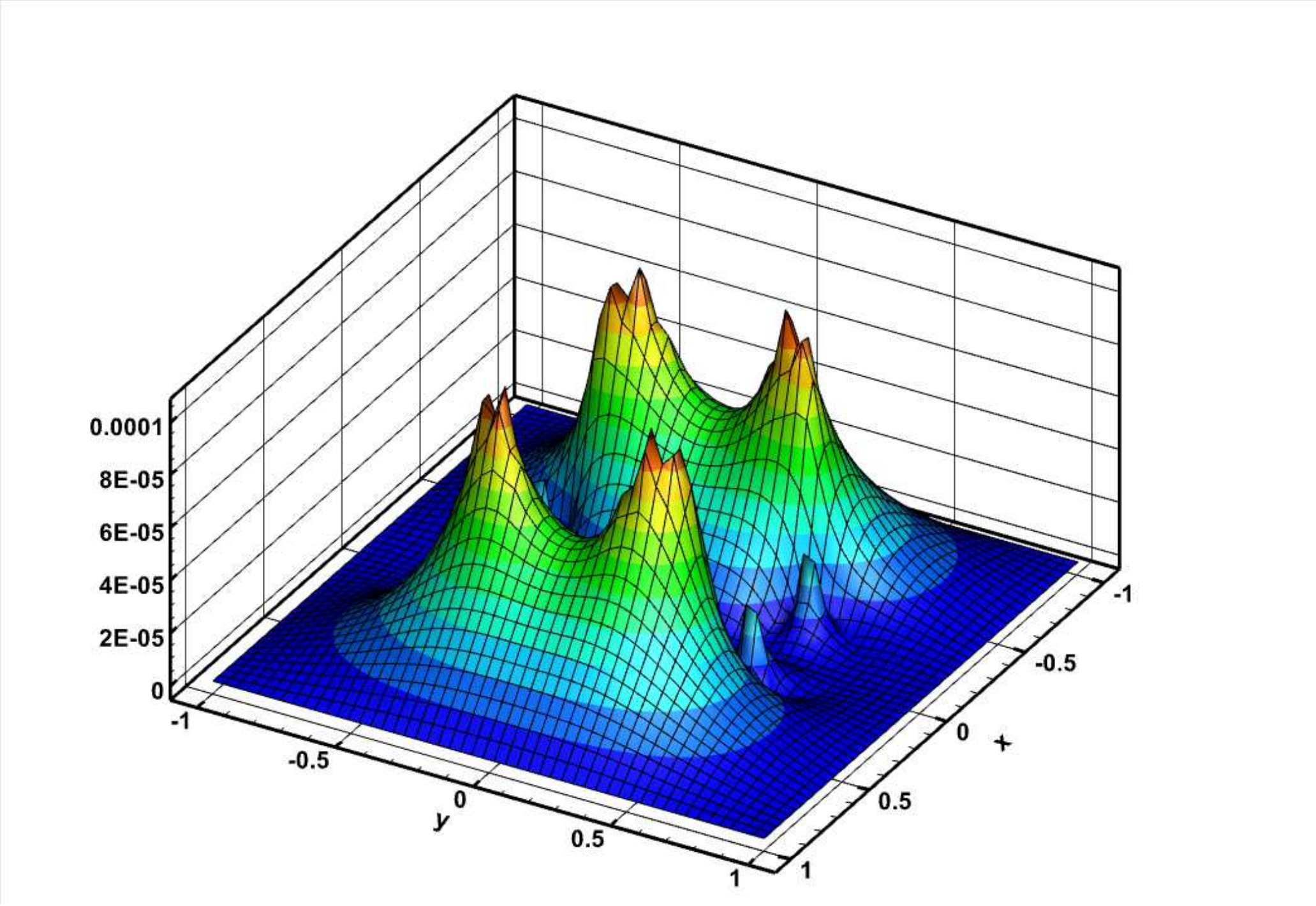} 
\end{minipage} 
\caption{{\sl  Surface plots of the numerical solution and the error for Test Case $4$  for $\rho=5000$ on grid size $50 \times 50$.} }
\label{tcase4a}
\end{figure}
\begin{table}[!h]
\caption{ Grid refinement analysis of maximum error for composite material Test Case $4$ with $\rho=5000$.}
\begin{center}
\begin{tabular}{|c|c|c|c|c|c|c|c|c|}  \hline
N &	Present  &  ROC & DIIM \cite{berthelsen2004decomposed} & ROC & FIIM  \cite{lifiim} & ROC & Interior EJIIM \cite{wiegmann2000explicit} & ROC  \\ \hline
 25   & $1.55 \times 10^{-3}$  & $-$ & $9.80 \times 10^{-4}$ & $-$ & $1.2 \times 10^{-2}$  &  $-$   & $1.4 \times 10^{-3}$ & $-$ \\
 50   & $1.03 \times 10^{-4}$ & $3.91$ & $2.73 \times 10^{-4}$ & $1.85$ & $9.2 \times 10^{-2}$   &  $-$  & $3.5 \times 10^{-4}$ & $2.0$\\
 100  & $1.44 \times 10^{-5}$  & $2.83$ & $4.84 \times 10^{-5}$ & $2.50$  & $5.9 \times 10^{-2}$ & $0.6$   & $9.0 \times 10^{-5}$ & $2.0$\\
 200  & $1.96 \times 10^{-6}$  & $2.87$ & $1.26 \times 10^{-5}$ & $1.94$ & $7.70 \times 10^{-3}$ & $2.9$ & $2.2 \times 10^{-5}$ & $2.0$\\ 
 400  & $2.57 \times 10^{-7}$  & $2.93$ & $3.49 \times 10^{-6}$ & $1.85$ & $-$ & $-$ & $-$ & $-$ \\ \hline
 \end{tabular}
\end{center}
\label{table_8}
\end{table}
We tabulate the maximum errors resulting from our computation corresponding to $\rho=5000$ and $\rho=1/5000$ in tables \ref{table_8} and \ref{table_9} respectively. We further compare our grid refinement studies with those of \cite{berthelsen2004decomposed,lifiim,mittal2016class,wiegmann2000explicit}. From the tables one can clearly see the errors resulting from our computation decaying at a rate close to three, which is extremely close to the best convergence rate accomplished for this problem by other methods. In figures \ref{tcase4a} and \ref{tcase4b}, we present the surface plots our computed solutions and errors corresponding to $\rho=5000$ and $\rho=1/5000$ respectively . Once again one can see excellent resolution of the sharp interface and relative smoothness of the error near the interface on a grid as coarse as $50 \times 50$.
\begin{figure}[!h]
\begin{minipage}[b]{.45\linewidth}  
\includegraphics[scale=0.40]{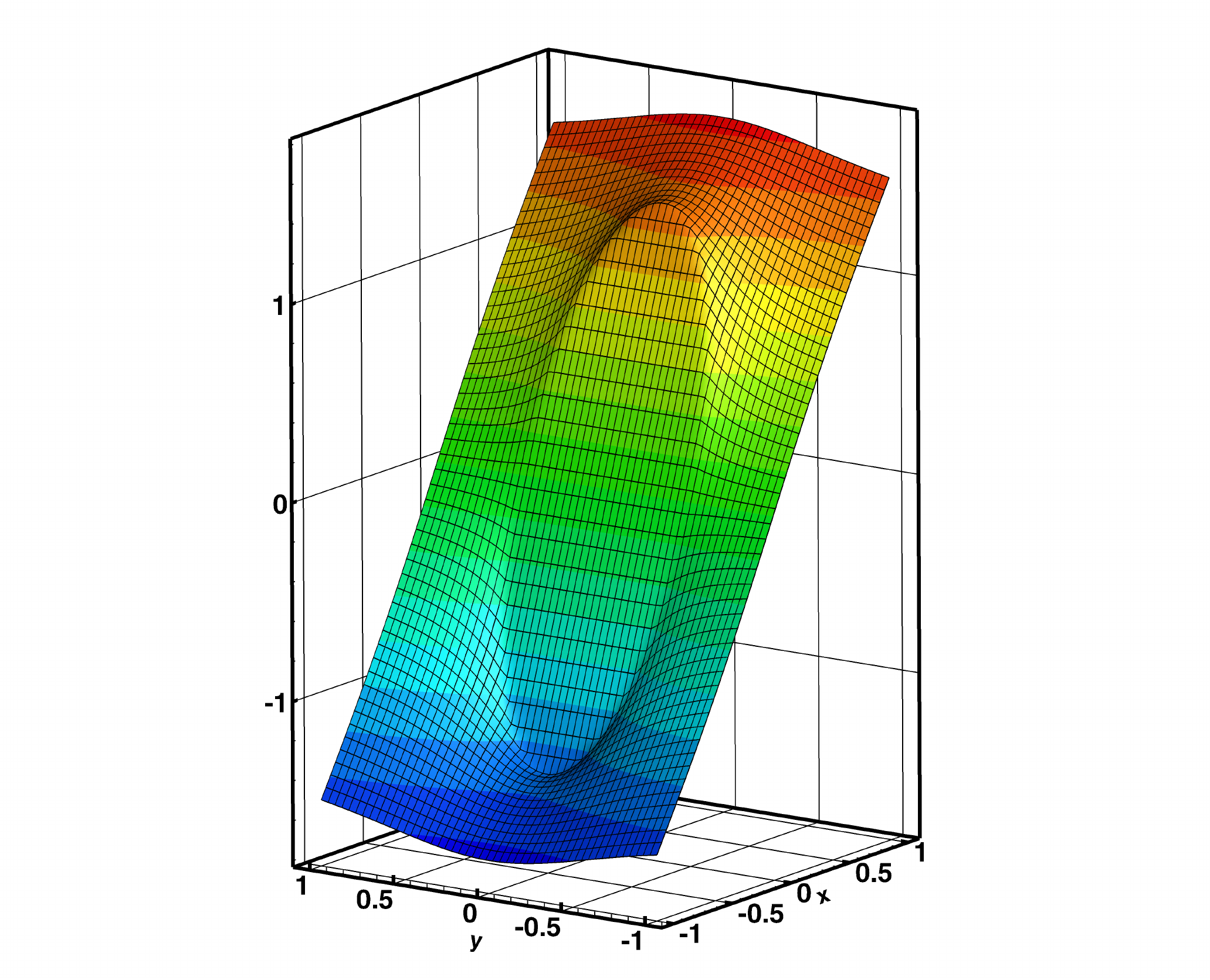} 
\end{minipage}           
\begin{minipage}[b]{.45\linewidth}
\includegraphics[scale=0.40]{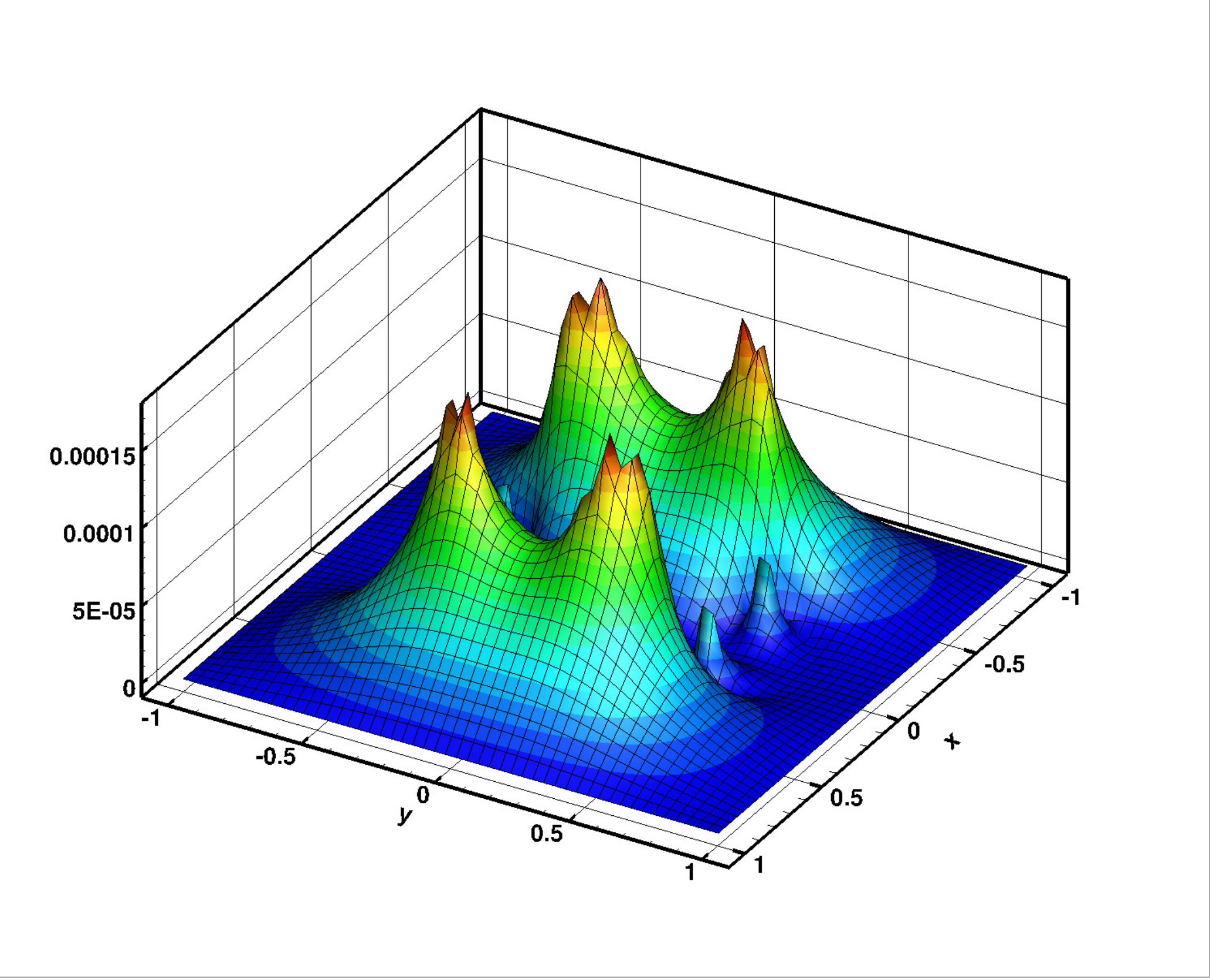} 
\end{minipage} 
\caption{{\sl Surface plots of the numerical solution and the error for Test Case 4 for $\rho=1/5000$ on grid size $50 \times 50$.} }
\label{tcase4b}
\end{figure}
\begin{table}[!h]
\caption{Grid refinement analysis of maximum error for composite material Test Case $4$ with $\rho=1/5000$. }
\begin{center}
\begin{tabular}{|c|c|c|c|c|c|c|c|c|}  \hline
N &	Present  &  ROC & DIIM \cite{berthelsen2004decomposed} & ROC & FIIM \cite{lifiim} & ROC & Interior EJIIM  \cite{wiegmann2000explicit} & ROC \\ \hline
 25   & $3.09 \times 10^{-3}$  & $-$ & $1.63 \times 10^{-3}$ & $-$ & $5.2 \times 10^{-3}$  &  $-$   & $1.9 \times 10^{-3}$ & $-$\\
 50   & $1.72 \times 10^{-4}$ & $4.16$ & $4.55 \times 10^{-4}$ & $1.85$ & $1.6 \times 10^{-3}$   &  $1.7$  & $5.5 \times 10^{-4}$ & $1.8$\\
 100  & $2.40 \times 10^{-5}$  & $2.84$ & $8.06 \times 10^{-5}$ & $2.50$  & $2.3 \times 10^{-4}$ & $2.8$   & $1.3 \times 10^{-4}$ & $2.1$\\
 200  & $3.27 \times 10^{-6}$  & $2.87$ & $2.10 \times 10^{-5}$ & $1.94$ & $5.0 \times 10^{-5}$ & $2.2$ & $3.2 \times 10^{-5}$ & $2.0$\\ 
 400  & $4.23 \times 10^{-7}$  & $2.95$ & $5.82 \times 10^{-6}$ & $1.85$ & $-$ & $-$ & $-$ & $-$ \\ \hline
 \end{tabular}
\end{center}
\label{table_9}
\end{table}
\subsection{Test case 5: Flow Past a Circular Cylinder}
The next problems considered are ones, where the equations under consideration do not have analytical solutions. They are 2D steady-state flow past bluff bodies, which are governed by the 2D Navier-Stokes (N-S) equations  for incompressible viscous flow. We solve the N-S equations in streamfunction-vorticity ($\psi$-$\omega$) formulation. Here the vorticity $\omega$ is defined as $\displaystyle \omega=\frac{\partial v}{\partial x}-\frac{\partial u}{\partial y}$, where $u$ and $v$ are the horizontal and vertical components of the velocity of the fluid. The incompressibility condition facilitates defining the velocities in terms of streamfunction $\psi$ as 
\begin{equation}
u=\frac{\partial \psi}{\partial y} \quad {\rm and} \quad
v=-\frac{\partial \psi}{\partial x}. \label{e_uv}
\end{equation}
The vorticity transport equation is given by
\begin{equation}
			u \frac{\partial \omega}{\partial x}+v \frac{\partial \omega}{\partial y}=\frac{1}{Re}\left(\frac{\partial^2 \omega}{ \partial x^2}+\frac{\partial^2 \omega}{ \partial y^2}\right)\label{e_vt}
\end{equation}
From the definition of vorticity and streamfunction provided above, one can obtain the following Poisson equation for the streamfunction
\begin{equation}
			\frac{\partial^2 \psi}{ \partial x^2}+\frac{\partial^2 \psi}{ \partial y^2}=-\omega . \label{e_sf}
\end{equation}
Here all the variables $u$, $v$, $\psi$ and $\omega$ are in non-dimensional form and $Re$ is the Reynolds number representing the ratio of inertial and viscous forces acting on the fluid.	
\begin{figure}[!h] 
\begin{center}
 \includegraphics[height=3in,width=6in, keepaspectratio]{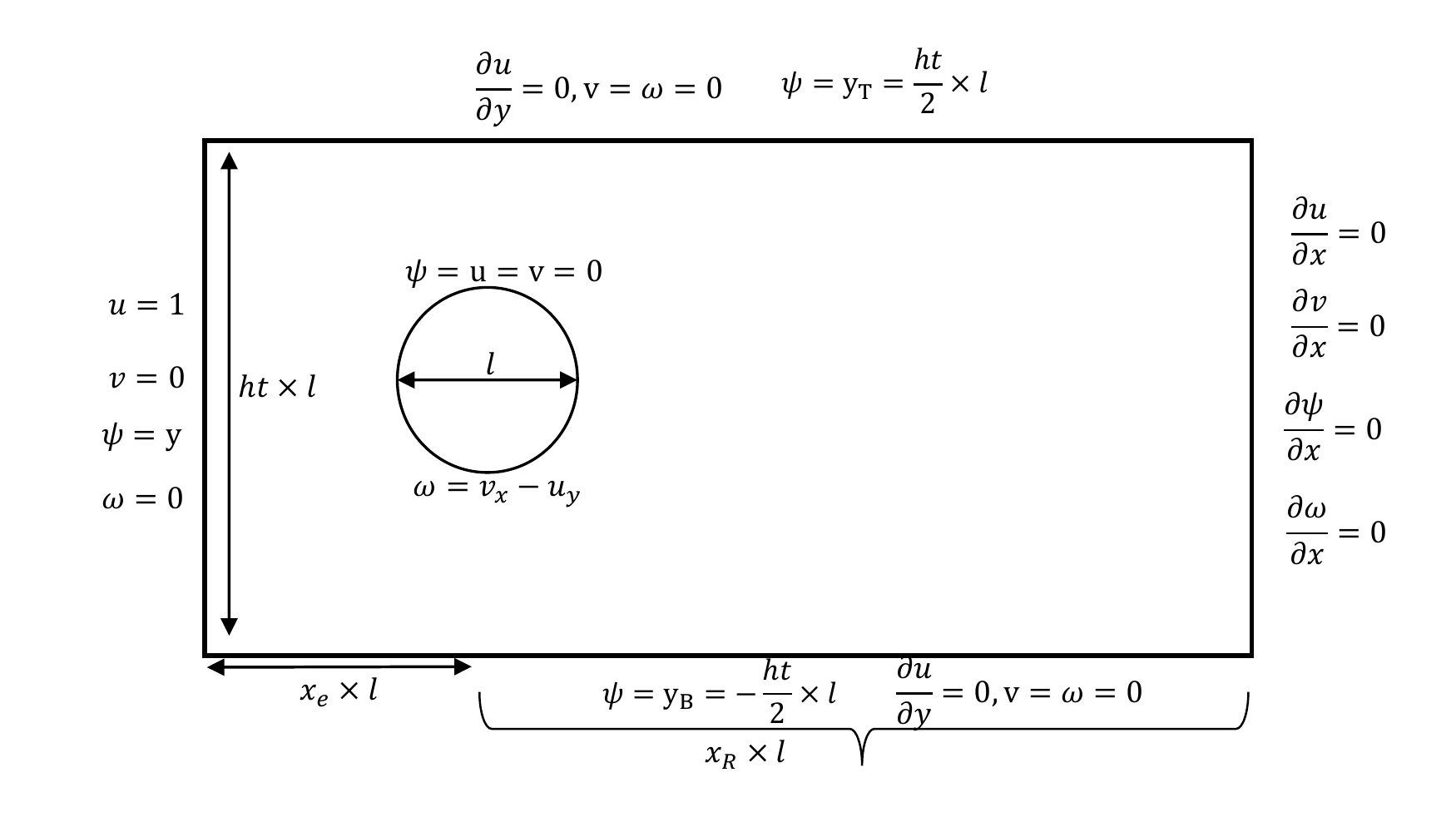} 
\caption{\sl{ Schematic for the flow past a circular cylinder and the boundary conditions.} }
\label{cyl_conf}
\end{center}
\end{figure}

The first problem considered under this category is the 2D steady-state flow past a circular cylinder in uniform stream \cite{calhoun, kalita2014effect, pankaj, le2006immersed, chen2020immersed, sharma2020steady, bailoor2019vortex, suzuki2021local, kalita2017alpha}. The schematic of the problem along with the boundary conditions used is presented in figure \ref{cyl_conf}. Here, Reynolds number is defined as $\displaystyle Re=\frac{U_{av}l}{\nu}$, where $U_{av}$ is the average inlet velocity, $l$ is the cylinder diameter and $\nu$ is the kinematic viscosity of the fluid. In all our computations $l$ was set as 1.0. Note that for the Reynolds numbers under consideration, the flow is always symmetric about the x-axis as the results would suggest.

As depicted in figure \ref{cyl_conf}(a), the computational domain is considered as $-x_el\leq x \leq x_Rl$,  $\displaystyle -\frac{ht}{2}l \leq y \leq \frac{ht}{2}l$, where $ht$ and $x_R$ are respectively the height and the length behind the cylinder of the computational domain being considered, and $x_e$ is the entrance length. The cylinder was placed at $(x,y)=(0,0)$ as its center. On the surface of the cylinder $u= v=\psi=0$; the same conditions were imposed
inside the cylinder as well during computation including that for the vorticity $\omega$.  At the inlet, uniform flow is considered as $u=1,\; v=0$ while at the outlet,
Neumann boundary conditions are prescribed as $\displaystyle \frac{\partial u}{\partial x}=0=\frac{\partial v}{\partial x}=\frac{\partial \psi}{\partial x}=\frac{\partial \omega}{\partial x}$. At the top and bottom $\displaystyle \frac{\partial u}{\partial y}=0$, $v=0$, and $\displaystyle \psi=\frac{ht \times l}{2}$ and $\displaystyle \psi=-\frac{ht \times l}{2}$ at the upper and lower boundaries respectively. For $\omega$, at the inlet, top and bottom, a potential flow condition is used, viz., $\omega=0$.

On the surface of the circular cylinder (or other bluff bodies that would be considered the next examples), which constitutes the interface, jump conditions for $\psi$ is straightforward \cite{xu2006immersed} and hence its discretization thereat. On the other hand, the approximation of the vorticity on the interface is a tricky one, which we have accomplished through a specific interpolation strategy by mapping the $\omega$ values at the irregular and regular points outside the cylinder onto the interface. One can see the schematic  for the jump correction of vorticity on the circular interface in figure \ref{cyl_setup}(a) for the first quadrant. Here the points $1$, $2$ and $3$, denoted by the red solid circles, correspond to the types of irregular points described in sections \ref{secx}, \ref{secy} and  \ref{secxy} respectively.  

For the point $1$, firstly a one sided $O(h^2)$ approximation \cite{pletcher2012computational} is used to compute $\omega$ by the discretizing $-\nabla^2\psi$, making use of the points to the right and top of this point represented by the triangles as shown in the figure \ref{cyl_setup}(a).  This is followed by the computation of $-\nabla^2\psi$ employing the regular five point central difference formula at the next point right to $1$. Making use of the approximations of  $-\nabla^2\psi$ thus found at these points, the value of  $\omega$ at the interfacial point to the left of $1$ is calculated by fitting a linear Lagrange polynomial. Likewise, the value of $\omega$ at the interfacial point below $2$ is calculated by fitting a linear Lagrange polynomial by making use of the values of $-\nabla^2\psi$ at the point $2$ and the next point above it. 

For point $3$, after computing $-\nabla^2\psi$ on it by the same strategy used for $1$ and $2$, we further compute it at the point $4$ diagonally above it by the five point central difference formula; it is then followed the computation of  $\omega$ at the point on the surface intercepted by the straight line joining the points $3$ and $4$ by fitting a linear Lagrange polynomial once again. Setting the $\omega$ values inside the cylinder as zero, the jump condition for vorticity on the interface is simply the interpolated values of $\omega$ thereat.
\begin{figure}[!h]
\begin{minipage}[b]{.35\linewidth}  
\includegraphics[height=3.0in,width=2.8in, keepaspectratio]{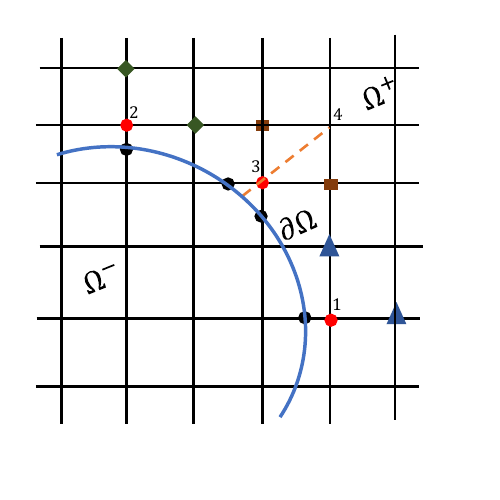} 
\centering{(a)}
\end{minipage}          
\begin{minipage}[b]{.58\linewidth} 
\includegraphics[height=6.2in,width=3.9in, keepaspectratio]{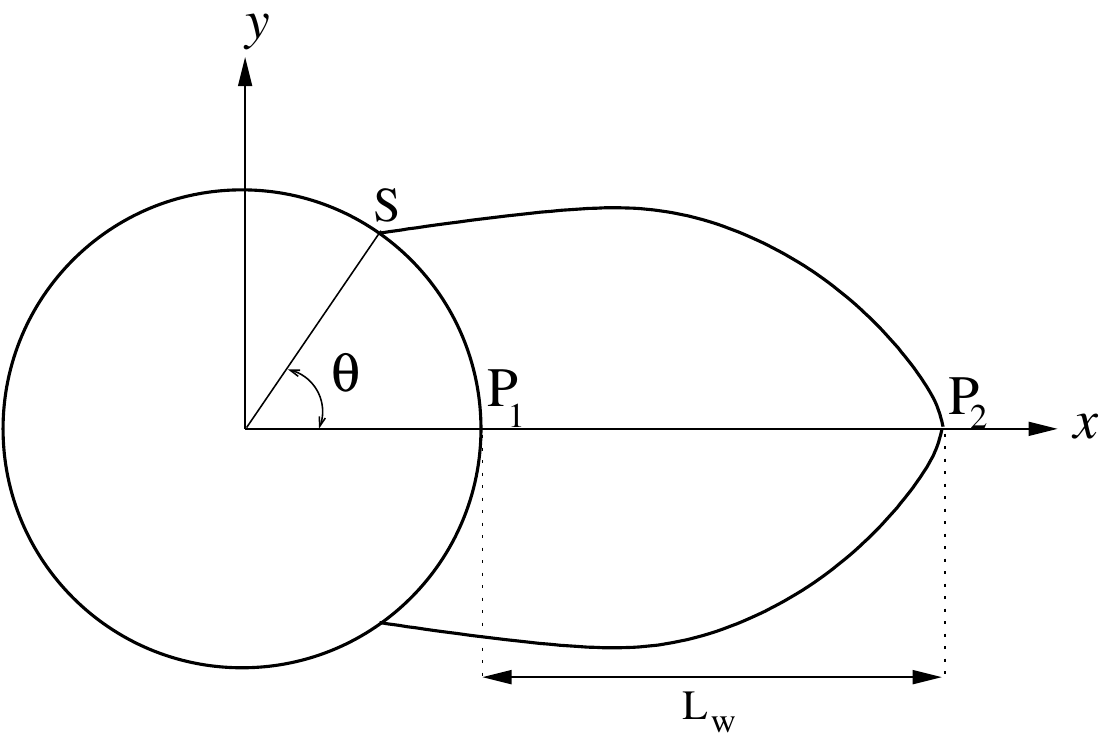} 
\centering{(b)}
\end{minipage}  
\caption{{\sl (a) Schematic for the jump correction of vorticity on the circular interface in the first quadrant and (b) flow parameters corresponding to table \ref{tab8} for the flow past an impulsively started circular cylinder: $P_1$ is the rear stagnation point, $P_2$, the wake stagnation point, $L_W$, the wake length, S, the separation point and $\theta$, the angle of separation.} }
\label{cyl_setup}
\end{figure}

We now discuss the solution procedure of the algebraic systems resulting from the discretization the system of equations \eqref{e_vt}-\eqref{e_sf}, which reduces to the form $A \phi={\bf b}$ with $\phi$ representing either $\psi$ or $\omega$ (see section \ref{algeb}). 

The computation of the steady-state solutions of fluid flow problems governed by coupled equations such as  \eqref{e_vt}-\eqref{e_sf} involves an outer-inner iteration procedure. After initializing $u$, $v$, $\psi$ and $\omega$ with appropriate boundary conditions and interior values (taken from the potential flow conditions), \eqref{e_sf} is solved for $\psi$. Once $\psi$ is computed, $u$ and $v$ are computed from equations \eqref{e_uv} by employing  a higher-order compact approximation as given by \cite{kalita2004transformation} after which $\omega$ is computed from \eqref{e_vt}. This constitutes one outer iteration. Making use of the updated values of $\omega$, $\psi$ is computed again. This process is repeated till maximum $\omega$-error reaches $5 \times 10^{-10}$. The inner iterations involve solving the matrix equations  at each outer iteration by iterative solvers. We have used biconjugate gradient stabilized method (BiCGStab) \cite{kelley1995iterative} with preconditioning, where Incomplete LU decomposition is used as a pre-conditioner. Preconditioning has been particularly useful for high Reynolds numbers on finest grids  where we have used the Lis library \cite{lis}. The inner iterations were stopped when the Euclidean norm of the residual vector ${\bf r}={\bf b}-A \phi$ arising out of equation  \eqref{e_vt}-\eqref{e_sf}  fell below $10^{-13}$ as in section \ref{algeb}. We have used a relaxation parameter $\lambda$ for both inner and outer iteration cycles. Larger the value of Reynolds number, smaller is the value of $\lambda$.  

We have computed solutions for $Re=10$, $20$ and $40$ on grid sizes ranging from $121 \times 61$ to $549 \times 499$. Once the data is available for the lowest $Re$ considered here, the flow is computed for the next Reynolds numbers by using the data from the previous $Re$ as the initial data. In figure \ref{conv_all}, we present the convergence history of the infinity norm of the $\psi$ and $\omega$ errors against the outer iterations for the Reynolds numbers considered here. Here the error is defined as the difference between the value at the current and the previous outer iteration levels. In all the cases, one can observe a very smooth convergence pattern.
\begin{figure}[!h]
\begin{minipage}[b]{.48\linewidth}  
\includegraphics[scale=0.40]{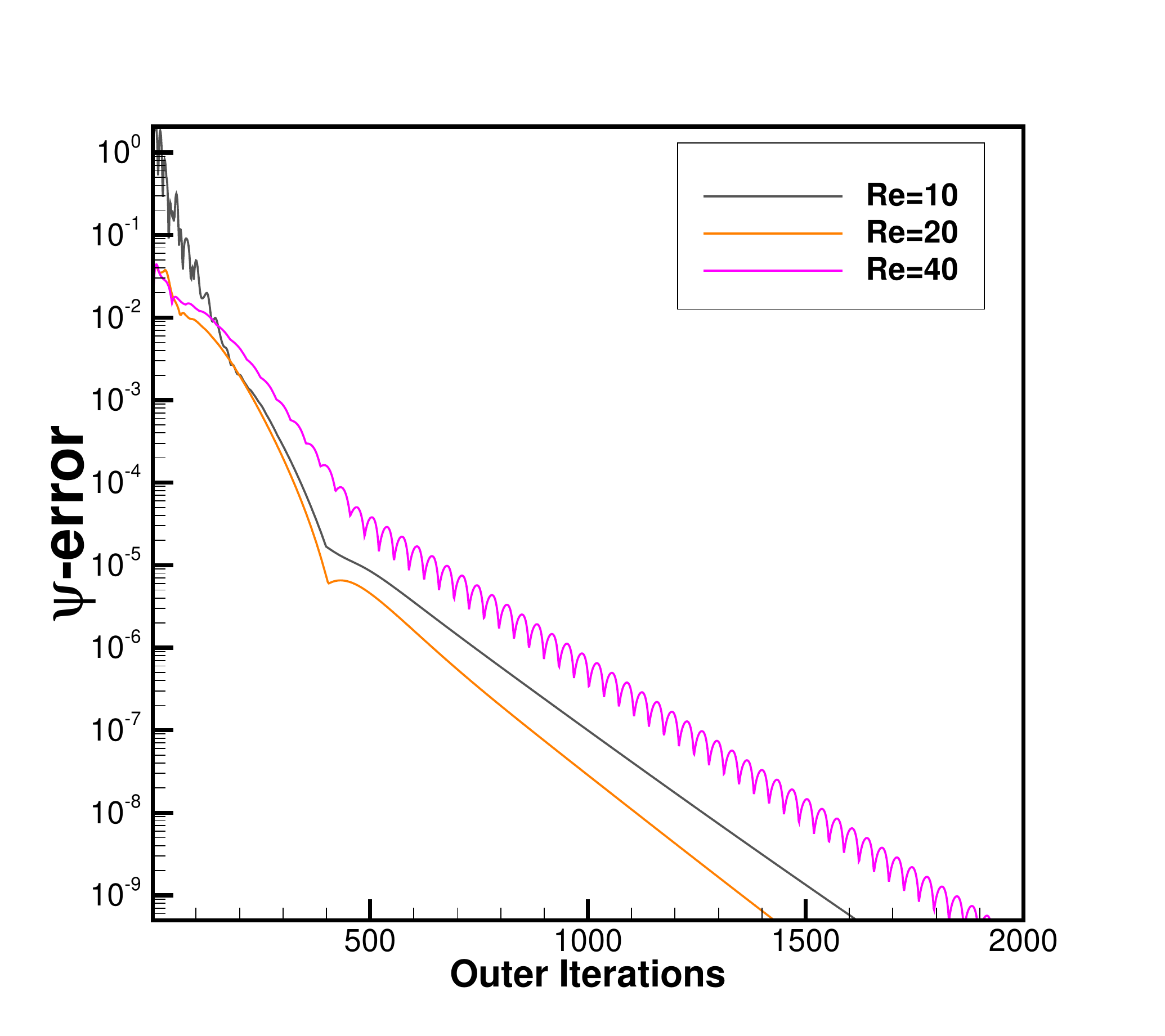} 
\centering{(a)}
\end{minipage}            
\begin{minipage}[b]{.48\linewidth} 
\includegraphics[scale=0.40]{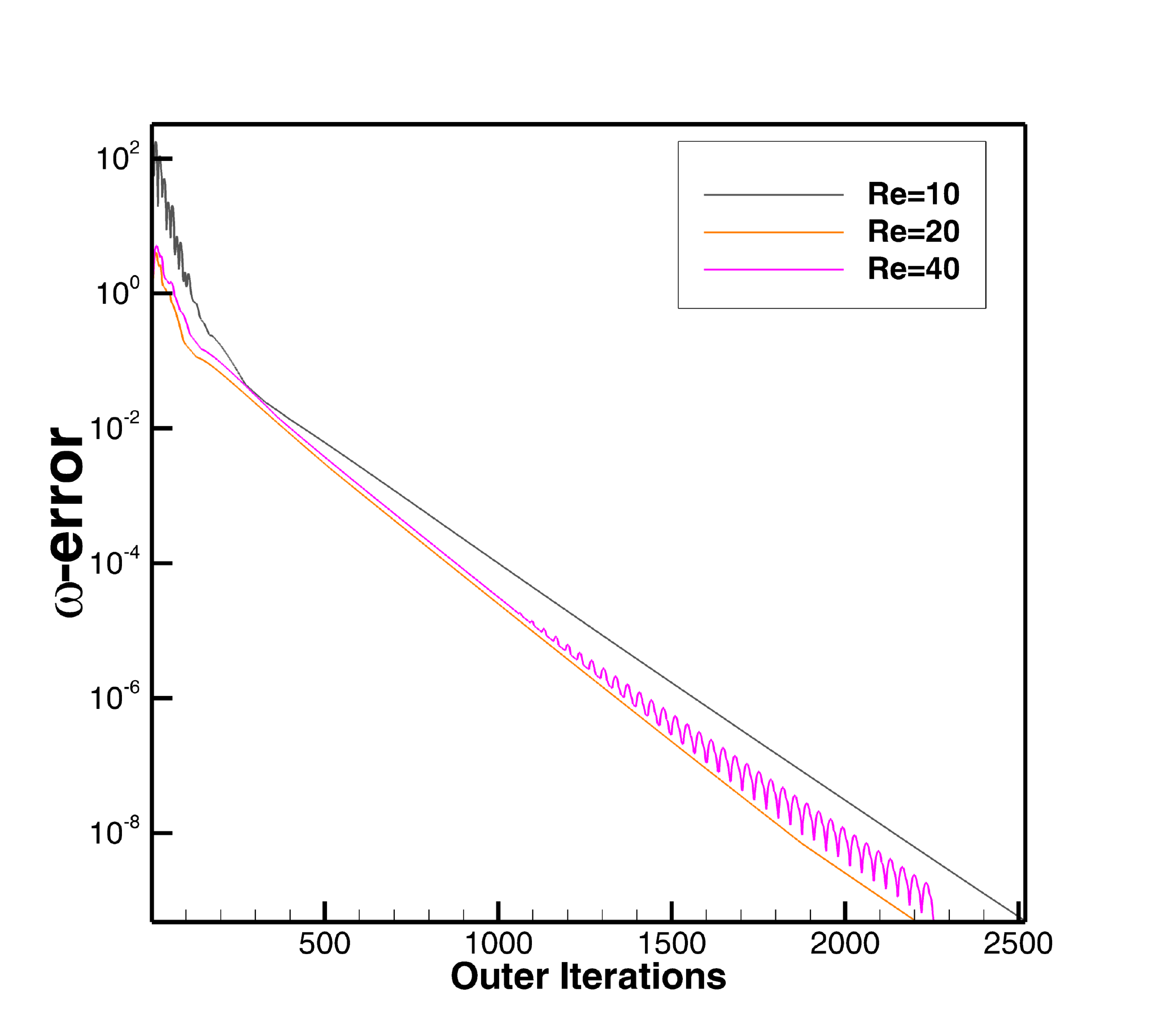} 
\centering{(b)}
\end{minipage}           
\caption{{\sl  Effect of Reynolds number on the convergence history on the finest grid ($549 \times 499$) for the flow past circular cylinder problem: (a) $\psi$-error and (b)  $\omega$-error.} }
\label{conv_all}
\end{figure}

In figure \ref{sf_vt}, we present our computed steady-state streamlines and vorticity contours for Reynolds numbers $10$, $20$ and $40$. One can see two symmetric vortices being formed behind the cylinder, whose size increase with the increase in Reynolds numbers. Once again, our simulations are very close to the well-known numerical results of \cite{calhoun,  kalita2014effect, pankaj, le2006immersed}.
\begin{table}[!h]
	\centering
	\caption{Comaprison of wake lengths, separation angles and drag coefficients for different Reynolds numbers..}
	\label{tab8}
		\begin{tabular}{c c c c c c c c c}
		\hline
		& Re &\cite{cout1} (exp) & \cite{dennis} & \cite{forn} & \cite{linnick} & \cite{xu2006immersed} & \cite{kumar2021comprehensive} & Present\\ \hline \hline
		
	 $L_w$    & 10  &--- & 0.530                                    & ---                              & ---                                   & ---                                                       & 0.531                          & 0.538                    \\
		& 20   & 1.86 & 1.880                                    & 1.820                            & 1.860                                 & 1.840                                                     & 1.874                        & 1.920                    \\
		& 40  &4.38 & 4.690                                    & 4.480                            & 4.56                                 & 4.420                                                     & 4.278                         & 4.540                    \\
		\hline 
		$C_D$ & 10  & --- & 2.846                                    & ---                              & ---                                  & ---                                                       & 2.690                           & 2.58                    \\
		& 20  & --- & 2.045                                    & 2.001                            & 2.06                                 & 2.23                                                     &                         2.160 & 2.04                    \\
		& 40   &--- & 1.522                                    & 1.498                            & 1.54                                 & 1.66                                                     & 1.576                         & 1.64                    \\ \hline
		$\theta$ & 10   &---& 29.6                                    & ---                              & ---                                  & ---                                                       & 29.69                           & 30.8                    \\
		& 20 &44.4  &43.7                                    & 42.9                           & 43.50                                 & 44.20                                                     & 42.66                         & 45.6                   \\
		& 40 &53.4  & 53.8                                    & 51.5                            & 53.60                                 & 53.50                                                     & 53.08    
		                    & 56.3                    \\ \hline
	\end{tabular}
\end{table}

We have also computed the wake length $L_w$, which is the distance between rear end point $P_1$ of cylinder and the end of the separation at the point $P_2$, and the angle $\theta$ between the $x$-axis and the line joining the center of the cylinder and the separation point $S$ on the cylinder (refer to figure \ref{cyl_setup}(b)), known as the separation angle. We have further computed the drag coefficient $C_D$ by utilizing the momentum balance along the streamwise direction.  All these flow parameters are tabulated in table \ref{tab8} along with some established numerical as well as the path-breaking experimental results of Coutanceau and Bouard \cite{cout1}. In all the cases, excellent comparison is observed between our computed results and the benchmark solutions.
\begin{figure}[!h]
\begin{minipage}[b]{.5\linewidth}  
\includegraphics[scale=0.425]{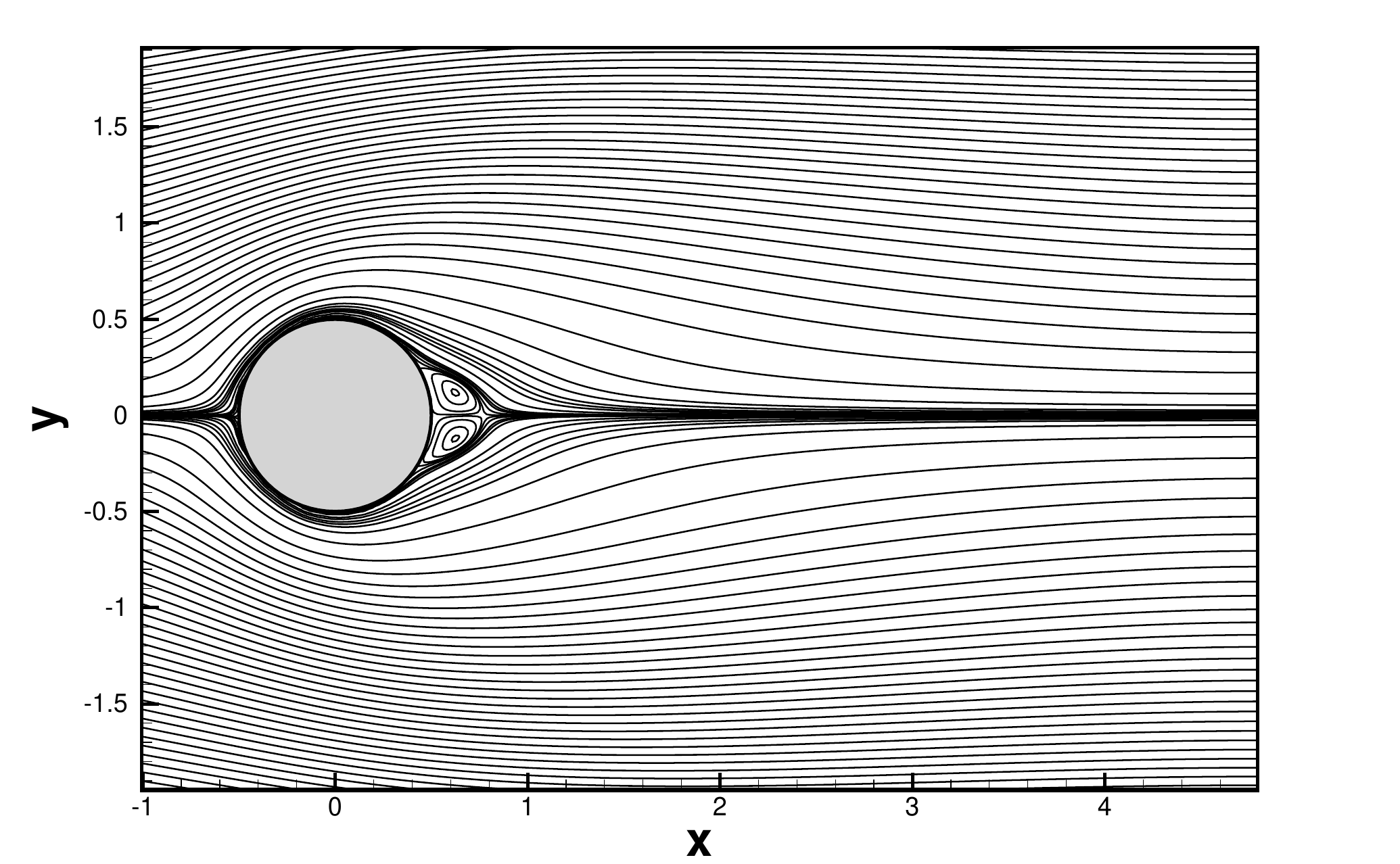} 
 \centering (a)
\end{minipage}            \hspace{-2.mm}
\begin{minipage}[b]{.5\linewidth}
\includegraphics[scale=0.425]{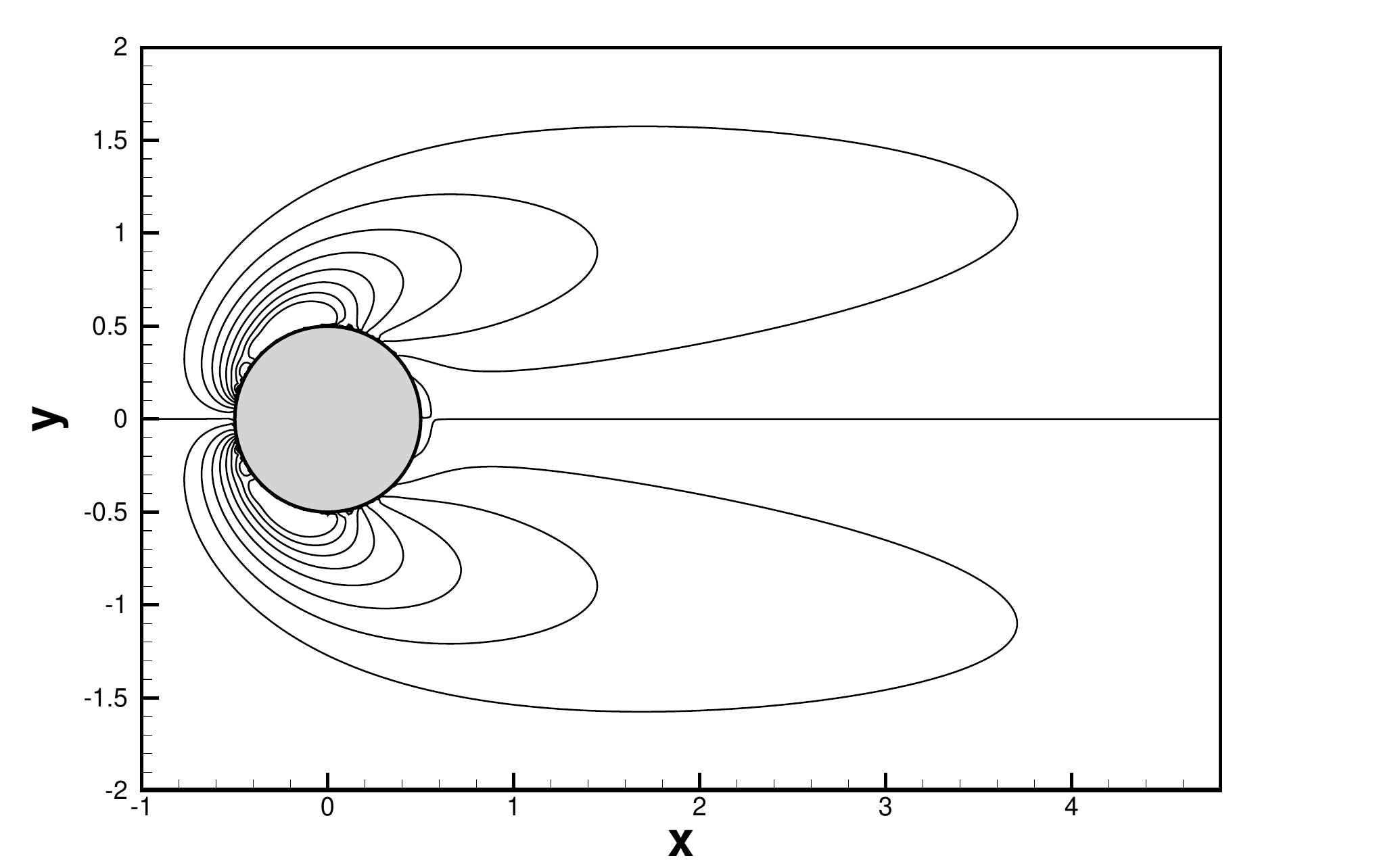} 
 \centering (a)
\end{minipage} 
\begin{minipage}[b]{.5\linewidth}  
\includegraphics[scale=0.425]{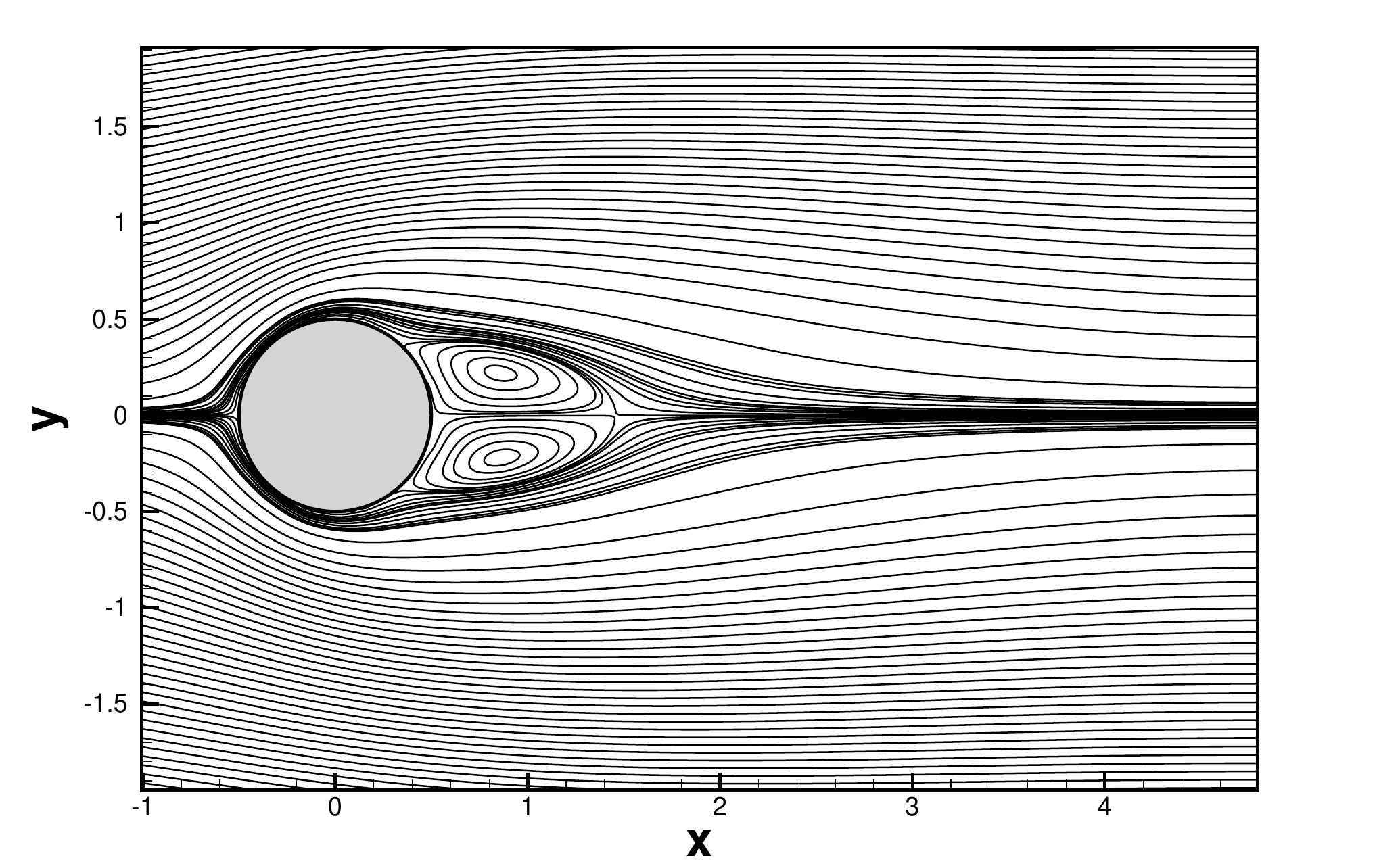} 
 \centering (b)
\end{minipage}            \hspace{-2.mm}
\begin{minipage}[b]{.5\linewidth}
\includegraphics[scale=0.425]{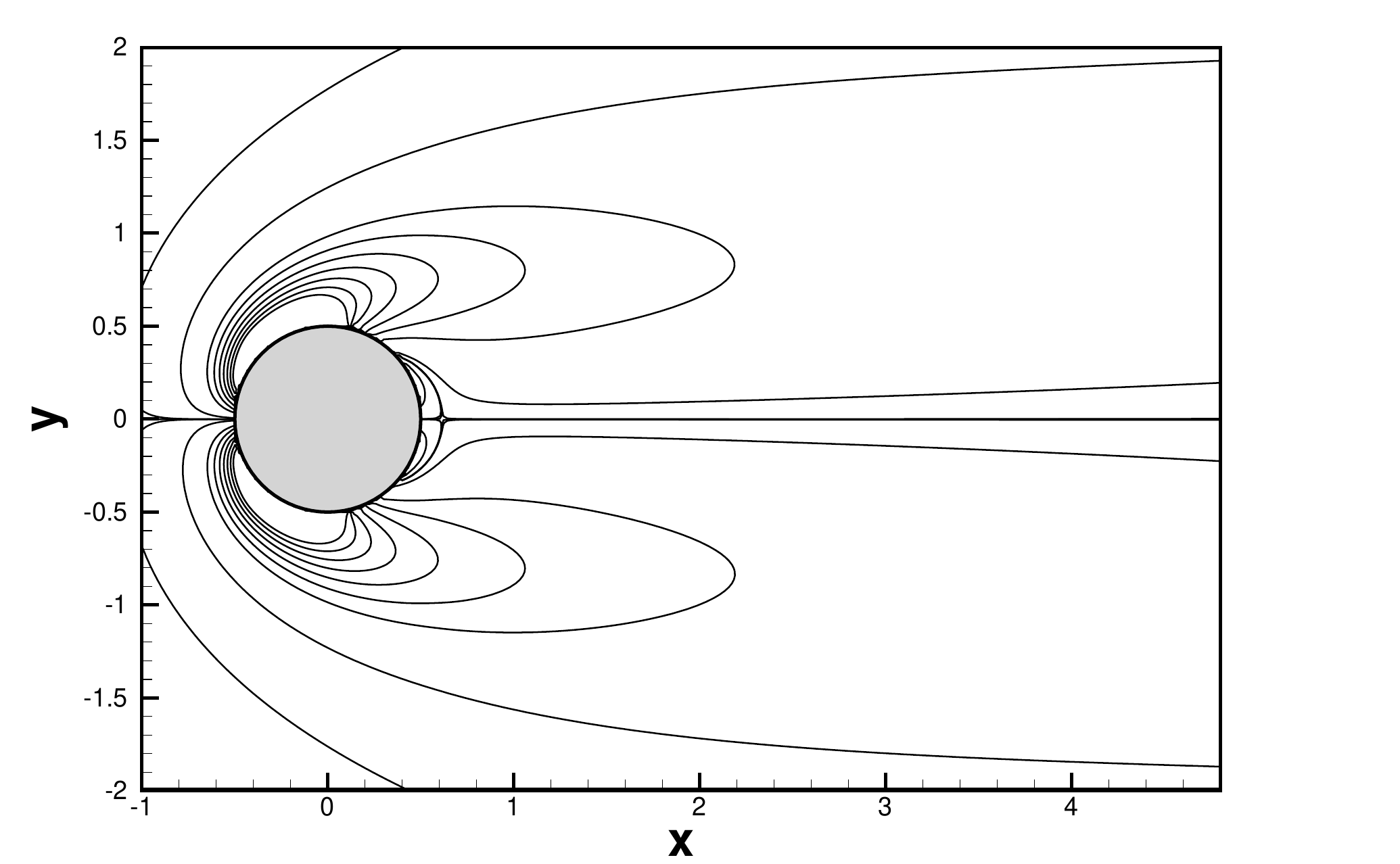} 
 \centering (b)
\end{minipage} 
\begin{minipage}[b]{.5\linewidth}   
\includegraphics[scale=0.425]{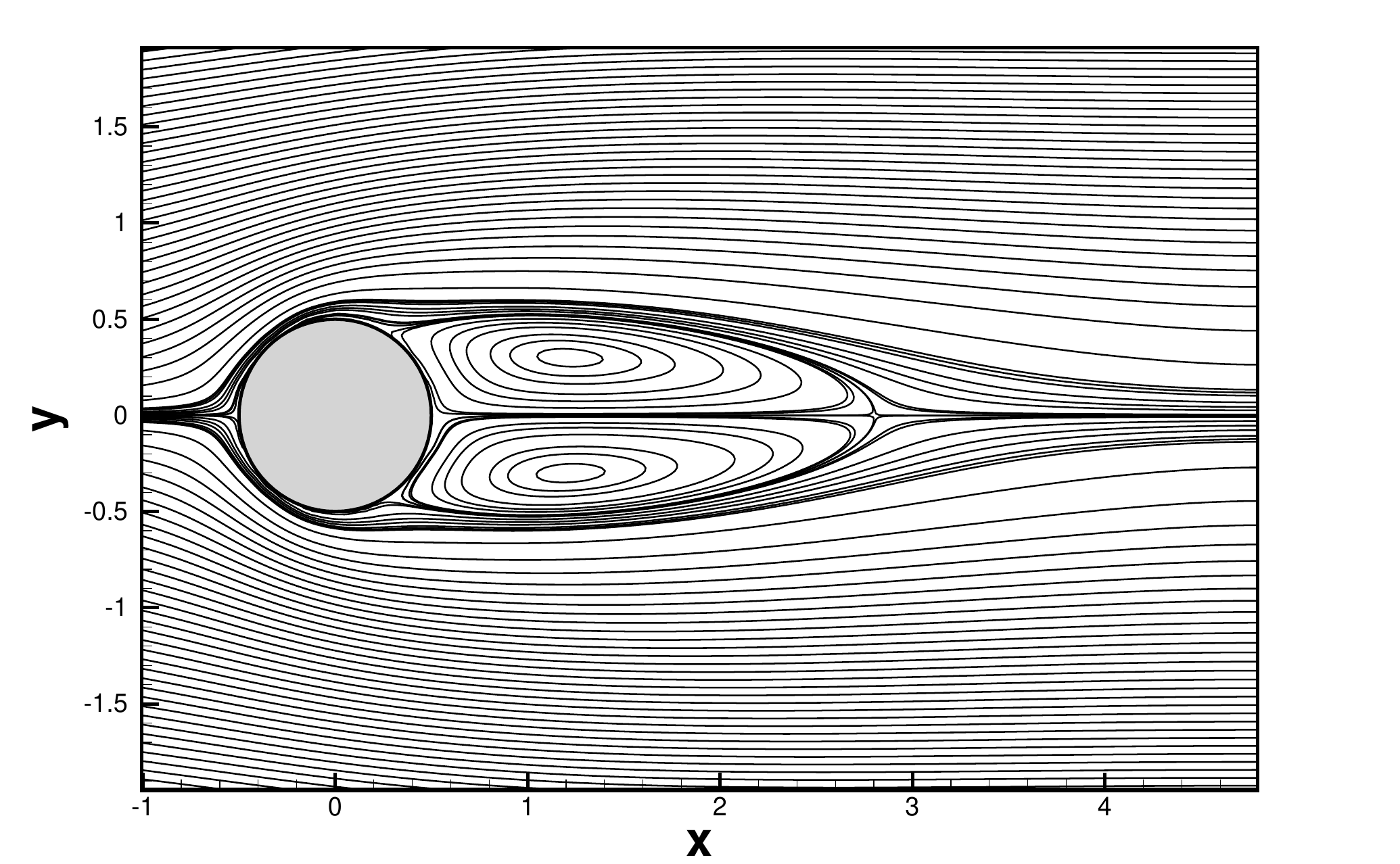} 
 \centering (c)
\end{minipage}          \hspace{-2.mm}
\begin{minipage}[b]{.5\linewidth}
\includegraphics[scale=0.425]{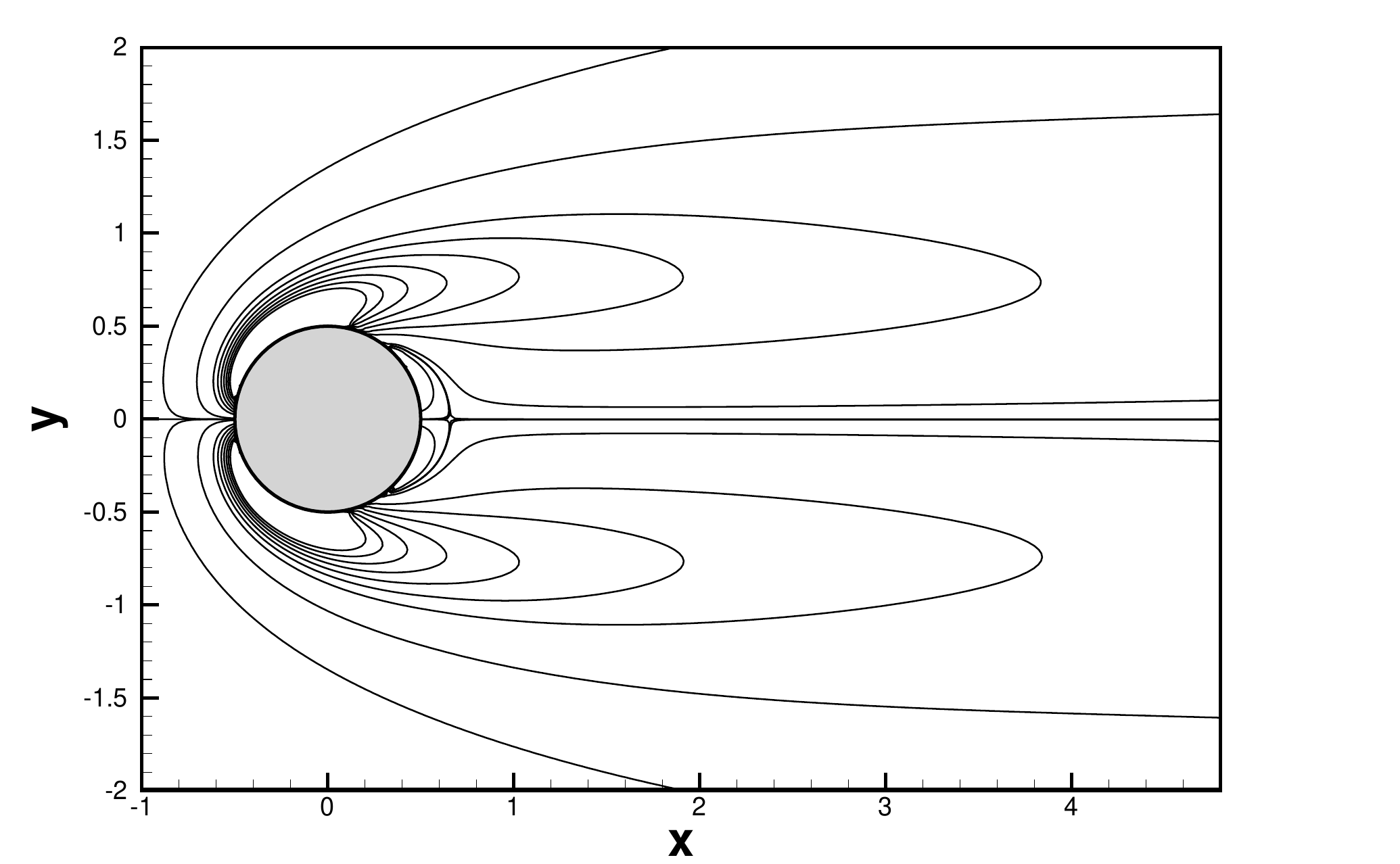} 
 \centering (c)
\end{minipage} 
\caption{{\sl  Simulation of flow past circular cylinder problem by present method: Steady-state  Streamlines (left) and Vorticity contours (right) for (a) $Re=10$, (b) $Re=20$ and (c) $Re=40$ on grid size $549 \times 499$.} }
\label{sf_vt}
\end{figure}

\subsection{Test case 6: Flow Past two randomly spaced inclined elliptic cylinders}
Here, we consider the laminar flow past two randomly spaced elliptic cylinders in uniform stream inclined at different angles to the stream for $Re=10.0$. The flow configuration is same as in figure \ref{cyl_conf} of the previous example, except that instead of one single cylinder, two elliptic cylinders with centres $(0,0.5)$ and  $(4,-0.5)$ with respective eccentricities $\displaystyle \frac{\sqrt{15}}{4}$ and  $\displaystyle \frac{2\sqrt{2}}{3}$ are placed in the uniform stream. The ellipses make angles $\displaystyle \frac{3\pi}{4}$ and $\displaystyle \frac{\pi}{4}$  respectively with the horizontal line. The rectangular region $[-2.5,2.5] \times [-3,7]$ in $xy$-plane is chosen as the computational domain. In figures \ref{elp}(a)-(b), we show the steady-state streamlines and vorticity contours respectively, where one can see the streamlines separating from the upper sides of the cylinder fronts to form a vortex lying across the imaginary line joining the centres of the ellipses. In table \ref{grid_independence}, we present the vortex data on three different grids depicting the coordinates of the center of the vortex, the $\psi$ and $\omega$ values thereat and the maximum vorticity value in the domain. For the same grids, we also present the distribution of $u$ and $v$ along the vertical and horizontal lines passing through the vortex center in figures \ref{vel}(a)-(b) respectively. Both table \ref{grid_independence} and the overlapping of the graphs in figure \ref{vel} clearly establish the grid independence of our computed results.

\begin{figure}[!h]
\begin{minipage}[b]{.5\linewidth}  
\includegraphics[scale=0.425]{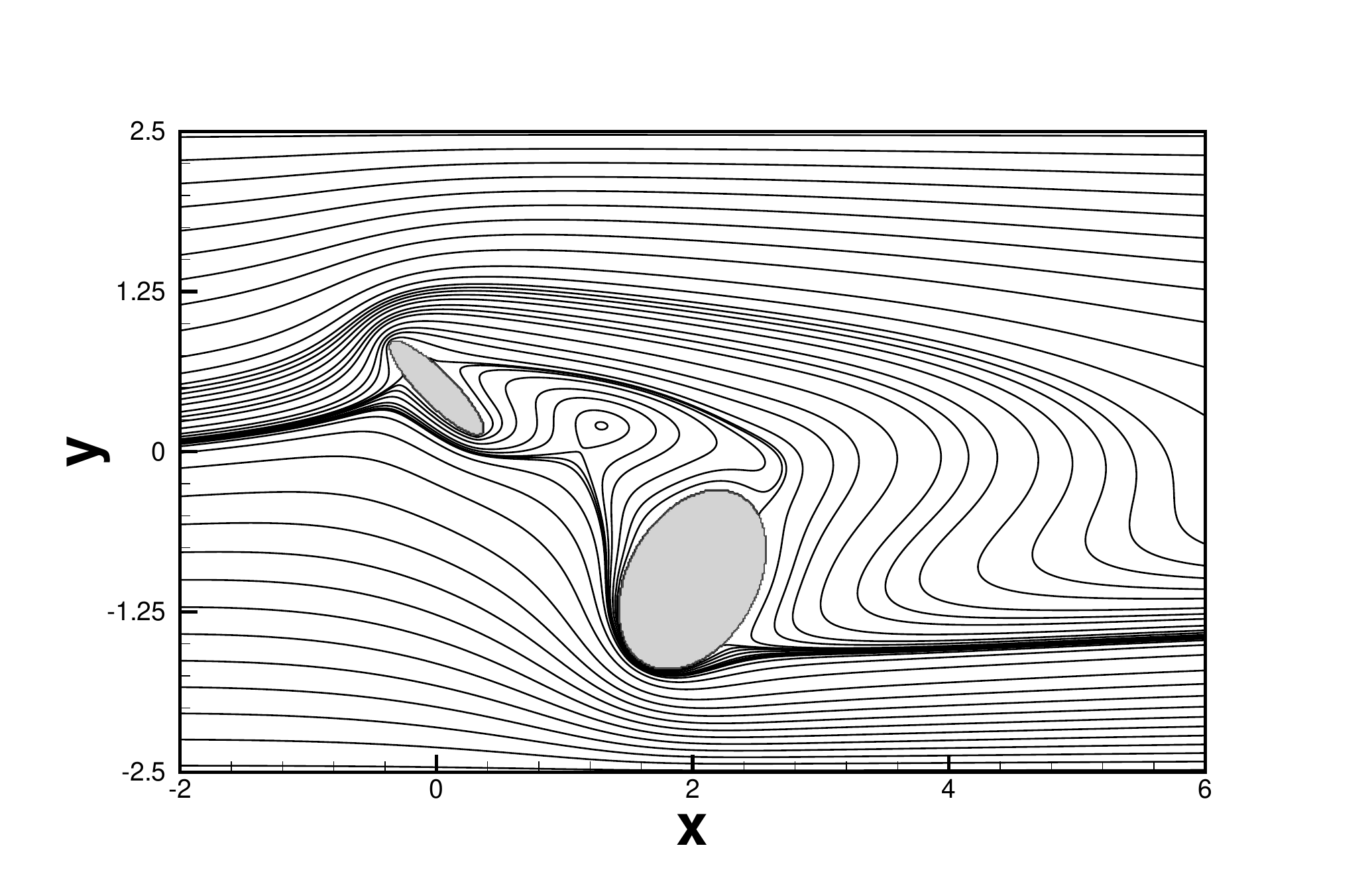} 
 \centering (a)
\end{minipage}            \hspace{-2.0mm}
\begin{minipage}[b]{.5\linewidth}
\includegraphics[scale=0.425]{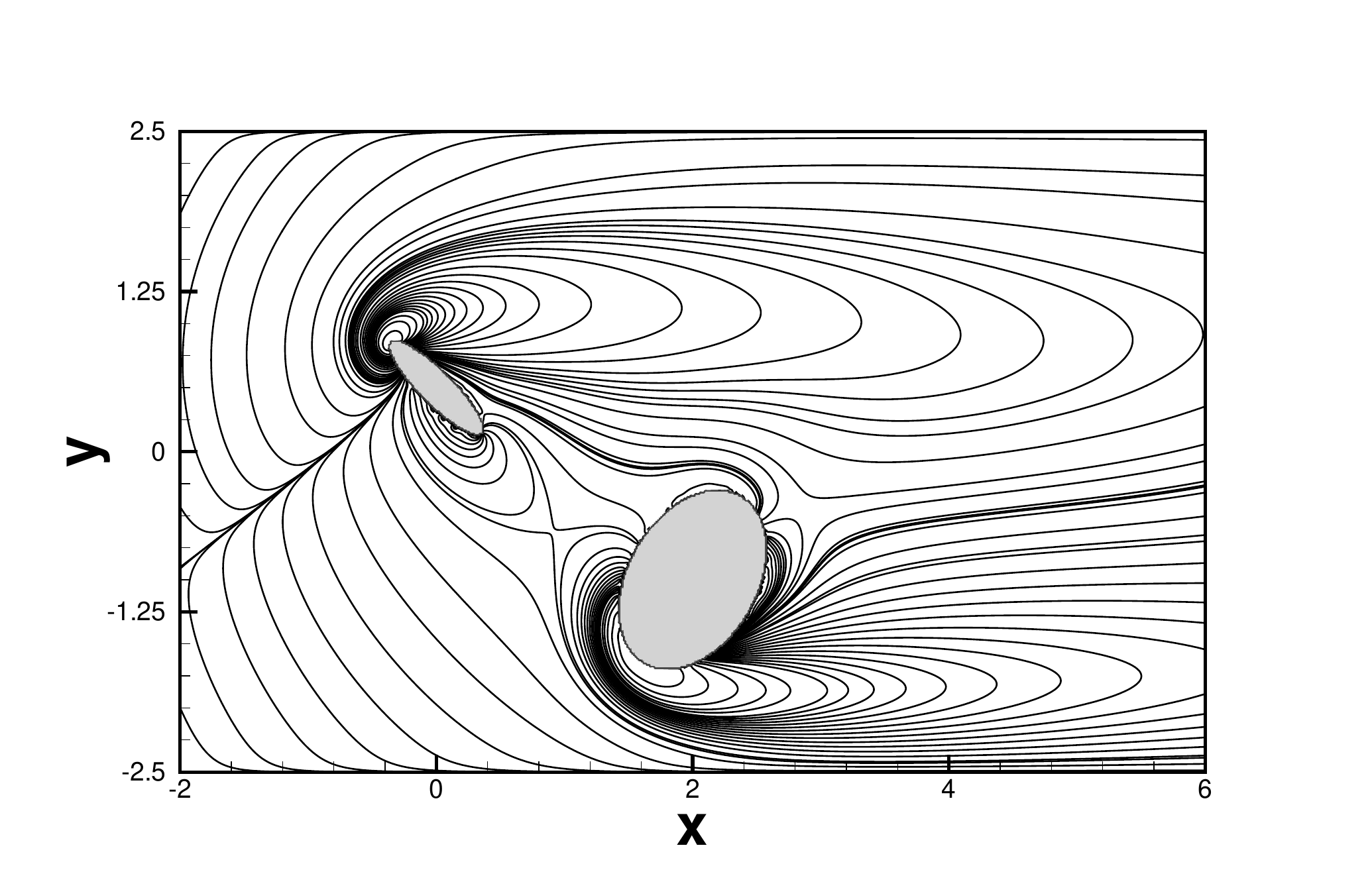} 
 \centering (b)
\end{minipage} 
\caption{{\sl Simulation of flow past two randomly spaced inclined elliptic cylinders of varied eccentricities in uniform flow for $Re=10.0$: (a) Streamlines and  (b) Vorticity contours on a grid of size $799 \times 399$.} }
\label{elp}
\end{figure}
\begin{table}[!h]
\caption{Grid independence study for Test Case 6.}
\begin{center}
\begin{tabular}{|c|c|c|c|c|}  \hline
$h$ &   Vortex Centre $(x,y)$  &  $\psi$ & $\omega$  & $\omega_{max}$ \\ \hline
$ \frac{1}{40}$   & $(1.29293, 0.202020)$  & $-0.028026$ & $-0.154637$ & $21.61155$     \\ \hline
$ \frac{1}{60}$   & $(1.27852, 0.201342)$  & $-0.028087$ & $-0.152500$ & $21.18219$     \\ \hline
 $ \frac{1}{80}$  & $(1.28392, 0.201005)$  & $-0.027970$ & $-0.152518$ & $21.23926$     \\ \hline
 \end{tabular}
\end{center}
\label{grid_independence}
\end{table}

\begin{figure}[!h]
\begin{minipage}[b]{.5\linewidth}  
\includegraphics[scale=0.425]{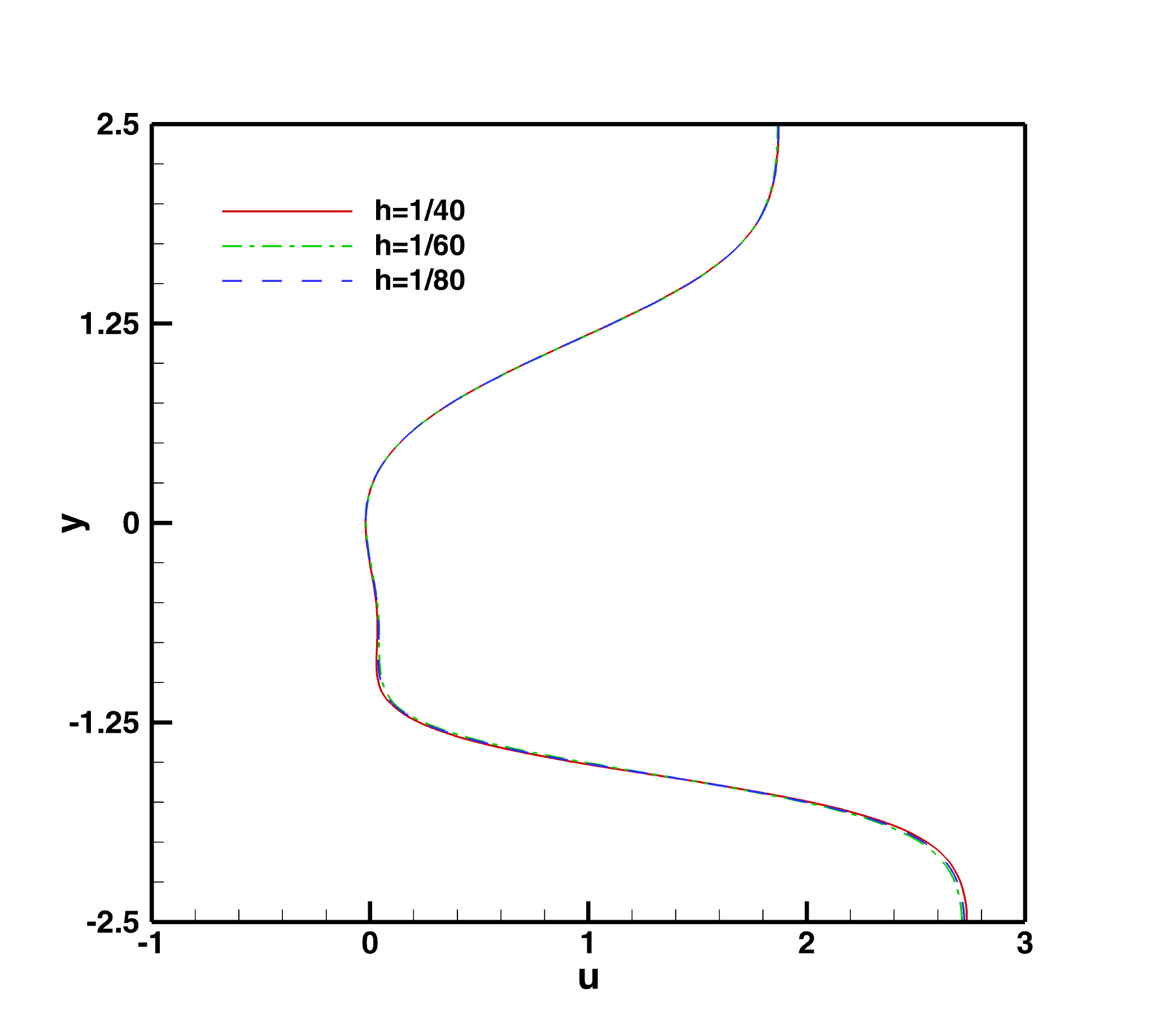} 
 \centering (a)
\end{minipage}            \hspace{-2.mm}
\begin{minipage}[b]{.5\linewidth}
\includegraphics[scale=0.425]{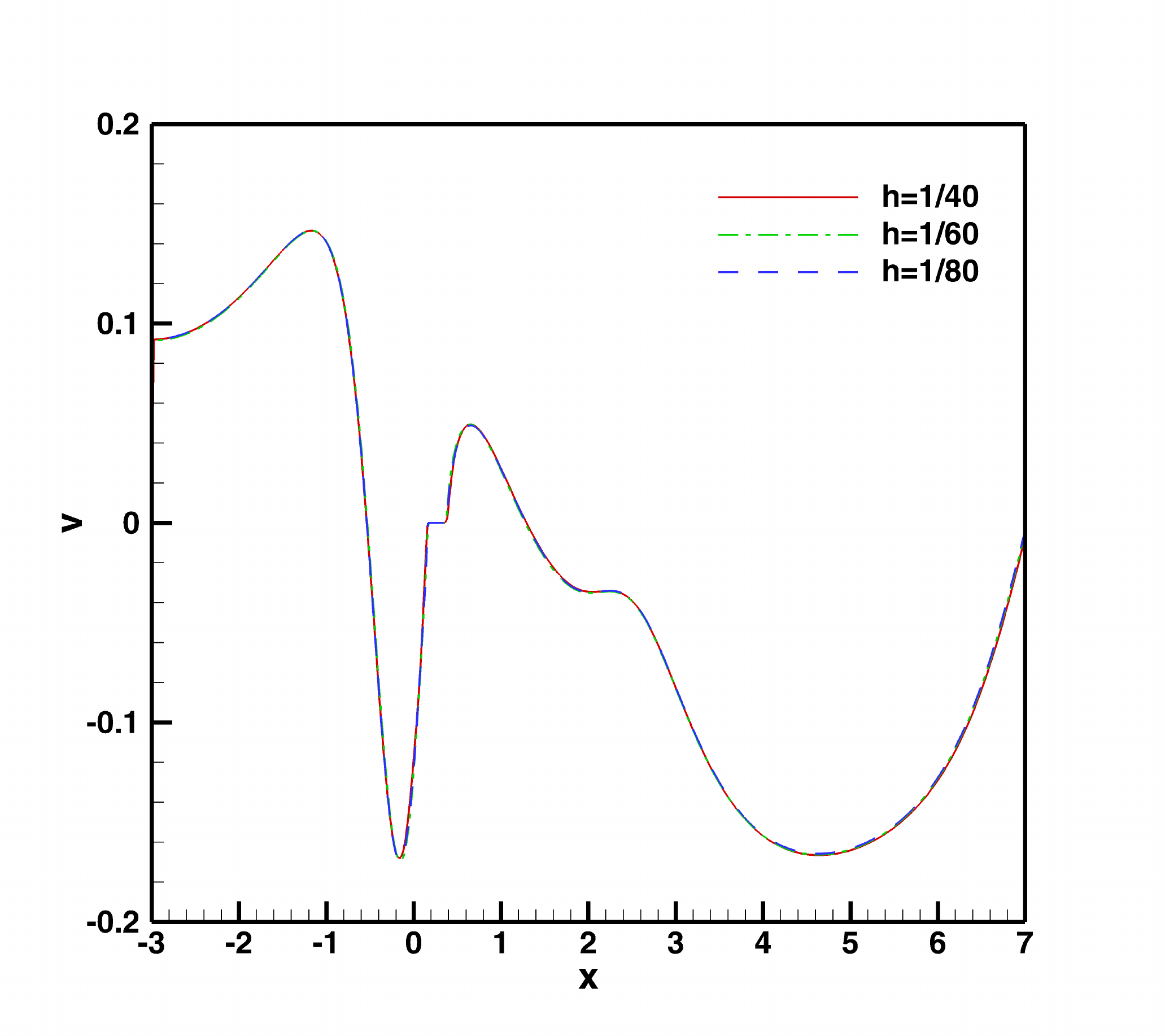} 
 \centering (b)
\end{minipage} 
\caption{{\sl Grid independence study for the flow past two randomly spaced inclined elliptic cylinders: (a) $u$-velocity along the vertical line and  (b) $v$-velocity along the horizontal line through the vortex center.} }
\label{vel}
\end{figure}

\subsection{Test case 7: Simulation of the experimental visualization by Taneda for Stokes flow \cite{taneda1979stokes}}
In this section, we endeavour to replicate the experimental visualization by Taneda \cite{taneda1979stokes}, which he had performed in laboratory during the late seventies pertaining to Stokes flows. We consider the following two cases:
\subsubsection{Case 1}
The problem considered here is the flow past two circular cylinders in tandem subjected to uniform flow for $Re=0.01$. Once again the problem configuration is similar to the one described in test case 1, the only change being, instead of one, two circular cylinders of unit diameters are placed in tandem at a distance $\epsilon d$ apart, $d$ being the diameter of the cylinder. We have chosen the cases when $\epsilon=0.5$ and $1.0$. In figure \ref{tandem}, we present our computed streamlines for these two cases and compare them with the experimental visualization of Taneda \cite{taneda1979stokes}. As reported in Taneda's experiment, for $\epsilon=1.0$, two vortex pairs appear while for   $\epsilon=0.5$, the two vortex pairs merge into one vortex pair.
Without resorting to the complicated mathematics that had been adopted in \cite{davis1976stokes} for transforming the physical plane and hence the governing equations, our simulation has very elegantly and accurately captured the flow physics, as is evidenced by the extreme closeness of our numerical results to the experimental ones.

\begin{figure}[!h]
\begin{minipage}[b]{.5\linewidth}  
\includegraphics[scale=0.425]{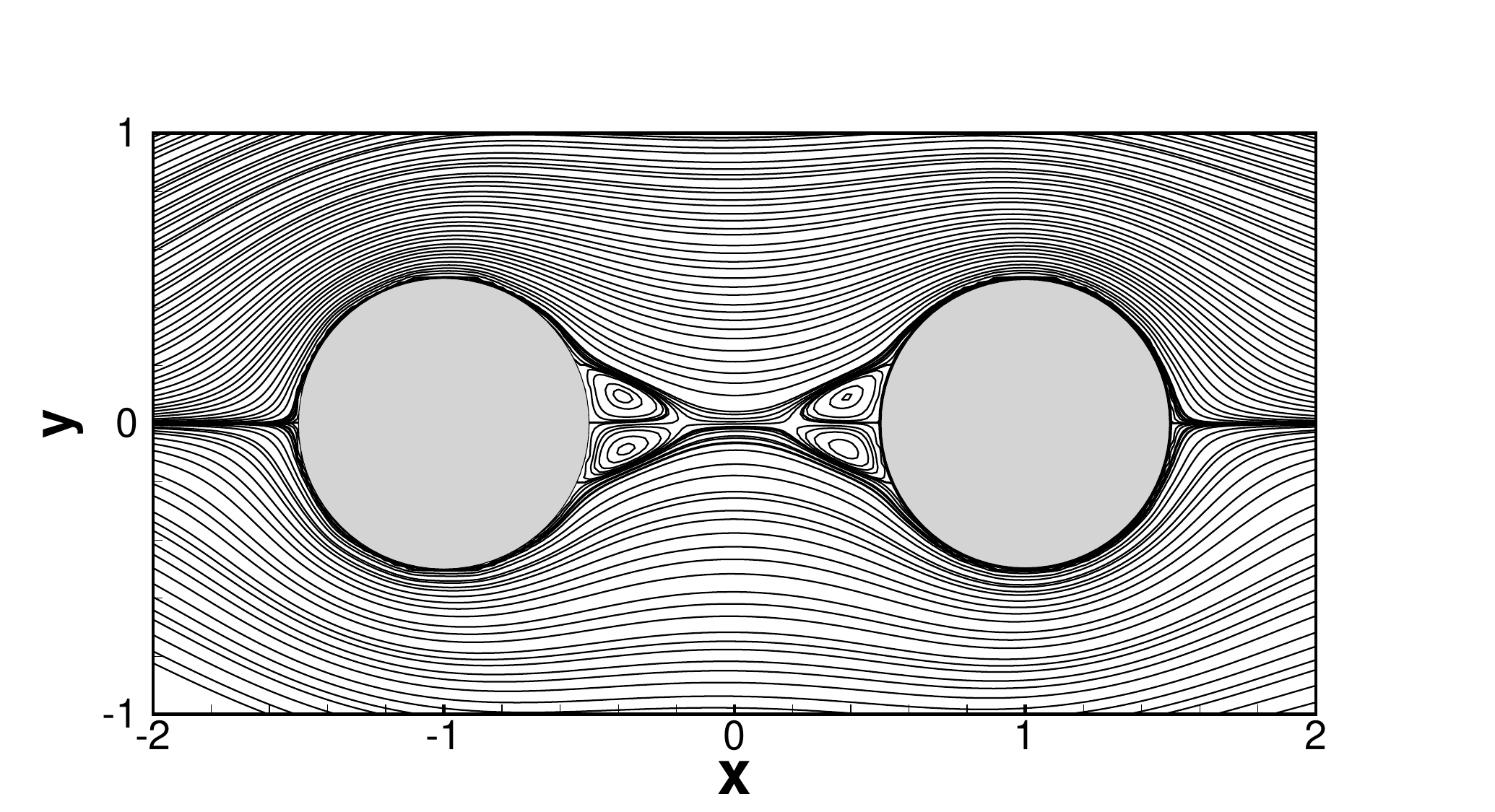} 
 \centering (a)
\end{minipage}            \hspace{-2.mm}
\begin{minipage}[b]{.5\linewidth}
\includegraphics[scale=0.425]{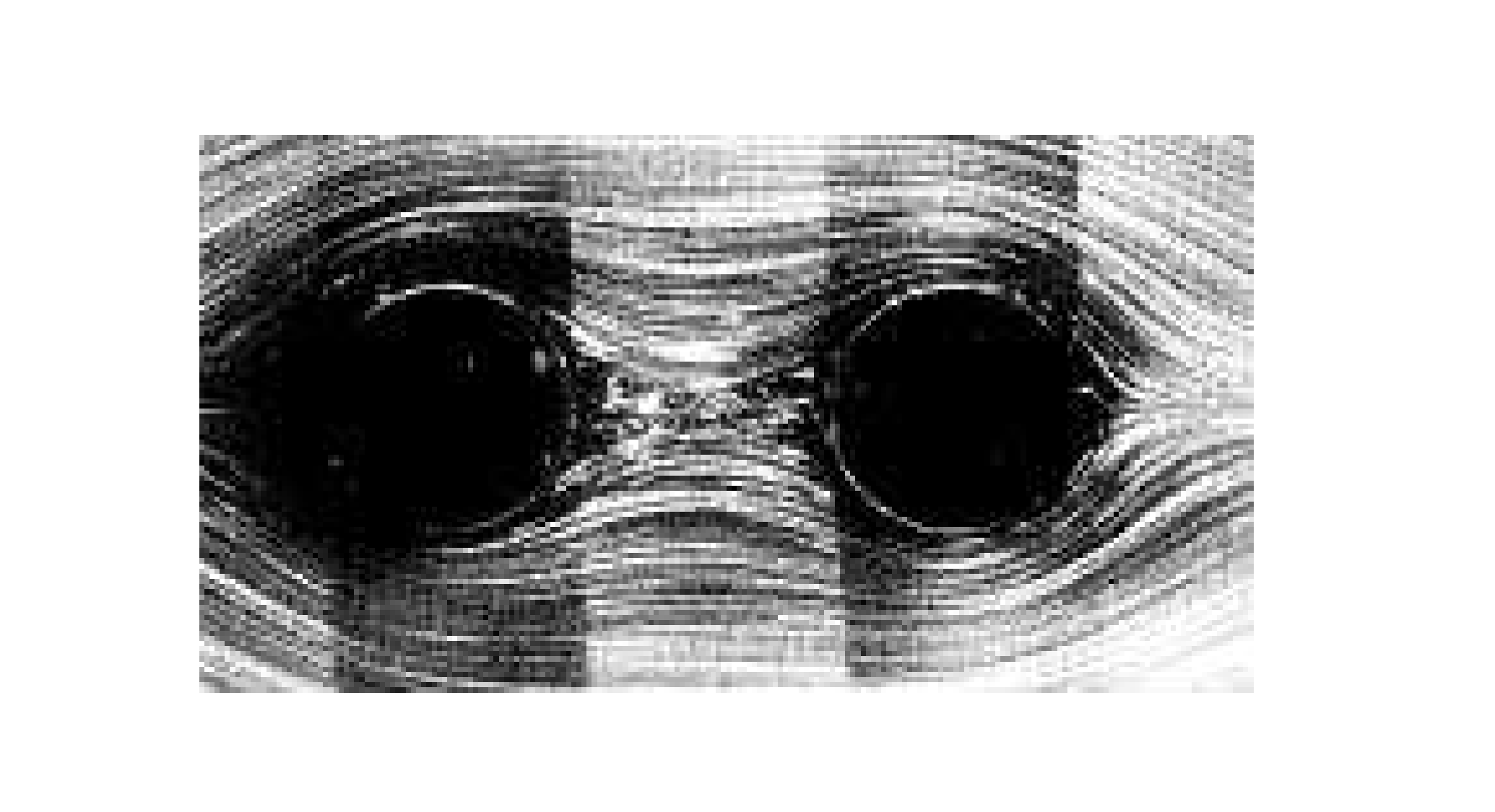} 
 \centering (a)
\end{minipage} 
\begin{minipage}[b]{.5\linewidth}  
\includegraphics[scale=0.425]{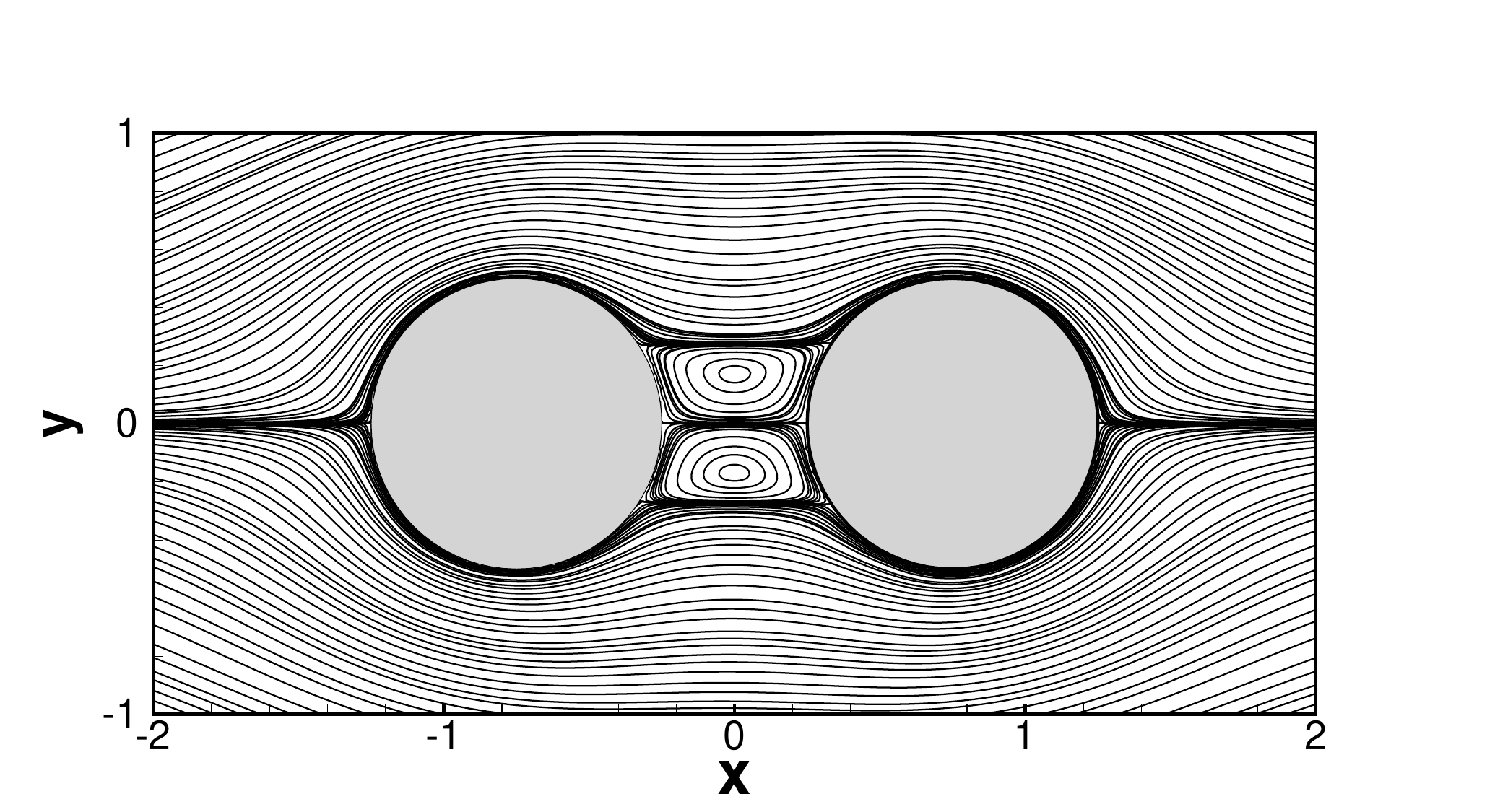} 
 \centering (b)
\end{minipage}            \hspace{-2.mm}
\begin{minipage}[b]{.5\linewidth}
\includegraphics[scale=0.425]{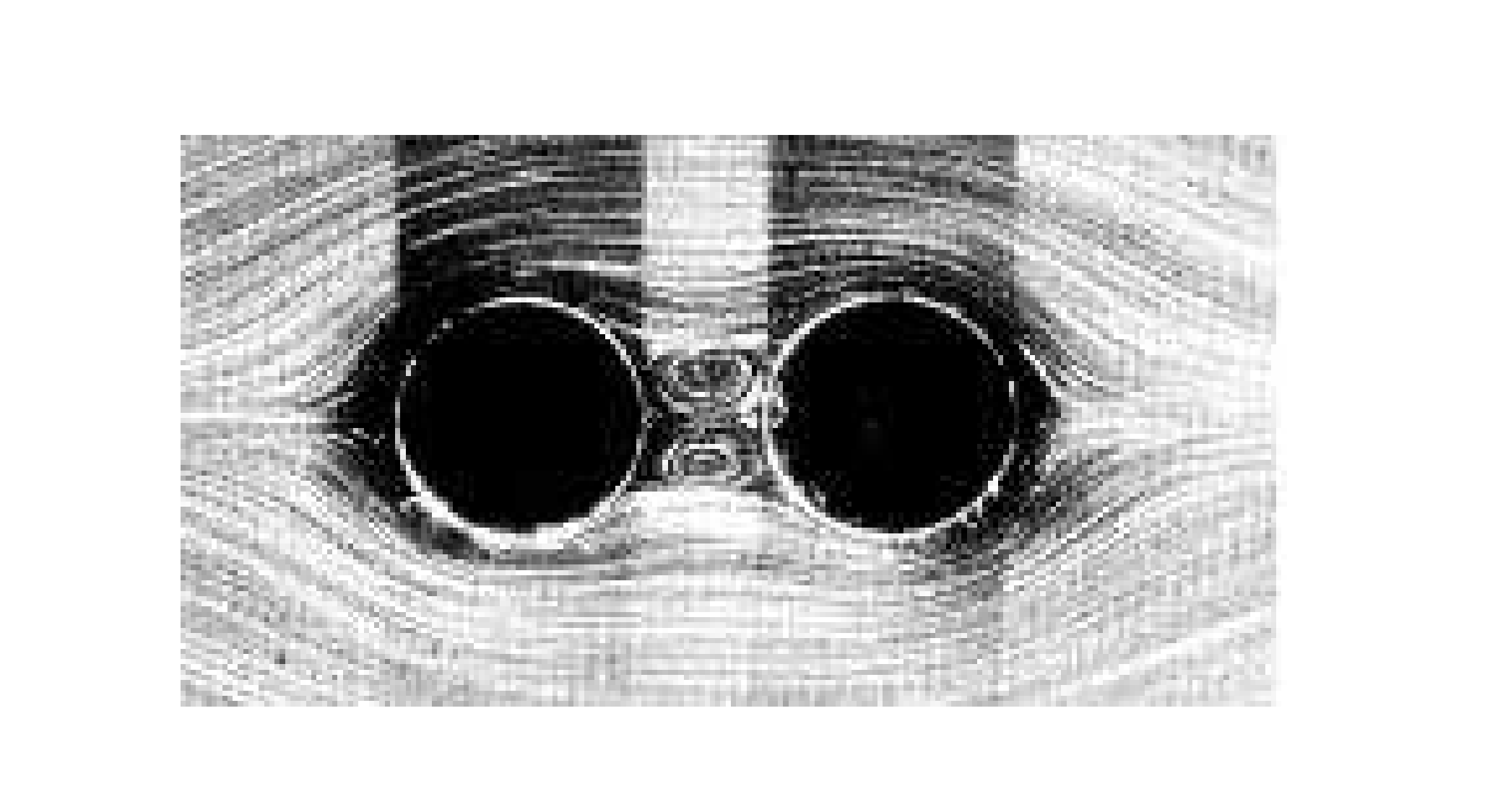} 
 \centering (b)
\end{minipage} 
\caption{{\sl Simulation of flow past two circular cylinders in tandem in uniform flow for $Re=0.01$: Computed (left) on a grid of size $659 \times 479$ and experimental Streamlines from Taneda \cite{taneda1979stokes}  $\copyright$ [1979] The Physical Society of Japan (J. Phys. Soc. Jpn. [46], [Visualization of Separating Stokes Flows/1935-1942].) (right) when the distance between the cylinders is (a) $1.0d$ and  (b) $0.5d$, $d$ being the diameter of the cylinder.} }
\label{tandem}
\end{figure}

\subsubsection{Case 2}
\begin{figure}[!h] 
\begin{center}
 \includegraphics[height=3in,width=6in, keepaspectratio]{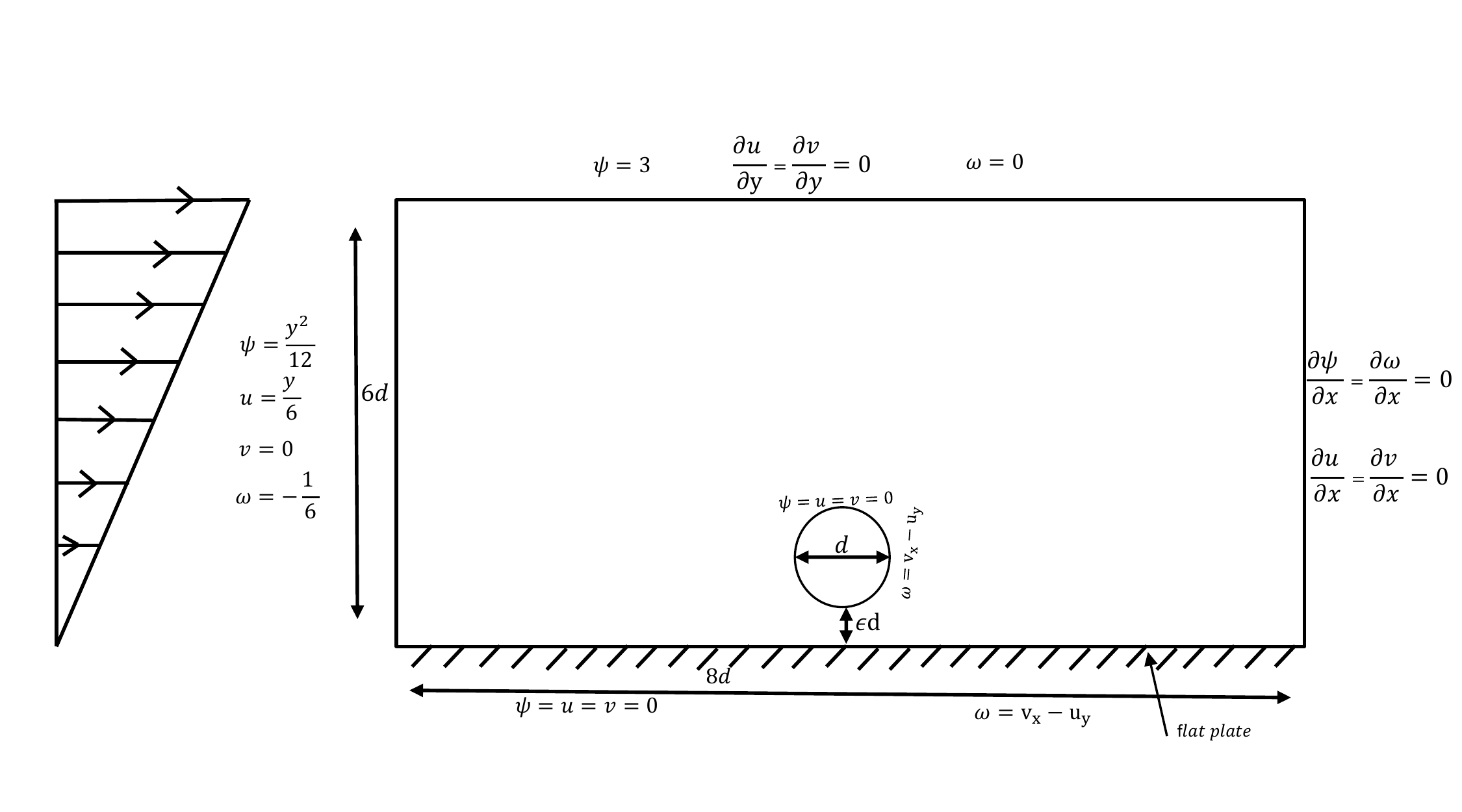} 
\caption{\sl{ Schematic for the shear flow past a circular cylinder and a plane with the boundary conditions.} }
\label{flat_conf}
\end{center}
\end{figure}
Here we consider the separation in a linear shear flow past a circular cylinder and a plane \cite{davis1977shear} for $Re=0.011$. The problem configuration is shown in figure \ref{flat_conf} along with the boundary conditions. In the original lab experiment of Taneda \cite{taneda1979stokes}, coloured glycerine was introduced into the working fluid along a straight line normal to a flat plate and the cylinder was placed above the plate maintaining a gap $\epsilon d$, $d$ being the diameter of the cylinder. When the test body was set in motion by slowly moving the plate in the negative $x$-direction, a linear shear layer flow was seen to establish on the flat plate up to a height of approximately six times the cylinder diameter. As such, we have chosen the computational domain and the boundary conditions as shown in  figure \ref{flat_conf} so that the same shear layer effect could be accomplished. At the solid boundary at the bottom, an $O(h^3)$ approximation of vorticity $\omega_{i,0}$ can be obtained by the strategy adopted in \cite{kalita2001fully},
\begin{equation}
\delta^+_y\psi_{i,0}+\frac{h}{2}\omega_{i,0}+\frac{h^2}{6}\delta^+_y\omega_{i,0}=0.
\end{equation}
Making use of the fact that $\psi=0$ on the bottom boundary, this reduces to 
\begin{equation}
2\omega_{i,0}+ \omega_{i,1}=-\frac{6}{h^2} \psi_{i,1}.
\end{equation}
Although the above is an implicit expression, it could be easily assimilated into the matrix equation for $\omega$. 

\begin{figure}[!h]
\begin{minipage}[b]{.5\linewidth}  
\includegraphics[scale=0.425]{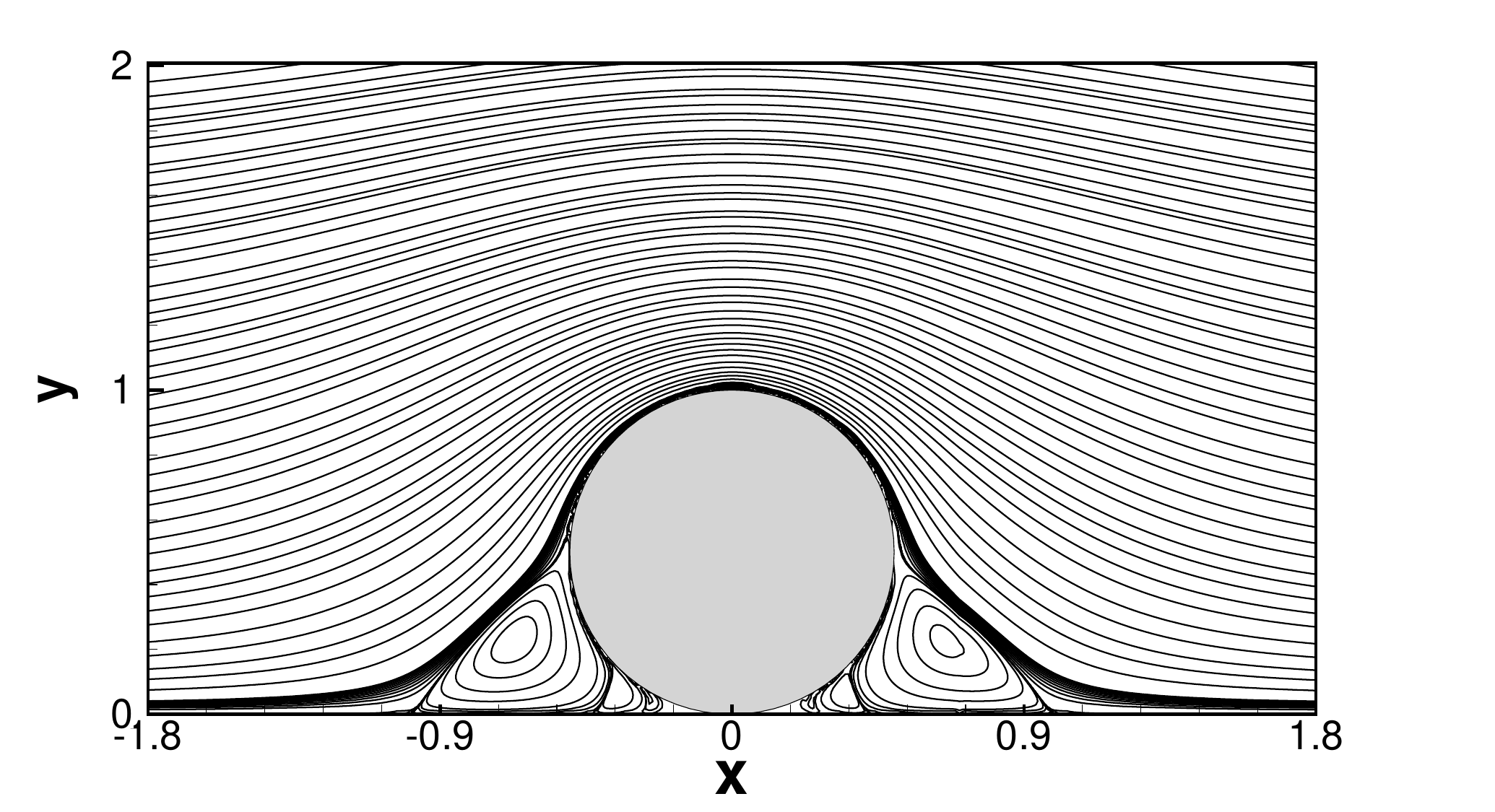} 
 \centering (a)
\end{minipage}            \hspace{-2.mm}
\begin{minipage}[b]{.5\linewidth}
\includegraphics[scale=0.425]{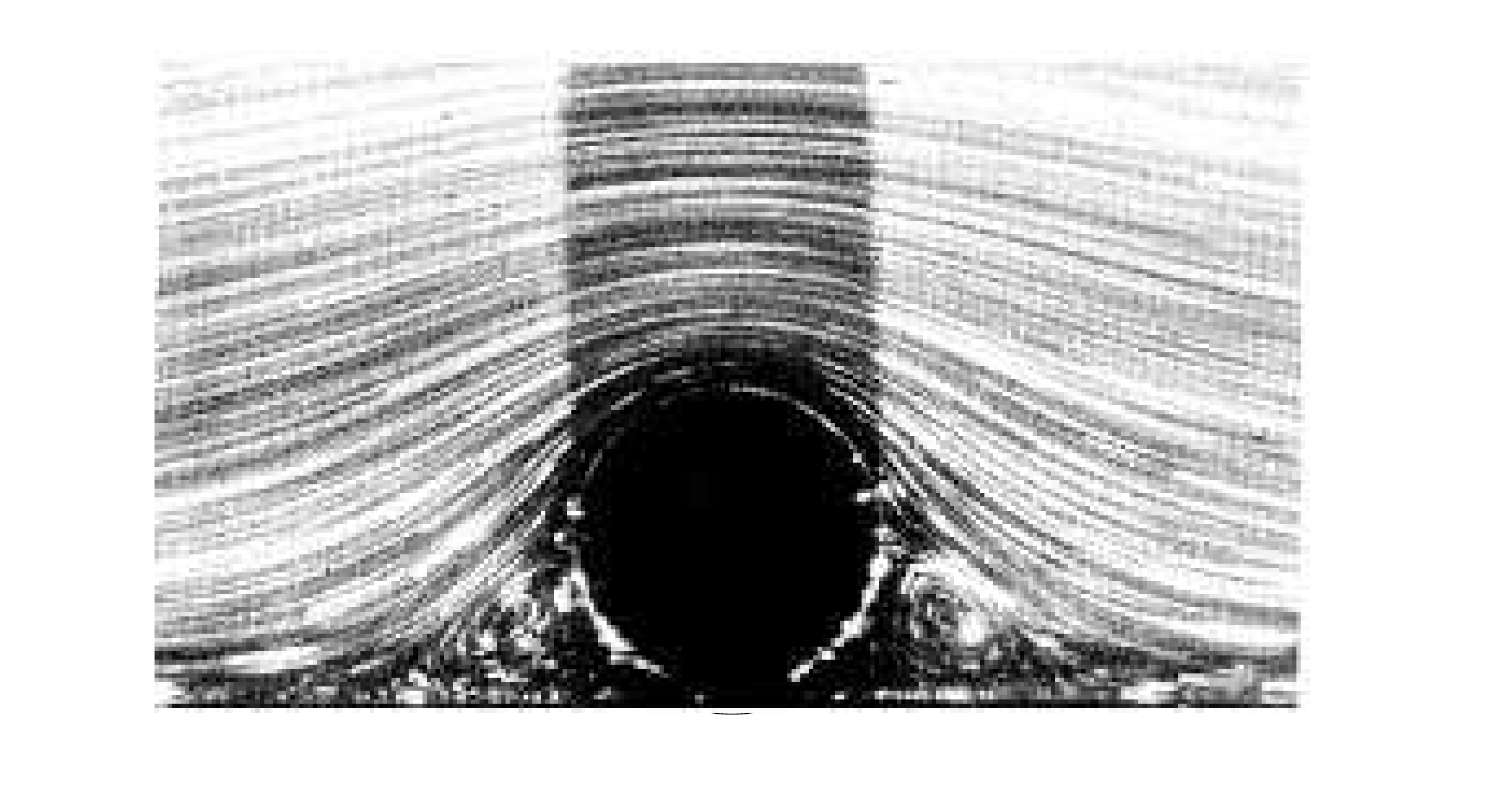} 
 \centering (a)
\end{minipage} 
\begin{minipage}[b]{.5\linewidth}  
\includegraphics[scale=0.425]{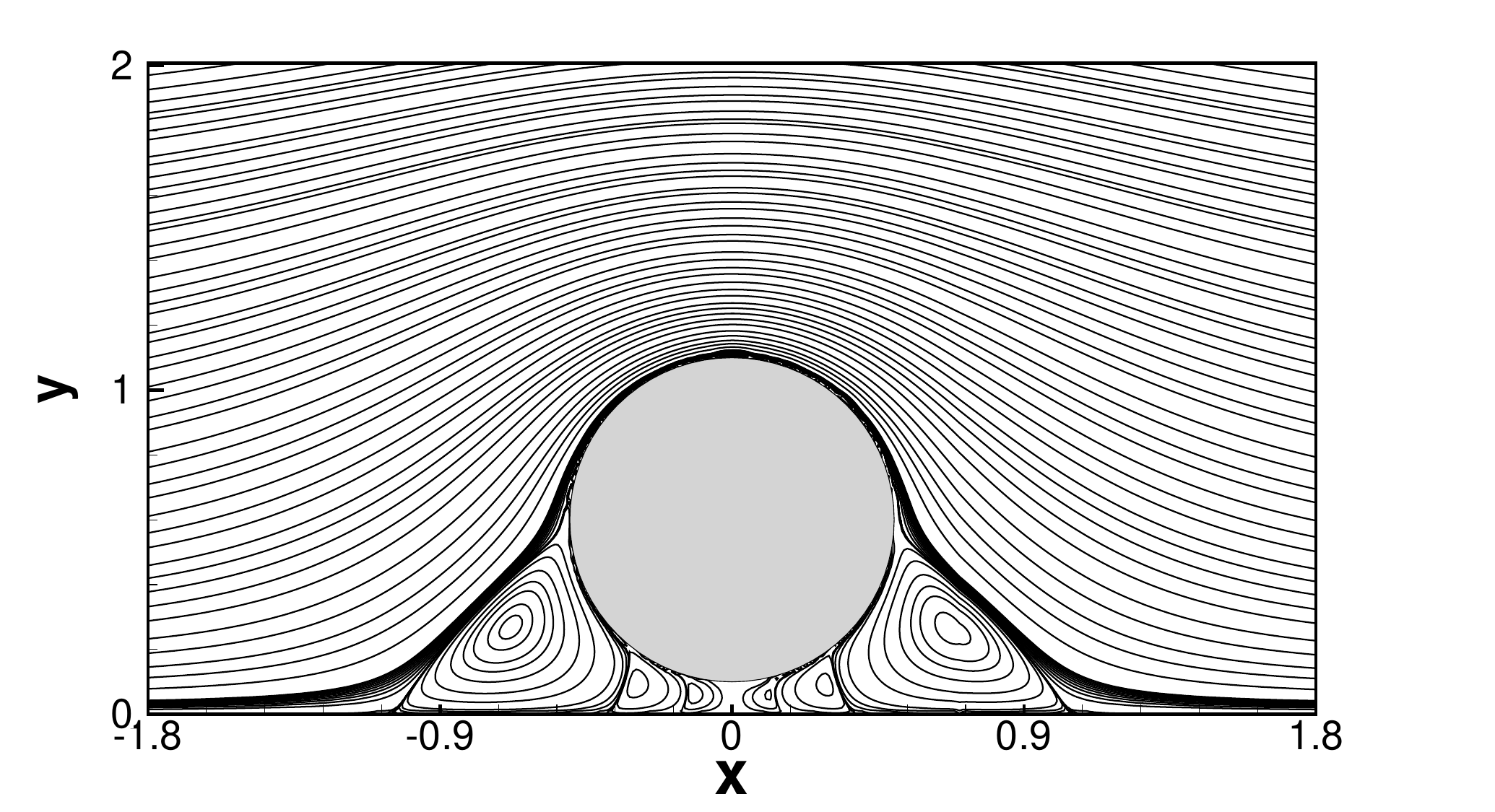} 
 \centering (b)
\end{minipage}            \hspace{-2.mm}
\begin{minipage}[b]{.5\linewidth}
\includegraphics[scale=0.425]{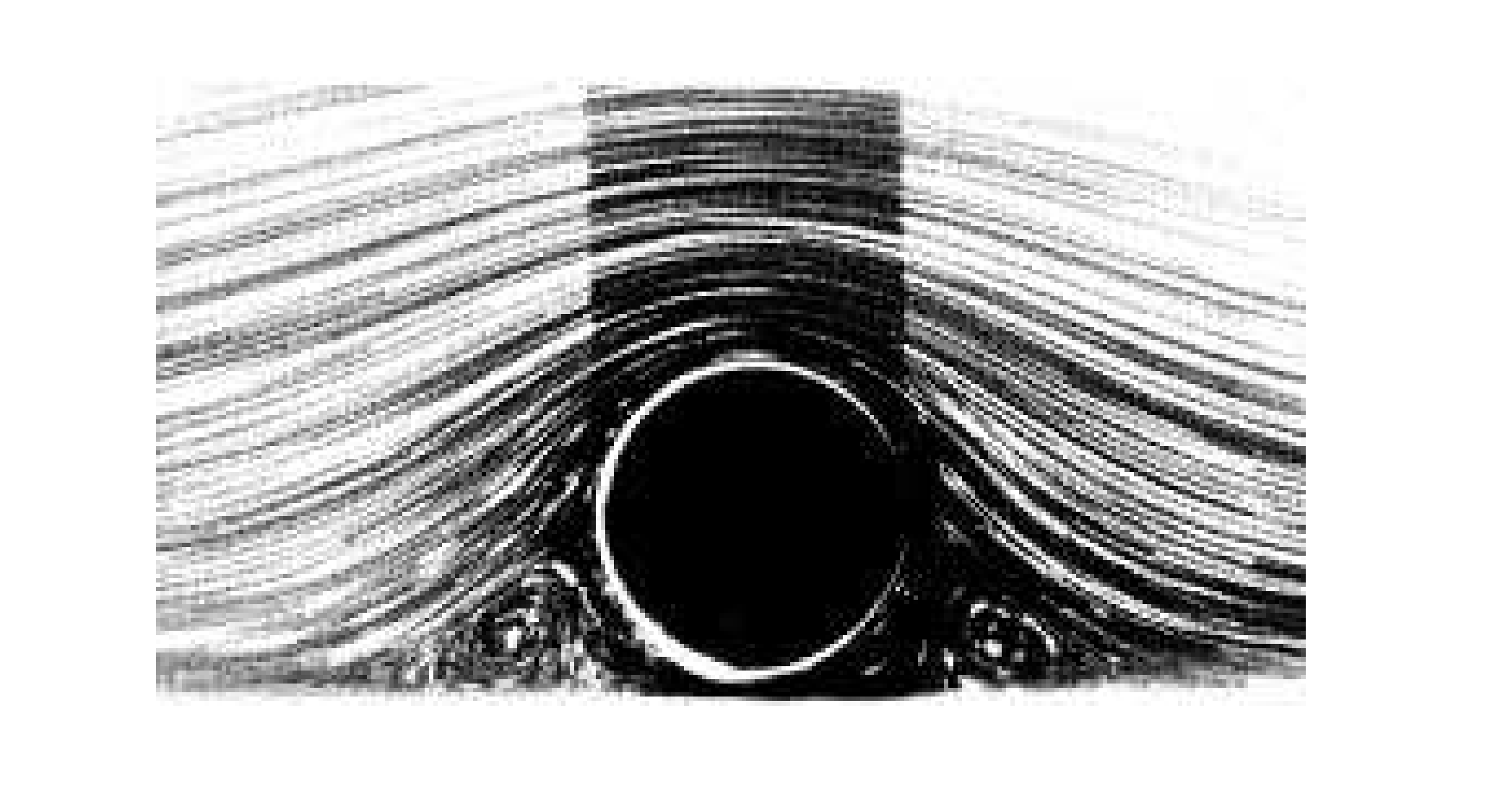} 
 \centering (b)
\end{minipage} 
\caption{{\sl Simulation of flow past circular cylinder placed near a plane in shear flow for $Re=0.011$: Computed (left) on a grid of size $799 \times 299$ and experimental Streamlines from Taneda \cite{taneda1979stokes} $\copyright$ [1979] The Physical Society of Japan (J. Phys. Soc. Jpn. [46], [Visualization of Separating Stokes Flows/1935-1942].) (right) when the cylinder is (a) $0.0d$ and  (b) $0.1d$ above the plane, $d$ being the diameter of the cylinder.} }
\label{shear}
\end{figure}

In figure \ref{shear}, we present our computed streamlines side by side with the experimental visualization of Taneda  \cite{taneda1979stokes} for $\epsilon=0.0$, which corresponds to the scenario when the cylinder touches the plate and $\epsilon=0.1$ corresponding to a gap of one tenth of the diameter between the cylinder and the plate. As can be seen from the figures, our simulation is extremely close to the experimental results. Note that our findings are in conformity with the theoretical structures of the Stokes flow past a circular cylinder near a plane given by Davis {\it et al.} \cite{davis1977shear}.

\section{Conclusion}
In the current work, we propose a new higher-order accurate explicit jump Immersed Interface approach (HEJIIM) on compact stencil for solving two-dimensional elliptic problems with singular source and discontinuous coefficients in the irregular region on a Cartesian mesh. A new strategy for discretizing the solution at irregular points is provided. The scheme is employed to solve four problems embedded with circular and star shaped interfaces in a rectangular region having analytical solutions and varied discontinuities across the interface in source and the coefficient terms. In the process, the order of convergence of the computed solutions are also established.  Solutions are compared with numerical results from existing IIMs and in all the cases, much improved accuracy of the solutions were observed.  The robustness of the proposed scheme is however better realized when applied to compute the flow in a host of fluid flow problems past bluff bodies governed by the Navier-Stokes equations. They not only include flows past a stationary circular cylinder in uniform and shear flow, but also multiple bodies of varied shape immersed in the flow. In the process, we have also established the grid independence of our computed results confirming the accuracy of the simulations performed. Our simulation of the flows were extremely close to well established numerical results and flow visualization from laboratory experiments. In all, we consider the proposed scheme to be an important addition to the already existing immersed interface methods. Currently, we are working on the development the transient counterpart of the proposed scheme for unsteady flows and the early indication is that it would be successful.
\section*{Acknowledgments}
The authors acknowledge Grant No. MTR/2017/000482 from Science and Engineering Research Board of the Department of Science and Technology, Government of India to carry out the current research.

\vspace{1cm}
\noindent
{\bf Data Availability Statement:} The data that support the findings of this study are available from the corresponding author upon reasonable request.

\bibliographystyle{plain}
\nocite{*}
\bibliography{aipsamp}

\end{document}